\documentclass[11pt, a4paper]{article}
\title{Large deviations of Schramm-Loewner evolutions: A survey}
\usepackage[T1]{fontenc}

\usepackage[latin9]{inputenc}

\usepackage{mathtools}
\usepackage{lmodern}
\usepackage{amsthm,amsmath,amsfonts,amssymb}
\usepackage{graphicx}
\usepackage{enumerate,enumitem}
\usepackage{authblk}
\usepackage{cite,url}
\usepackage[normalem]{ulem}

\usepackage{multirow}
\usepackage{array}
\newcolumntype{P}[1]{>{\centering\arraybackslash}p{#1}}
\usepackage{longtable,booktabs}

\usepackage{hyperref}
\hypersetup{
    colorlinks=true, 
    linkcolor=blue, 
    urlcolor=black, 
    citecolor=black,
    linktoc=all 
}

\usepackage[activate={true,nocompatibility},final,tracking=true,kerning=true,spacing=true,factor=1100,stretch=20,shrink=20]{microtype}

\microtypecontext{spacing=nonfrench}

\allowdisplaybreaks

\setlist[enumerate]{topsep = 1ex, leftmargin=1cm, itemsep= -2pt}

\let\OLDthebibliography\thebibliography
\renewcommand\thebibliography[1]{
  \OLDthebibliography{#1}
  \setlength{\parskip}{1pt}
  \setlength{\itemsep}{2pt}
}

\allowdisplaybreaks

\newtheorem{thm}{Theorem}[section]
\newtheorem{cor}[thm]{Corollary}

\newtheorem{lem}[thm]{Lemma}
\newtheorem{prop}[thm]{Proposition}

\theoremstyle{definition} 
\newtheorem{df}[thm]{Definition}
\newtheorem{ex}[thm]{Example}
\newtheorem{remark}[thm]{Remark}

\numberwithin{equation}{section}

\usepackage[font=normal]{caption}
\usepackage{floatrow}

\usepackage{mathtools}

\usepackage[dvipsnames]{xcolor}

\global\long\def\domain{D}

\global\long\def\open{O}
\global\long\def\closed{F}

\global\long\def\bpt{a} 
\global\long\def\ept{b} 

\global\long\def\np{n} 

\global\long\def\link#1#2{\{#1,#2\}}

\global\long\def\Hmin{\mc M}

\global\long\def\PartF{\mc Z}

\global\long\def\ii{\mathfrak{i}}

\global\long\def\Catalan{C}

\renewcommand{\liminf}{\varliminf}
\renewcommand{\limsup}{\varlimsup}

\newcommand{\abs}[1]{\left\lvert #1 \right \rvert}

\newcommand{\brac}[1]{\left \langle #1 \right \rangle}
\newcommand{\norm}[1]{\lVert #1 \rVert}

\newcommand{\mc}[1]{\mathcal{#1}}
\newcommand{\m}[1]{\mathbb{#1}}

\newcommand{\rar}[0]{\rightarrow}

\renewcommand\Re{\operatorname{Re}}
\renewcommand\Im{\operatorname{Im}}

\def\PSL{\operatorname{PSL}}
\def\SLE{\operatorname{SLE}}

\def\VMO{\operatorname{VMO}}

\def\a{\alpha}
\def\b{\beta}
\def\g{\gamma}

\def\d{\delta}

\def\z{\zeta}
\def\t{\theta}

\def\l{\lambda}
\def\k{\kappa}

\def\s{\sigma}

\def\O{\Omega}
\def\vare{\varepsilon}
\def\HH{{\mathbb H}}
\def\Chat{\hat{\m{C}}}

\def\diam{{\rm diam}}

\def\dd{\mathrm{d}}

\newcommand{\ad}[1]{\overline{#1}}

\def\P{\m P}

\def\1{\mathbf{1}}

 \newcommand{\weld}{w}

 \def \1{\mathbf{1}}

\def\Id{\operatorname{Id}}

\author{Yilin Wang \\
\emph{Massachusetts Institute of Technology}\\
\protect\url{yilwang@mit.edu}}

\begin{document}



\maketitle
\begin{abstract}
These notes survey the first results on large deviations of Schramm-Loewner evolutions (SLE) with emphasis on interrelations between rate functions and applications to complex analysis. More precisely, we describe the large deviations of SLE$_\kappa$ when the $\kappa$ parameter goes to zero in the chordal and multichordal case and to infinity in the radial case. The rate functions, namely Loewner and Loewner-Kufarev energies, are closely related to the Weil-Petersson class of quasicircles and real rational functions. 
\end{abstract}
\tableofcontents

\section{Introduction} \label{sec:intro}
These notes aim to overview the first results on the large deviations of Schramm-Loewner evolutions (SLE). 
SLE is a one-parameter family,  indexed by $\k \ge 0$, of random non-self-crossing and conformally invariant curves in the plane. They are introduced by Schramm \cite{Schramm2000}  in 1999 by combining  stochastic analysis with Loewner's century-old theory \cite{Loewner23} for the evolution of planar slit domains.  When $\k > 0$, these curves are fractal, and the parameter $\kappa$ reflects the curve's roughness.
 SLEs play a central role in 2D random conformal geometry. For instance, they describe interfaces in conformally invariant systems arising from scaling limits of discrete statistical physics models, which was also Schramm's original motivation, see, e.g.,~\cite{LSW04LERWUST,Schramm:ICM,Smi:ICM,SS09GFF}. More recently, SLEs are shown to be coupled with random surfaces and provide powerful tools in the study of probabilistic Liouville quantum gravity, see, e.g.,~\cite{duplantier_PRL,Quantum_zipper,IG1,MoT}.
SLEs are also closely related to conformal field theory whose central charge is a function of $\k$, 
see, e.g., 
\cite{BB:CFTSLE, Car03,Friedrich_Werner_03,FK_CFT, Dub_SLEVir1, Peltola}.

Large deviation principle describes the probability of rare events of a given family of probability measures on an exponential scale. The formalization of the general framework of large deviation was introduced by Varadhan \cite{Var66} with many contributions by Donsker and Varadhan around the eighties.  Large deviations estimates have proved to be the crucial tool required to handle many questions in statistics, engineering, statistical mechanics, and applied probability.

In these notes, we only give a minimalist account of basic definitions and ideas from both SLE and large deviation theory, only sufficient for considering the large deviations of SLE. We by no means attempt to give a thorough reference to the background of these two theories and apologize for the omission.  
Our approach focuses on showing how large deviation consideration propels to the discovery (or rediscovery) of interesting deterministic objects from complex analysis, including Loewner energy, Loewner-Kufarev energy, Weil-Petersson quasicircles, real rational functions, foliations, etc., and leads to novel results on their interrelation.
Unlike objects considered in random conformal geometry that are often of a fractal or discrete nature, these deterministic objects, arising from the $\k \to 0+$ or $\infty$ large deviations of SLE (on which the rate function is finite), live in the continuum and are more regular. Nevertheless, we will see that the interplay between these deterministic objects are analogous to many coupling results from random conformal geometry whereas proofs are rather simple and based in analysis.
Impatient readers may skip to the last section where we summarize and compare the quantities and theorems from both random conformal geometry and the large deviation world to appreciate the similarity.   The main theorems presented here are collected from \cite{W1,W2,RW,VW1,peltola_wang,APW,VW2}. Compared to the original papers, in most of the time we choose to outline the intuition and omit proofs or only present the proof in a simple case to illustrate the idea. But we also take the opportunity to clarify some subtle points in the original papers.

\bigskip

\textbf{Acknowledgments:} I would like to thank Fredrik Viklund and an anonymous referee for helpful comments on the manuscript. These notes are written based on the lecture series that I gave at the joint webinar of Tsinghua-Peking-Beijing Normal Universities and at Random Geometry and Statistical Physics online seminars in 2020 during the Covid-19 pandemic.  I thank the organizers for the invitation and the online lecturing experience under pandemic's unusual situation and am supported by NSF grant DMS-1953945.

\subsection{Large deviation principle} \label{sec:general_LDP}
 
We first consider a simple example to illustrate the concept of large deviations.
Let $X \sim \mc N(0,\s^2)$ be a real, centered Gaussian random variable of variance $\s^2$. The density function of $X$ is given by 
$$p_{X} (x) = \frac{1}{\sqrt{2\pi \s^2}} \exp\Big(- \frac{x^2}{2 \s^2}\Big).$$
Let $\vare > 0$, $\sqrt \vare X \sim \mc N(0, \s^2 \vare)$. As $\vare \to 0+$, $\sqrt \vare X$ converges almost surely to $0$, so the probability measure $p_{\sqrt \vare X}$ on $\m R$ converges to the Dirac measure $\d_0$.
Let $M > 0$, the rare event $\{\sqrt \vare X \ge M\}$ has probability
$$\P (\sqrt \vare X \ge M) = \frac{1}{\sqrt{2 \pi \s^2 \vare}} \int_M^{\infty} \exp\Big(- \frac{x^2}{2 \s^2 \vare }\Big) \dd x.$$
To quantify how rare this event happens when $\vare \to 0+$, we have
\begin{align}\label{eq:gaussian_LDP}
\begin{split}
\vare \log \P (\sqrt \vare X \ge M)  =& \,  \vare \log \left(\frac{1}{\sqrt{2 \pi \s^2 \vare}} \int_M^{\infty} \exp\Big(- \frac{x^2}{2 \s^2 \vare}\Big) \dd x \right) \\
 = & - \frac{1}{2} \vare \log (2\pi \s^2 \vare) + \vare \log \int_M^{\infty} \exp\Big(- \frac{x^2}{2 \s^2 \vare}\Big) \dd x \\
 \xrightarrow{\vare\to 0+} & - \frac{M^2}{2\s^2} = : - I_X (M) = - \inf_{x \in [M, \infty)} I_X(x)
\end{split}
\end{align}
where $I_X (x) =  x^2 / 2\s^2$ is called the large deviation rate function of the family $\{\sqrt \vare X\}_{\vare > 0}$.

Now let us state the large deviation principle more precisely. Let $\mc X$ be a Polish space, $\mc B$ its completed Borel $\s$-algebra, $\{\mu_{\vare}\}_{\vare > 0}$ a family of probability measures on $(\mc X, \mc B)$. 
\begin{df}
 A \emph{rate function} is a lower semicontinuous mapping $I :\mc X \to
[0,\infty]$, namely, for all $\a \ge 0$, the sub-level set $\{x : I(x) \le \a \}$ is
a closed subset of $\mc X$. A \emph{good rate function} is a rate function for which all
the sub-level sets are compact subsets of $\mc X$.
\end{df}

\begin{df} \label{df:LDP}
We say that a family of probability measures $\{\mu_{\vare}\}_{\vare > 0}$ on $(\mc X, \mc B)$ satisfies the \emph{large deviation principle of rate function $I$} if for all open set $O \in \mc B$ and closed set $F \in \mc B$, 
 $$ \liminf_{\vare \to 0+} \vare \log \mu_{\vare} (O) \ge - \inf_{x \in O} I(x); \quad   \limsup_{\vare \to 0+} \vare \log \mu_{\vare} (F) \le - \inf_{x \in F} I(x). $$
 It is elementary to show that if a large deviation rate function exists then it is unique, see, e.g., \cite[Lem.\,1.1]{Dinwoodie}.
\end{df}

\begin{remark}\label{rem:continuity_set}
If $A \in \mc B$ satisfies $\inf_{x \in A^o} I(x) = \inf_{x \in \ad A} I(x)$ (we call such Borel set $A$ a \emph{continuity set} of $I$), 
then the large deviation principle gives 
$$\lim_{\vare \to 0+} \vare  \log  \mu_{\vare} (A) =- \inf_{x \in A} I(x).$$
\end{remark}

\begin{remark} \label{rem:LDP_gaussian} Using \eqref{eq:gaussian_LDP}, it is easy to show that the distribution of $\{\sqrt \vare X\}_{\vare >0}$ from the example above satisfies the large deviation principle with good rate function $I_X$. 
\end{remark}

The reader should mind carefully that large deviation results depend on the topology involved which can be a subtle point. On the other hand, it follows from the definition that the large deviation principle transfers nicely through continuous functions:

\begin{thm}[Contraction principle {\cite[Thm.\,4.2.1]{DZ10}}]\label{thm:contraction}
If $\mc X, \mc Y$ are two Polish spaces, $f: \mc X \to \mc Y$ a continuous function, and a family of probability measures  $\{\mu_\vare\}_{\vare > 0}$ on $\mc X$ satisfying the large deviation principle with good rate function $I: \mc X \to [0,\infty]$.
Let $I' : \mc Y \to [0,\infty]$ be defined as 
$$I'(y) := \inf_{x \in f^{-1}\{y\}} I(x).$$ 
Then the family of pushforward probability measures $\{f_* \mu_\vare\}_{\vare > 0}$ on $\mc Y$ satisfies the large deviation principle with good rate function $I'$.
\end{thm}

One classical result, of critical importance to our discussion, is the large deviation principle of the scaled Brownian path. Let $T \in (0,\infty)$, we write
$$C^0[0,T] : = \{W : [0,T] \to \m R \,|\,  t \mapsto W_t \text{ is continuous and }W_0 = 0\}$$
and define similarly $C^0[0,\infty)$.
The \emph{Dirichlet energy} of
$W \in C^0[0,T]$  (resp. $W \in C^0[0,\infty)$) is given by
\begin{align} \label{eq_DE}
I_T(W) := \frac{1}{2} \int_{0}^{T} \abs{\frac{\dd  W_t}{\dd t}}^2 \, \dd t \quad 
\left(\text{resp. } I_\infty(W) := \frac{1}{2} \int_{0}^{\infty} \abs{\frac{\dd  W_t}{\dd t}}^2 \dd t \right)
\end{align}
 if $W$ is absolutely continuous, and set to equal $\infty$ otherwise.
 Equivalently, we can write
 \begin{equation}\label{eq:partition}
 I_T(W) = \sup \sum_{i = 0}^{k-1} \frac{(W_{t_{i+1}} - W_{t_{i}})^2}{2(t_{i+1} - t_i)},    \qquad T \in (0,\infty],
 \end{equation}
where the supremum is taken over all $k \in \m N$ and all partitions $\{0 = t_0 < t_1<\cdots <t_k\le T\}$. In fact,
note that the sum on the right-hand side of \eqref{eq:partition} is the Dirichlet energy of the linear interpolation of $W$ from its values at $(t_0, \cdots, t_k)$ which is set to be constant on $[t_k, T]$. The identity \eqref{eq:partition} then follows from the density of piecewise constant functions in $L^2 (\m R_+)$ applied to the approximation of the function $f(t) := \dd W_t/\dd t$. Notice that on any interval $[s,r]$, the constant $a$ minimizing $\int_{s}^r |f(t) - a|^2 \,\dd t$ is the average of $f$ on $[s,r]$. Therefore, the best approximating piecewise linear functions of $W$ with respect to the partition $\{0 = t_0 < t_1<\cdots <t_k\le T\}$ for the Dirichlet inner product is the linear interpolation of $W$.

\begin{thm} \label{thm:Schilder}
\textnormal{(Schilder; see, e.g.,~\cite[Ch.\,5.2]{DZ10})} 
Fix $T \in (0,\infty)$.
The family of processes $\{(\sqrt{\vare} B_t)_{t \in [0,T]}\}_{\vare > 0}$, viewed as a family of random functions in $(C^0[0,T],\norm{\cdot}_{\infty})$, satisfies the large deviation principle 
 with good rate function $I_T$.
\end{thm}

\begin{remark}\label{rem:path_regular_BM}
We note that Brownian path has almost surely infinite Dirichlet energy, i.e., $I_T(B) = \infty$. In fact, $W$ has finite Dirichlet energy implies that $W$ is $1/2$-H\"older, whereas Brownian motion is only a.s.  $(1/2-\d)$-H\"older for $\d >0$.
However, Schilder's theorem shows that Brownian motion singles out the Dirichlet energy which quantifies, as $\vare \to 0+$, the density of Brownian path around a deterministic function $W$. In fact, let $O_\d (W)$ denote the open ball of radius $\d$ centered at $W$ in $C^0[0,T]$. We have for $\d' >\d$, 
$O_\d (W) \subset \ad{O_\d (W)} \subset O_{\d'} (W).$
From the monotonicity of $\d \mapsto \inf_{\tilde W \in O_\d (W)} I_T(\tilde W)$, $O_\d (W)$ is a continuity set for $I_T$ with exceptions for at most countably many $\d$ (which induce a discontinuity of $\inf_{\tilde W \in O_\d (W)} I_T(\tilde W)$ in $\d$). Hence, by possibly avoiding the exceptional values of $\d$, we have
\begin{equation}\label{eq:path_neighbor}
  - \vare \log \P( \sqrt \vare B \in O_\d (W)) \xrightarrow{\vare \to 0+} \inf_{\tilde W \in O_\d (W)} I_T(\tilde W) \xrightarrow{\d \to 0+}  I_T(W).
\end{equation}
The second limit follows from the lower semicontinuity of $I_T$.
More intuitively, we write with some abuse
\begin{equation}\label{eq:heuristic_LDP}
``  \P( \sqrt \vare B \text{ stays close to }  W) \sim_{\vare \to 0+} \exp(- I_T(W) / \vare)".
\end{equation}
\end{remark}

We now give some heuristics to show that the Dirichlet energy appears naturally as the large deviation rate function of the scaled Brownian motion. 
Fix $0= t_0 < t_1 < \ldots < t_k \le T$.
The finite dimensional marginals of Brownian motion $(B_{t_0}, \ldots,  B_{t_{k}})$ gives a family of independent Gaussian random variables $(B_{t_{i+1}} - B_{t_{i}})_{0 \le i \le k-1}$ with variances $(t_{i+1} - t_i)$ respectively. Multiplying the Gaussian vector by $\sqrt \vare$, we obtain the large deviation principle  of the finite dimensional marginal with rate function 
 $W \mapsto \sum_{i = 0}^{k -1} \frac{(W_{t_{i+1}} - W_{t_{i}})^2}{2(t_{i+1} - t_i)}$ from Remark~\ref{rem:LDP_gaussian}, Theorem~\ref{thm:contraction}, and the independence of the family of increments. Approximating Brownian motion on the finite interval $[0,T]$ by its linear interpolations, it suggests that the scaled Brownian paths satisfy the large deviation principle of rate function the supremum of the rate function of all of its finite dimensional marginals which then turns out to be the Dirichlet energy by \eqref{eq:partition}.  
 
A rigorous proof of Schilder's theorem  uses the Cameron-Martin theorem which allows generalization to any abstract Wiener space. Namely, the associated family of Gaussian measures scaled  by $\sqrt \vare$ satisfies the large deviation principle with the rate function being $1/2$ times its Cameron-Martin norm. See, e.g., \cite[Thm.\,3.4.12]{DeuschelStroock}. This result applies to the \emph{Gaussian free field} (GFF), which is the generalization of Brownian motion where the time parameter belongs to a higher dimension space, and the rate function is again the Dirichlet energy (on the higher dimension space). 

Schilder's theorem also holds when $T = \infty$ using the following projective limit argument. 

\begin{df}\label{df:projective}
 A \emph{projective system} $(\mc Y_j, \pi_{ij})$ consists of Polish spaces\footnote{In fact, one may require $\mc Y_j$ to be just  Hausdorff topological spaces and $j \in J$ belong to a partially ordered, right-filtering set $(J, \le)$ which may be uncountable, see \cite[Sec.\,4.6]{DZ10}.} $\{\mc Y_j\}_{j \in \m N}$ and continuous maps $\pi_{ij} : \mc Y_j \to \mc Y_i$ such that $\pi_{jj}$ is the identity map on $\mc Y_j$ and $\pi_{ik} = \pi_{ij} \circ \pi_{jk}$ whenever $i \le j \le k$. The \emph{projective limit} of this system is the subset $$\mc X := \varprojlim \mc Y_j := \{(y_j)_{j \in \m N} \, | \, y_i = \pi_{ij} (y_j), \, \forall i \le j\} \subset \prod_{j \in \m N} \mc Y_j,$$ endowed with the induced topology by the infinite product space $\prod_{j \in \m N} \mc Y_j$.
 In particular, the canonical projection $\pi_j : \mc X \to \mc Y_j$ defined as the $j$-th coordinate map is continuous.
\end{df}
\begin{ex}\label{ex:proj_C}
The projective limit of $(C^0[0,j], \pi_{ij})$, where $\pi_{ij}$ is the restriction map from $C^0[0,j] \to C^0[0,i]$ for $i\le j$, is homeomorphic to $C^0[0,\infty)$ endowed with the topology of uniform convergence on compact sets. 
\end{ex}

\begin{thm}[Dawson-G\"artner {\cite[Thm.\,4.6.1]{DZ10}}]\label{thm:Dawson-Gartner}
Assume that $\mc X$ is the projective limit of $(\mc Y_j, \pi_{ij})$. Let $\{\mu_\vare\}_{\vare >0}$ be a family of probability measures on $\mc X$, such that for any $j \in \m N$, the probability measures $\{\mu_\vare \circ \pi_j^{-1}\}_{\vare > 0}$ on $\mc Y_j$ satisfies the large deviation principle with the good rate function $I_j$. Then $\{\mu_\vare\}_{\vare > 0}$ satisfies the large deviation principle with the good rate function
$$I ((y_j)_{j \in \m N}) = \sup_{j \in \m N} I_j(y_j), \quad (y_j)_{j \in \m N} \in \mc X.$$
\end{thm}

Example~\ref{ex:proj_C}, Theorems~\ref{thm:contraction} and \ref{thm:Dawson-Gartner} imply the following Schilder's theorem on the infinite time interval.

\begin{cor}
The family of processes $\{(\sqrt{\vare} B_t)_{t \ge 0}\}_{\vare > 0}$ satisfies the large deviation principle in $C^0[0,\infty)$ endowed with the topology of uniform convergence on compact sets
 with good rate function $I_\infty$.
\end{cor}

\subsection{Chordal Loewner chain}\label{sec:chordal_Loewner}

The description of SLE is based on the Loewner transform, a deterministic procedure that encodes a non-self-crossing curve on a 2-D domain into a driving function.
In this survey, we use two types of Loewner chain:  the \emph{chordal Loewner chain} in $(\domain;\bpt,\ept)$, where $\domain$ is a simply connected domain with two distinct boundary points  $\bpt$ (starting point) and $\ept$ (target point); and later in Section~\ref{sec:radial_infty}, the \emph{radial Loewner chain} in $\domain$ targeting at an interior point.
The definition is invariant under conformal maps (namely, biholomorphic functions). Hence, by Riemann mapping theorem, it suffices to describe in the chordal case when $(\domain;\bpt,\ept) = (\m H; 0,\infty)$, and in the radial case when $\domain = \m D$, targeting at $0$. 
Throughout the article, $\m H = \{z\in \m C : \Im (z) > 0\}$ is the upper halfplane, $\m H^*  = \{z\in \m C : \Im (z) < 0\}$ is the lower halfplane, $\m D = \{z \in \m C : |z| < 1\}$ is the unit disk, and $\m D^* =\{z \in \m C : |z| > 1\} \cup \{\infty\}$.

We say that $\g$ is a {\em simple curve} in $(D;\bpt,\ept)$, if $D$ is a simply connected domain, $\bpt, \ept$ are two distinct prime ends of  $D$, and $\g$ has a continuous and injective parametrization $(0,T) \to D$ such that $\g(t) \to \bpt$ as $t \to 0$ and $\g(t) \to c \in D \cup \{\ept\}$ as $t \to T$. If $c = \ept$ then we say $\g$ is a {\em chord} in $(D;\bpt,\ept)$.

Let us start with this chordal Loewner description of a simple curve $\gamma$ in $(\m H; 0,\infty)$. We parameterize the curve by the halfplane capacity. More precisely, $\g$ is continuously parametrized by $[0,T)$, where $T \in (0, \infty]$ with $\g_0 = 0$, $\g_t \to \m H \cup \{\infty\}$ as $t \to T$,
in the way such that for all $t \in [0,T)$,
the unique conformal map $g_t$ from $\HH \smallsetminus \gamma_{[0,t]}$ onto $\HH$ with the expansion at infinity $g_t (z)=  z + o(1)$ satisfies
\begin{equation}\label{eq:hydro}
g_t (z)= z + \frac{2t}{ z} + o\left( \frac{1}{z}\right).
\end{equation}
The coefficient $2t$ is the \emph{halfplane capacity} of $\g_{[0,t]}$.
It is easy to  show that $g_t$ can be extended by continuity to the boundary point $\gamma_t$ and that the real-valued function $W_t := g_t ( \gamma_t)$ is continuous with $W_0 = 0$ (i.e., $W \in C^0[0,T)$). This function $W$ is called the \emph{driving function} of $\g$ and $2T$ the total capacity of $\g$.

\begin{remark}\label{rem:pathological}
There are chords with finite total capacity. Namely, $T <\infty$ and $\g_t \to \infty$ as $t \to T$. It happens only when $\g$ goes to infinity while staying close to the real line \cite[Thm.\,1]{LLN_capacity}. 
\end{remark}

Conversely, the \emph{chordal Loewner chain} in $(\m H; 0,\infty)$ driven by a continuous real-valued function $W \in C^0[0,T)$ is the family of conformal maps $(g_t)_{t \in [0,T)}$, obtained
 by solving the Loewner equation for each $z \in \ad{\m H}$,
\begin{align} \label{eqn:LE}
\partial_{t} g_t(z) = \frac{2}{g_t(z)-W_t}  
\qquad \text{with initial condition} \qquad  g_0(z)=z. 
\end{align}
In fact, the solution $t \mapsto g_t(z)$ to~\eqref{eqn:LE} is defined up to the swallowing time of $z$
\begin{align*}
\tau(z) := \sup\{ t \geq 0 \, |\, \inf_{s\in[0,t]}|g_{s}(z)-W_{s}|>0\},
\end{align*}
 which is set to $0$ when $z = 0$. We obtain an increasing family of $\ad{\m H}$-hulls $(K_t := \{z \in \ad{\m H} \, |\,  \tau(z) \leq t\})_{t  \in [0,T)}$ (a compact subset $K \subset \ad{\m H}$ is called a $\ad{\m H}$\emph{-hull} if  $\ad{K \cap \m H} = K$ and $\m H \smallsetminus K$ is simply connected).
Moreover, the solution $g_t$ of \eqref{eqn:LE} restricted to $\m H \smallsetminus K_t$ is the unique conformal map from $\m H \smallsetminus K_t$ onto $\m H$ that satisfies the expansion \eqref{eq:hydro}. See, e.g., \cite[Sec.\,4]{Law05} or \cite[Sec.\,2.2]{WW_St_Flour}.
Clearly $K_t$ and $g_t$ uniquely determine each other. We list a few properties of the Loewner chain.

\begin{itemize}[itemsep = 1pt]
    \item If $W$ is the driving function of a simple chord $\g$ in $(\m H; 0, \infty)$, we have $K_t = \g_{[0,t]}$, and the solution $g_t$ of \eqref{eqn:LE} is exactly the conformal map constructed from $\g$ as in \eqref{eq:hydro}.
\item The  imaginary axis $\ii \m R_+$ is driven by $W \equiv 0$ defined on $\m R_+$.
\item (\emph{Additivity}) Let $(K_t)_{t  \in [0,T)}$ be the family of hulls generated by the driving function $W$. Fix $s > 0$, the driving function generating $(\ad{g_s(K_{t+s} \smallsetminus K_s)} - W_s)_{t \in [0,T-s)}$ is $ t \mapsto W_{s+t} - W_s$.
\item (\emph{Scaling}) Fix $\l >0$, the driving function generating the scaled and capacity-reparameterized family of hulls $(\l K_{\l^{-2}t})_{t  \in [0, \l^2 T)}$ is  $t\mapsto \l W _{\l^{-2}t}$.
This property implies that the driving function of the ray $\{z \in \m H \,|\, \arg z = \alpha\}$ is $W_t = C \sqrt t$, where $C$ only depends on  $\alpha \in (0,\pi)$.
      \item Not every continuous driving function arises from a simple chord. It is unknown how to characterize analytically the class of functions which generate simple curves. Sufficient conditions exist, such as when $W$ is $1/2$-H\"older with H\"older norm strictly less than $4$ \cite{Marshall_Rohde,Lind_sharp}.
\end{itemize}

\subsection{Chordal SLE}\label{sec:SLE}
  We now very briefly review the definition and relevant properties of chordal SLE. 
 For further SLE background, we refer the readers to, e.g., \cite{Law05,WW_St_Flour}.
 The \emph{chordal Schramm-Loewner evolution} of parameter $\k$ in $({\m H}; 0, \infty )$, denoted by $\SLE_\k$, is the process of hulls $(K_t)_{t\ge 0}$ generated by $\sqrt \k B$ via the Loewner transform, where $B$ is the standard Brownian motion and $\k \ge 0$. 
 Rohde and Schramm showed that $\SLE_\k$ is almost surely traced out by a continuous non-self-crossing curve $\g^\k$, called the \emph{trace} of $\SLE_\k$, such that $\m H \smallsetminus K_t$ is the unbounded connected component of $\m H \smallsetminus \g^\k_{[0,t]}$ for all $t \ge 0$. 
 Moreover, SLE traces exhibit phase transitions depending on the value of $\k$:
\begin{thm}[\!\!\cite{Rohde_Schramm,LSW04LERWUST}]\label{thm_transient_sle}
The following statements hold almost surely: 
For $\k\in [0,4]$, $\g^\k$ is a chord. For $\k\in (4,8)$, $\g^\k$ is a self-touching curve. For $\k\in [8,\infty)$, $\g^\k$ is a space-filling curve.
Moreover, for all $\k \ge 0$, $\g^\k$ goes to $\infty$ as $t \rar \infty$. 
\end{thm}
 
 The SLEs have attracted a great deal of attention during the last 20 years, as they are the first construction of random self-avoiding paths 
and describe the interfaces in the scaling limit of various statistical mechanics models, e.g.,  
 \begin{itemize}[itemsep=-2pt]
\item $\SLE_2$ $\leftrightarrow$ Loop-erased random walk \cite{LSW04LERWUST};
\item $\SLE_{8/3}$ $\leftrightarrow$ Self-avoiding walk (conjecture);
\item $\SLE_3$ $\leftrightarrow$ Critical Ising model interface \cite{Smirnov10};
\item $\SLE_4$ $\leftrightarrow$ Level line of the Gaussian free field \cite{SS09GFF};
\item $\SLE_6$ $\leftrightarrow$ Critical independent percolation interface \cite{Smi:ICM}; 
\item $\SLE_8$ $\leftrightarrow$ Contour line of uniform spanning tree \cite{LSW04LERWUST}.
\end{itemize}

The reason that SLE curves describe those interfaces arising from conformally invariant systems is that they are the unique random Loewner chain that are \emph{scaling}-\emph{invariant} and satisfy the \emph{domain Markov property}. More precisely, for $\l >0$, the law of SLE is invariant under the scaling transformation
$$ (K_t)_{t\geq 0}  \mapsto (K^\l_t  := \l K_{\l^{-2}t})_{t\geq 0}$$ 
 and for all $s\in[0,\infty)$, if one defines
$K_t^{(s)} =  \ad{g_s(K_{s+t}\smallsetminus K_s)} - W_s$,
 where $W$ drives $(K_t)$, 
then $(K_t^{(s)})_{t\geq 0}$ has the same distribution as $(K_t)_{t \geq 0}$ and is independent of $\s(W_r: r\leq s)$. 
In fact, these two properties on $(K_t)$ translate into the properties of the driving function $W$: having independent stationary increments (i.e., being a L\'evy process) and being invariant under the transformation $W_t \leadsto \l W_{\l^{-2} t}$.  Multiples of Brownian motions are the only continuous processes satisfying these two properties. 

The scaling-invariance of SLE in $(\m H; 0,\infty)$ makes it possible to define $\SLE$ in other simply connected domains $(\domain; \bpt, \ept)$ as the preimage of SLE in $(\m H; 0, \infty )$ by a conformal map $\varphi : \domain \to \m H $ sending respectively the prime ends $\bpt, \ept$ to $0, \infty$, since another choice of $\tilde \varphi$ equals $\l \varphi$ for some $\l > 0$.  The chordal SLE is therefore conformally invariant from the definition.

\begin{remark}\label{rem:geodesic}
The $\SLE_0$ in $(\m H; 0,\infty)$ is simply the Loewner chain driven by $W \equiv 0$, namely the imaginary axis $\ii \m R_+$. It implies that the $\SLE_0$ in $(\domain; \bpt, \ept)$ equals $\varphi^{-1} (\ii \m R_+)$ (i.e., the \emph{hyperbolic geodesic} in $\domain$ connecting $\bpt$ and $\ept$).
\end{remark}

\section{Large deviations of chordal $\SLE_{0+}$} \label{sec:LDP_single}
\subsection{Chordal Loewner energy and large deviations}
To describe the large deviations of chordal $\SLE_{0+}$ (see Theorem~\ref{thm:main_single_LDP}), 
let us first specify the topology on the space of simple chords that we consider.

\begin{df} \label{defn: Hausdorff}
The \emph{Hausdorff distance} $d_h$ of two compact subsets $\closed_1, \closed_2 \subset \ad{\m D}$ is defined as
\begin{align*}
d_h(\closed_1, \closed_2) := \inf \Big\{ \vare \geq 0 \; \Big| \; \closed_1 \subset \bigcup_{x \in \closed_2} \ad{B}_\vare(x) \; \text{ and } \; \closed_2 \subset \bigcup_{x \in \closed_1} \ad{B}_\vare(x) \Big\} ,
\end{align*}
where $B_\vare(x)$ denotes the Euclidean ball  of radius $\vare$ centered at $x \in \ad{\m D}$.
We then define the Hausdorff metric on the set of closed subsets of a Jordan domain\footnote{When $\domain$ is bounded by a Jordan curve, Carath\'eodory theorem implies that a uniformizing conformal map $\domain \to \m D$ extends to a homeomorphism between the closures $\ad \domain \to \ad{\m D}$.} $\domain$ via the pullback by a uniformizing conformal map $\domain \to \m D$. 
Although the metric depends on the choice of the conformal map, the topology induced by $d_h$ is canonical, as conformal automorphisms of $\m D$ are fractional linear functions (i.e., M\"obius transformations) which are uniformly continuous on $\ad {\m D}$. 
\end{df}

\begin{df}\label{df:chordalLE}
The \emph{Loewner energy} of a simple curve $\g$ is defined as the Dirichlet energy~\eqref{eq_DE} of its driving function,
\begin{align} \label{eq_LE}
I_{\domain;\bpt, \ept} (\g)
:= I_{\m H; 0 ,\infty}(\varphi (\g)) 
: = I_{T} (W),
\end{align}
where $\varphi$ is any conformal map from $D$ to $\m H$ 
such that $\varphi(\bpt) = 0$ and $\varphi(\ept) = \infty$,
$W$ is the driving function of $\varphi (\g)$, $2T$ is the total capacity of $\varphi(\g)$, and $I_{T} (W)$ is the Dirichlet energy as defined in \eqref{eq_DE} and \eqref{eq:partition}. 
\end{df}
Note that the definition of $I_{\domain;\bpt, \ept} (\g)$ does not depend on the choice of $\varphi$ either. In fact, two choices differ only by post-composing by a scaling factor. From the scaling property of the Loewner driving function, $W$ changes to $t\mapsto \l W_{\l^{-2} t}$ for some $\l > 0$, which has the same Dirichlet energy as $W$.
\begin{remark}\label{rem:minimizer_geodesic}
The Loewner energy $I_{\domain;\bpt, \ept} (\g)$ is non-negative and minimized by the 
hyperbolic geodesic $\eta$ since the driving function of $\varphi(\eta)$ is the constant function $W \equiv 0$ and $I_{\domain;\bpt, \ept}(\eta) = 0$.
\end{remark}

\begin{thm}[]\label{thm:single_qchord}
If $\g$ is a chord and $I_{D;a,b} (\g) < \infty$, then $T = \infty$ (namely, $\g$ has infinite total capacity).  If a driving function $W$ defined on $\m R_+$ satisfies $I_{\infty} (W) <\infty$, then $W$ generates a chord $\g$ in $(\m H; 0,\infty)$. Moreover, $\g$ is a \emph{quasichord}, i.e., the image of $\ii \m R_+$ by a quasiconformal homeomorphism $\m H \to \m H$ fixing $0$ and $\infty$.
\end{thm}

 A \emph{quasiconformal} map is a weakly differentiable homeomorphism that maps infinitesimal circles to infinitesimal ellipses with uniformly bounded eccentricity.  For a brief introduction to the theory of quasiconformal maps, readers might refer to \cite[Ch.\,1]{lehto2012univalent}.

\begin{proof}
 From \cite[Prop.\,3.1]{W1}, if a chord in $(\m H; 0, \infty)$ has finite energy, then there is $\t \in (0, \pi/2)$, such that $\g$ is contained in the cone $\{z \in \m H \,|\, \t \le \arg z \le \pi - \t\}$. This implies that the total capacity of $\g$ is infinite by Remark~\ref{rem:pathological}.
The second claim is proved in \cite[Prop.~2.1]{W1}, which is essentially a consequence of the fact that $1/2$-H\"older driving function with small H\"older norm generates quasichords.
\end{proof}

Theorem~\ref{thm:single_qchord} motivates us to consider the space $\mc X(\domain;\bpt;\ept)$ of unparametrized simple chords with infinite total capacity in $(\domain;\bpt;\ept)$. We endow this space with the relative topology induced by the Hausdorff metric.  
Theorem~\ref{thm:Schilder} suggests that the Loewner energy is the large deviation rate function of $\{\SLE_\k\}_{\k > 0}$, with $\vare = \k$. Indeed, the following result is proved in \cite{peltola_wang} which strengthens a similar result in \cite{W1}.  As we are interested in the $0+$ limit, we only consider $\k \le 4$ where the \emph{trace} $\g^\k$ of $\SLE_\k$ is almost surely in $\mc X(\domain;\bpt,\ept)$.   

\begin{thm}[\!\! {\cite[Thm.\,1.5]{peltola_wang}}]
   \label{thm:main_single_LDP}
The family of distributions $\{\m P^\k\}_{\k > 0}$ on $\mc X (\domain;\bpt,\ept)$  of the chordal $\SLE_\k$ curves
satisfies the large deviation principle with good rate function $I_{\domain;\bpt, \ept}$. That is, for any open set $\open$ and closed set $\closed$  of $\mc X(\domain;\bpt,\ept)$, we have
\begin{align*}
\liminf_{\k \to 0+} \k \log  \m P^\k [\g^\k \in \open ] 
\geq \; & - \inf_{\g \in \open}  I_{\domain;\bpt, \ept}(\g) , \\
\limsup_{\k \to 0+} \k \log  \m P^\k [ \g^\k \in \closed ] 
\leq \; & - \inf_{\g \in \closed}  I_{\domain;\bpt, \ept}(\g) , 
\end{align*}
and the sub-level set $\{\g \in \mc X (\domain;\bpt,\ept) \,| \, I_{\domain;\bpt, \ept}(\g) \le c \}$ is compact for any $c \geq 0$.
\end{thm}

We note that the Loewner transform which maps a continuous driving function to the union of the hulls it generates is not continuous with respect to the Hausdorff metric. Therefore, we cannot deduce trivially the result using Schilder's theorem and the contraction principle (Theorem~\ref{thm:contraction}). This result thus requires some work and is rather technical, see \cite[Sec.\,5]{peltola_wang} for details.

\begin{remark}\label{rem:more_regular_than_SLE}
As Remark~\ref{rem:path_regular_BM}, we emphasize that finite energy chords are more regular than SLE$_\k$ curves for any $\k > 0$. In fact, we will see in Theorem~\ref{thm:intro_equiv_energy_WP} that finite energy chord is part of a Weil-Petersson quasicircle which is rectifiable, see Theorem~\ref{thm:asymptotically_smooth}. 
 Moreover, finite energy curves do not have corners. If the curve as an angle at time $t$, then the driving function after time $t$ is approximated by $W_{t+s} \approx C \sqrt {s} +W_t$ by additivity and scaling property, see Section~\ref{sec:chordal_Loewner}, which has infinite Dirichlet energy. 
On the other hand, Beffara \cite{Beffara_dimension} shows that for $0< \k \le 8$, $\SLE_\k$ has Hausdorff dimension $1 + \k/8 > 1$ and thus is not rectifiable when $\k >0$.
\end{remark}

\subsection{Reversibility of Loewner energy}

Given that for specific values of $\k$, $\SLE_\kappa$ curves are the scaling limits of interfaces in statistical mechanics lattice models, it was natural to conjecture that they are reversible 
 since interfaces are a priori unoriented. This conjecture was first proved by Zhan \cite{Zhan}  for all $\kappa \in [0, 4]$, i.e., in the case of simple curves, via  couplings of both ends of the  SLE
path. See also Dub\'edat's commutation relations \cite{Dub_comm}, and Miller and Sheffield's approach based on the
Gaussian Free Field  \cite{IG1,IG2,IG3} which also provides a proof in the non-simple curve case when $\kappa \in (4,8]$. 

\begin{thm} [SLE reversibility \cite{Zhan}]\label{thm:reversibility_SLE}
   For $\k \in [0,4]$, the distribution of the trace $\g^\k$ of $\SLE_\k$ in $(\m{H}; 0, \infty)$ coincides with that of its image under $\iota: z \rar -1/z$ upon forgetting the time parametrization.
\end{thm}

We deduce from Theorem~\ref{thm:main_single_LDP} and Theorem~\ref{thm:reversibility_SLE} the following result.
\begin{thm}[Energy reversibility \cite{W1}] \label{thm:intro_rev} We have
    $I_{\domain;\bpt,\ept} (\g) = I_{\domain;\ept,\bpt} (\g)$
    for any chord $\g \in \mc X (\domain;\bpt,\ept)$.
\end{thm} 

\begin{proof} 
Without loss of generality, we assume that $(\domain; \bpt, \ept) = (\m H; 0,\infty)$ and show that $I_{\m H; 0,\infty} (\g) =  I_{\m H; 0,\infty} (\iota (\g))$.

We use a conformal map $\vartheta :\m H \to \m D$ that maps $\ii$ to $0$ to define the pullback Hausdorff metric $\vartheta^* d_h$ on the set of closed subsets of $\ad{\m H}$ as in Definition~\ref{defn: Hausdorff}. Our choice of $\vartheta$ satisfies $\vartheta \circ \iota \circ \vartheta^{-1} = - \Id_{\m D}$. In particular, $\iota$ induces an isometry on closed subsets of $\ad{\m H}$.
 Let $\d_n$ be a sequence of numbers converging to $0$ from above, such that $$O_{\d_n} (\g): = \{\tilde \g \in \mc X(\m H; 0, \infty) : \vartheta^* d_h (\g, \tilde \g) < \d_n\}$$
 is a continuity set for $I_{\m H; 0,\infty}$. The sequence exists since there are at most a countable number of $\d$ such that $O_{\d} (\g)$ is not a continuity set as we argued in Remark~\ref{rem:path_regular_BM}. 
From Remark~\ref{rem:continuity_set}, 
\begin{equation}\label{eq:SLE_decay}
\lim_{\k \to 0+} \k \log \m P^\k (\g^\k \in  O_{\d_n} (\g))  = -\inf_{\tilde \g \in O_{\d_n} (\g)} I_{\m H; 0, \infty} (\tilde \g),
\end{equation}
which tends to $- I_{\m H; 0, \infty} (\g)$ as $n \to \infty$ from the lower-semicontinuity of $I_{\m H; 0, \infty}$.
 Theorem~\ref{thm:reversibility_SLE} then shows that
\begin{equation*}
   \m P^\k (\g^\k \in  O_{\d_n} (\g)) = \m P^\k (\iota(\g^\k) \in  \iota (O_{\d_n} (\g))) =  \m P^\k (\g^\k \in  O_{\d_n} (\iota (\g))). 
\end{equation*}
The last equality used the fact that $\iota$ induces an isometry.
We obtain the claimed energy reversibility by applying \eqref{eq:SLE_decay} to $\iota(\g)$.
\end{proof}

\begin{remark}
This proof is different from \cite{W1} but the idea is very close. We use here Theorem~\ref{thm:main_single_LDP} from the recent work \cite{peltola_wang}, whereas the original proof in \cite{W1} uses the more complicated left-right passing events without the strong large deviation result at hand. 
\end{remark}
 
\begin{remark} The energy reversibility is a result about deterministic chords although the proof presented above relies on the probabilistic theory of SLE. 

We note that from the definition alone, the reversibility is not obvious as the setup of Loewner evolution is directional. 
To illustrate this, consider a driving function $W$ with finite Dirichlet energy that is constant  (and contributes $0$ Dirichlet energy) after time $1$. From the additivity property of driving function, $\g_{[1,\infty)}$ is the hyperbolic geodesic in $\m H \smallsetminus \g_{[0,1]}$ with end points $\g_1$ and $\infty$. The reversed curve $\iota (\g)$ is a chord starting with an analytic curve which is different from the imaginary axis. Therefore unlike $\g$, the energy of $\iota (\g)$ typically spreads over the whole time interval $\m R_+$. 
\end{remark}

\subsection{Loop energy and Weil-Petersson quasicircles}
We now generalize the Loewner energy to Jordan curves (simple loops) on the Riemann sphere $\Chat = \m C \cup\{\infty\}$. This generalization reveals more symmetries of the Loewner energy (Theorem~\ref{thm:intro_root_invariance}). Moreover, an equivalent description (Theorem~\ref{thm:intro_equiv_energy_WP}) of the loop energy will provide an analytic proof of those symmetries including the reversibility and a rather surprising link to the class of Weil-Petersson quasicircles.

 Let $\g: [0,1] \to \Chat$ be a continuously parametrized Jordan curve with the marked point $\g (0) = \g(1)$. 
For every $\vare>0$, $\g [\vare, 1]$ is a chord connecting $\g(\vare)$ to $\g (1)$ in the simply connected domain $\Chat \smallsetminus \g[0, \vare]$.
The \emph{rooted loop  Loewner  energy} of $\g$ rooted at $\g(0)$ is defined as
$$I^L(\g, \g(0)): = \lim_{\vare \to 0} I_{\Chat \smallsetminus \g[0, \vare]; \g(\vare), \g(0)} (\g[\vare, 1]).$$

\begin{figure}
 \centering
 \includegraphics[width=0.8\textwidth]{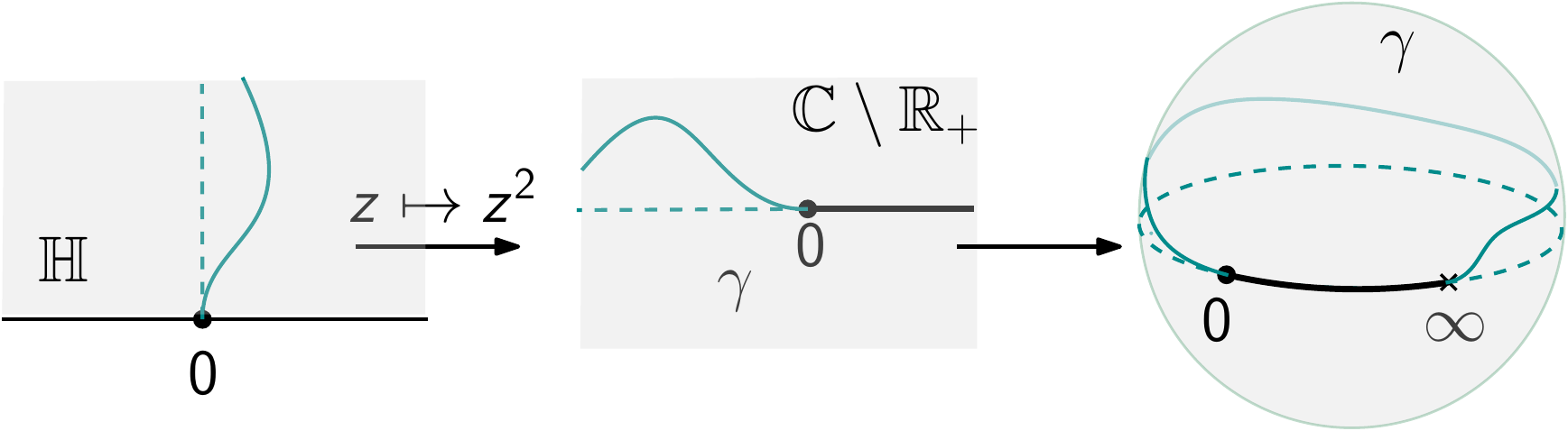}
 \caption{ 
 From chord in $(\m H; 0, \infty)$ to a Jordan curve. \label{fig:squared}}  
 \end{figure}

The loop energy generalizes the chordal energy. In fact, let $\eta$ be a simple chord in $(\m C \smallsetminus \m R_+; 0, \infty)$ and we parametrize $\g = \eta \cup \m R_+$ in a way such that $\g[0,1/2] = \m R_+ \cup \{\infty\}$ and $\g[1/2, 1] = \eta$. Then 
from the additivity of chordal energy (which follows from the additivity of the Loewner driving function),
\begin{align*}
I^L(\g, \infty) = & I_{\m C\smallsetminus \m R_+; 0, \infty} (\eta) + \lim_{\vare \to 0} I_{\Chat \smallsetminus \g[0, \vare]; \g(\vare), \g(1)} (\g[\vare, 1/2]) 
=  I_{\m C\smallsetminus \m R_+; 0, \infty} (\eta),
\end{align*}
since $\g[\vare, 1/2]$ is contained in the hyperbolic geodesic\footnote{Here, $\g[\vare, 1/2]$ is part of a chord but does not make all the way to the target point $\g(1)$, its energy is defined as $I_T(W)$ where $W$ is the driving function of $\g[\vare, 1/2]$ which is defined on an interval $[0,T]$.}  between $\g(\vare)$ and $\g(0)$ in $\Chat \smallsetminus \g[0, \vare]$ for all $0 < \vare <1/2$, see Figure~\ref{fig:squared}.
Rohde and the author proved the following result.

\begin{thm}[\!\! \cite{RW}] \label{thm:intro_root_invariance}
The loop energy does not depend on the root chosen.
\end{thm}
We do not present the original proof of this theorem since it will follow immediately from Theorem~\ref{thm:intro_equiv_energy_WP}, see Remark~\ref{rem:new_proof_invariance}.
\begin{remark}
 From the definition, the loop energy $I^L$ is invariant under M\"obius transformations of $\Chat$, and $I^{L} (\g) = 0$ if and only if $\g$ is a circle (or a line).
\end{remark}

 \begin{remark}
 The loop energy is presumably the large deviation rate function of SLE$_{0+}$ loop measure on $\Chat$ constructed in \cite{zhan2020sleloop} (see also \cite{Werner_loop,Benoist_loop} for the earlier construction of SLE loop measure when $\k = 8/3$ and $2$). However, the conformal invariance of the SLE loop measures implies that they have infinite total mass and has to be renormalized properly for considering large deviations. We do not claim it here and think it is an interesting question. However,
 these ideas will serve as heuristics to speculate results for finite energy Jordan curves in Section~\ref{sec:interplay}.
 \end{remark}

In \cite{RW} we also showed that if a Jordan curve has finite energy, then it is a \emph{quasicircle}, namely the image of a circle or a line under a quasiconformal map of $\m C$.
However, not all quasicircles have finite energy since they may have Hausdorff dimension larger than $1$.
The natural question is then to identify the family of finite energy quasicircles. 
The answer is surprisingly a family of so-called {\em Weil-Petersson quasicircles}, which has been studied extensively by both physicists and mathematicians since the eighties, see, e.g., \cite{BowickRajeev1987string,witten88,NagVerjovsky,Nag_Sullivan,STZ_KdV,cui00,TT06,sharon20062d,Figalli_circle,shen13,GGPPR,Bishop_WP,johansson2021strong},
and is a very active research area.  See the introduction of \cite{Bishop_WP} for a summary and a list of currently more than twenty equivalent definitions of very different nature. 

The class of Weil-Petersson quasicircles is preserved under M\"obius transformation, so without loss of generality, we will use the following definition of a \emph{bounded} Weil-Petersson quasicircle which is the simplest to state. 
      Let $\gamma$ be a bounded Jordan curve. We write $\Omega$ for the bounded connected component of $\Chat \smallsetminus \gamma$ and $\Omega^*$ for the connected component containing $\infty$. Let $f$ be a conformal map $\m D \to \Omega$ and $h : \m D^* \to \Omega^*$ fixing~$\infty$. 
 
\begin{df} The bounded Jordan curve $\gamma$ is a \emph{Weil-Petersson quasicircle} if and only if the following equivalent conditions hold:
\begin{enumerate}
\item  $\mc D_{\m D} (\log \abs{f'})  =  \frac{1}{\pi}\int_{\m D} \left|\nabla \log \abs{f'}\right|^2  \dd  A =  \frac{1}{\pi}\int_{\m D} \left|f''/f'\right|^2  \dd  A < \infty;$
\item $ \mc D_{\m D^*} (\log \abs{h'})  <\infty$,
\end{enumerate}
where $\dd A$ denotes the Euclidean area measure and $\mc D_{\O} (\varphi) : = \frac{1}{\pi} \int_\O |\nabla \varphi|^2 \dd A$ denotes the Dirichlet energy of $\varphi$ in $\O$.
\end{df}

\begin{thm}[\!{\cite[Thm.\,1.4]{W2}}]\label{thm:intro_equiv_energy_WP}
A bounded Jordan curve $\gamma$ has finite Loewner energy if and only if $\gamma$ is a Weil-Petersson quasicircle. Moreover, we have the identity
       \begin{equation} \label{eq_disk_energy}
   I^L(\gamma) =  \mc D_{\m D} (\log \abs{f'}) + \mc D_{\m D^*} (\log \abs{h'})+4 \log \abs{f'(0)} - 4 \log \abs{h'(\infty)},
 \end{equation}
 where $h'(\infty):=\lim_{z\to \infty} h'(z)$.
\end{thm}

\begin{remark}\label{rem_infinite_curve_LE}
If $\gamma$ is a Jordan curve passing through $\infty$, then
\begin{equation}\label{eq:loewner_energy_infinite}
  I^L(\gamma) =  \mc D_{\m H} (\log \abs{f'}) + \mc D_{\m H^*} (\log \abs{h'})
\end{equation}
    where $f$ and $h$ map conformally $\m H$ and $\m H^*$ onto, respectively, $H$ and $H^*$, the two components of $\mathbb{C} \smallsetminus \gamma$, while fixing $\infty$. See \cite[Thm.\,1.1]{W2}.
    
 In fact, the identity \eqref{eq:loewner_energy_infinite} was proved first: We approximate $\g$ by curves generated by piecewise linear driving function, in which case \eqref{eq:loewner_energy_infinite} is checked by explicit computations. We then establish from \eqref{eq:loewner_energy_infinite} an expression of $I^L(\g)$ in terms of zeta-regularized determinants of Laplacians \cite[Thm.\,7.3]{W2} via the Polyakov-Alvarez formula. The expression using determinants has the advantage of being invariant under conformal change of metric and allows us to move the point $\infty$ away from $\g$ and obtain \eqref{eq_disk_energy}.
\end{remark}
\begin{remark}\label{rem:new_proof_invariance}
 Note that the proof of Theorem~\ref{thm:intro_equiv_energy_WP} outlined above is purely deterministic and does not rely on previous results on the reversibility and root-invariance of the Loewner energy. The right-hand side of \eqref{eq_disk_energy} clearly does not depend on any parametrization of $\gamma$, thus this theorem provides another proof of Theorem~\ref{thm:intro_rev} and Theorem~\ref{thm:intro_root_invariance}.
\end{remark}

 Now let us comment on the regularity of Weil-Petersson quasicircles. 

 \begin{thm}\label{thm:asymptotically_smooth}
  Weil-Petersson quasicircles are \emph{asymptotically smooth}, namely, chord-arc with local constant $1$: for all $x,y$ on the curve, the shorter arc $\g_{x,y}$ between $x$ and $y$ satisfies $$\lim_{|x-y| \to 0} \textrm{length}\, (\g_{x,y})/|x-y| = 1.$$
 \textnormal{(}We say $\g$ is \emph{chord-arc} if length$ (\g_{x,y})/|x-y|$ is uniformly bounded.\textnormal{)}
 \end{thm} 

 For this, we recall the definition of a few classical functional spaces.
For $S = S^1$ or $\m R$ or a locally rectifiable Jordan curve $\g$, the space $\VMO (S)$ consists of functions $u \in L^1_{loc} (S)$ {\em with vanishing mean oscillation} such that
$$\lim_{|I| \to 0} \frac{1}{|I|} \int_I |u (x) - u_I| \, \dd x = 0,$$
where $I$ denotes an arc on $S$, $\dd x$ denotes the arclength measure, and
$$u_I :=\frac{1}{|I|} \int_I  u \,\dd \t.$$
The \emph{homogeneous Sobolev space $H^{1/2} (S)$} consists of function $u$ defined a.e. on $S$ for which the $H^{1/2}$ semi-norm
$$\norm{u}_{H^{1/2}(S)}^2 : = \frac{1}{2\pi^2}\iint_{S \times S} \abs{\frac{u(s) - u(t)}{s -t}}^2 \,\dd t \,\dd s < \infty.$$
\begin{remark}
When $S = S^1$, Douglas formula  says that $\norm{u}_{H^{1/2}(S^1)}^2 = \mc D_{\m D} (\tilde u)$ where $\tilde u$ is the harmonic extension of $u$ to $\m D$, see, \cite[Thm.\,2-5]{Ahlfors_conformal_invariants}. Applying the Cayley transform $\vartheta: z \mapsto \dfrac{z-i}{z+i}$ from $\m H$ to $\m D$, it is straightforward to check that 
\begin{equation}\label{eq:douglas_R}
    \norm{u \circ \vartheta}_{H^{1/2}(\m R)}^2 = \norm{u}_{H^{1/2}(S^1)}^2 = \mc D_{\m D} (\tilde u) = \mc D_{\m H} (\tilde u \circ \vartheta)
\end{equation}
where the last equality follows from the conformal invariance of Dirichlet energy. This gives the Douglas formula  for $S = \m R$. (We note that there are different conventions in the definition of $\norm{\cdot}_{H^{1/2}(\m R)} $ in the literature depending on if one adds the $L^2$ norm to it to define a norm. 
We opt for the semi-norm here as it coincides with the Dirichlet energy of the harmonic extension and that the space $H^{1/2}(\m R)$ coincides with the pull-back of $H^{1/2}(S^1)$ by the Cayley map.)
\end{remark}

We have $H^{1/2} (S) \subset \VMO (S)$. In fact, let $I \subset S$ be any bounded arc, 
\begin{align}
\frac{1}{|I|}\int_{I}|u-u_{I}|\dd x  & \le \frac{1}{|I|^{2}} \iint_{I \times I}|u(x)-u(y)|\,\dd x \dd y \nonumber \\
& \le  \left(\iint_{I \times I}\frac{|u(x)-u(y)|^{2}}{|x-y|^2} \,\dd x \dd y \right)^{1/2} \label{eq:H1/2_VMO}
\end{align}
by Cauchy-Schwarz inequality.

A holomorphic function $g$ defined on $\m D$ is in {\em VMOA}, if $g$ is the harmonic extension of a function in $\VMO (S^1)$. See, e.g., \cite[Thm\,3.6, Sec.\,5]{Gir_BMOA} for this definition. (There are equivalent definitions which use Hardy spaces, see, e.g., \cite{Pommerenke_VMOA}. Interested readers may consult the survey \cite{Gir_BMOA} which gives a comprehensive introduction to the theory of analytic functions of
bounded mean oscillation in the unit disc including the equivalence between different definitions.)

\begin{proof}[Proof of Theorem~\ref{thm:asymptotically_smooth}]
 For a conformal map $f$ from $\m D$ to $\O$ bounded by a Weil-Petersson quasicircle $\g$, we have by definition
$\mc D_{\m D}(\log|f'|) < \infty$. Setting $g : = \log f'$, this implies that $\int_{\m D} |g'(z)|^2 \dd A (z) <\infty$. The boundary values (taken as radial limits) of Dirichlet functions are in $H^{1/2} (S^1)$. See \eqref{eq:trace_H1/2} for more detailed discussion about this fact. Therefore,
  $\log f' \in$ VMOA. 
   A theorem of Pommerenke \cite[Thm.\,2]{Pommerenke_VMOA} then shows that this implies that $\g$ is asymptotically smooth.
\end{proof}
\begin{remark} 
We have already discussed in Remark~\ref{rem:more_regular_than_SLE} that Weil-Petersson quasicircles cannot have corners. Theorem~\ref{thm:asymptotically_smooth} gives another justification of this fact as around the corner, the chord-arc ratio is bounded away from $1$.
However, to talk about corners, the curve has to have  left and right derivatives. Therefore, a Weil-Petersson quasicircle need not be $C^1$. 
In fact, it is not hard to check using Theorem~\ref{thm:intro_equiv_energy_WP} that the spiral defined by $t \mapsto t \exp (\ii \log \log |1/t|)$ in a neighborhood of $0$ can be completed into a Weil-Petersson quasicircle. See, e.g., \cite[Prop.\,6.5]{michelat2021loewner}.
\end{remark}

The connection between Loewner energy and Weil-Petersson quasicircles goes further: Not only Weil-Petersson quasicircles are exactly those Jordan curves with finite Loewner energy, the Loewner energy is also closely related to the K\"ahler structure on the Weil-Petersson Teichm\"uller space $T_0(1)$, identified to the class of Weil-Petersson quasicircles via a conformal welding procedure. 
In fact, the right-hand side of \eqref{eq_disk_energy} coincides with the universal Liouville action introduced by Takhtajan and Teo \cite{TT06} and shown by them to be a K\"ahler potential of the Weil-Petersson metric, which is the unique homogeneous K\"ahler metric on $T_0(1)$ up to a scaling factor. Summarizing, we obtain the following result.
\begin{cor}\label{cor:energy_potential}
   The Loewner energy is a K\"ahler potential of the Weil-Petersson metric on $T_0(1)$.
\end{cor}
We do not enter into further details as it goes beyond the scope of large deviations that we choose to focus on here and refer the interested readers to \cite{TT06,W2}. This result gives an unexpected link between the probabilistic theory of SLE and Teichm\"uller theory, although the deep reason behind the link remains rather obscure.

\section{Cutting, welding, and flow-lines} \label{sec:interplay}

Pioneering works \cite{Dub_couplings,Quantum_zipper,IG1} on couplings between SLEs and Gaussian free field (GFF) have led to many remarkable applications in 2D random conformal geometry. These coupling results are often speculated from the link with discrete models.
In \cite{VW1}, Viklund and the author provided another viewpoint on these couplings through the lens of large deviations by showing the interplay between Loewner energy of curves and Dirichlet energy of functions defined in the complex plane (which is the large deviation rate function of scaled GFF). 
These results are analogous to the SLE/GFF couplings, but the proofs are remarkably short and use only analytic tools without any of the probabilistic models. 

\subsection{Cutting-welding identity}
 To state the result, 
 we write $\mc{E}(\Omega)$ for the space of real functions on a domain $\Omega \subset \m C$ with weak first derivatives in $L^2(\Omega)$ and recall  the Dirichlet energy of $\varphi \in \mc E(\Omega)$ is 
 \[\mc D_\Omega(\varphi) := \frac{1}{\pi} \int_\Omega |\nabla \varphi|^2 \dd A (z).\]
\begin{thm}[Cutting {\cite[Thm.\,1.1]{VW1}}]\label{thm:welding_coupling1}
    Suppose $\gamma$ is a Jordan curve through $\infty$ and $\varphi \in  \mc E(\mathbb{C})$. Then we have the identity:
    \begin{equation}\label{jan26.1}
    \mc D_\mathbb{C}(\varphi) + I^L(\gamma)  = \mc D_{\mathbb{H}}(u) + \mc D_{\mathbb{H}^*}(v) ,\end{equation}
    where  
    \begin{equation}\label{eq:ppS}
     u =  \varphi \circ f + \log \abs{f'}, \quad v =  \varphi \circ h + \log \abs{h'},
\end{equation}
  and $f$ and $h$ map conformally $\m H$ and $\m H^*$ onto, respectively, $H$ and $H^*$, the two components of $\mathbb{C} \smallsetminus \gamma$, while fixing $\infty$. 
\end{thm}

A function $\varphi\in \mc E (\m C)$ has vanishing mean oscillation. 
 In fact, the following version of Poincar\'e inequality (which can be obtained using a scaling argument) shows
if $D \subset \mathbb{C}$ is a disk or square and $ \varphi \in W^{1,2}(D)$, then $$\int_{D}|\varphi-\varphi_{D}|^2 \,\dd A \le 4 \, (\diam \, D)^2 \int_{D}|\nabla \varphi|^2 \,\dd A.$$ Here we use the notation
 \[ \varphi_D = \frac{1}{|D|}\int_D \varphi \, \dd A
 \]
 for the average of $\varphi \in L^1(D)$ over $D$ and we write $|D|$ for the Lebesgue measure of $D$. Therefore,  using Cauchy-Schwarz inequality, we have
 $$\frac{1}{|D|} \int_{D} |\varphi - \varphi_D| \,\dd A \le \sqrt {\frac{1}{|D|} \int_{D} |\varphi - \varphi_D|^2 \,\dd A} \le \frac{2}{\sqrt \pi} \sqrt{\int_{D}|\nabla \varphi|^2 \,\dd A} $$
 which converges to $0$ as $\diam (D) \to 0$.

The John-Nirenberg inequality (see, \cite[Lem.\,1]{JohnNirenberg} or \cite[Thm.\,VI.2.1]{garnett}) shows that $e^{2\varphi}$ is locally integrable. In other words,
$e^{2 \varphi} \dd A$ defines a $\sigma$-finite measure supported on $\m{C}$, absolutely continuous with respect to Lebesgue measure $\dd A$. The transformation law \eqref{eq:ppS} is chosen such that $e^{2u} \dd A $ and $e^{2v} \dd A $ are the pullback measures by $f$ and $h$ of $e^{2 \varphi} \dd A $, respectively.

Let us first explain why we consider this theorem as a finite energy analog of an SLE/GFF coupling. Note that we do not make rigorous statement here and only argue heuristically. 
The first coupling result we refer to is the quantum zipper theorem, which couples $\SLE_\k$ curves with \emph{quantum surfaces} via a cutting operation and as welding curves 
   \cite{Quantum_zipper,MoT}. A quantum surface is a domain equipped with a
  Liouville quantum gravity ($\sqrt \k$-LQG) measure, defined using a regularization of $e^{\sqrt \k \Phi}\dd A (z)$, where $\sqrt \k \in (0,2)$, and $\Phi$ is a Gaussian field with the covariance of a free boundary GFF\footnote{In fact, $\Phi$ is a  free boundary GFF  plus a logarithmic singularity $- \sqrt \k \log |\cdot|$. Note that the factor of the logarithmic singularity converges to $0$ as $\k \to 0$.}. 
The analogy is outlined in the table below. In the left column we list concepts from random conformal geometry and in the right column the corresponding finite energy objects.

\renewcommand{\arraystretch}{1.06}

\begin{center} 
  \setlength{\LTpre}{12pt}
  \setlength{\LTpost}{-2pt}
\begin{longtable}{ l  l  }
 \toprule
    \multicolumn{1}{c}{\bf  SLE/GFF with $\k \to 0+$}  &   \multicolumn{1}{c}{\bf  Finite energy} \\ \hline 
     $\SLE_\kappa$ loop    & Jordan curve $\gamma$ with $I^L(\g) <\infty$ \\
     & i.e.,  a Weil-Petersson quasicircle $\g$\\
      \hline
        Free boundary GFF $\sqrt \k \Phi$ on $\mathbb{H}$ (on $\m C$) &  $2u \in \mathcal{E}(\m H)$  ($2\varphi \in \mathcal{E}(\m C)$)   \\ \hline
              $\sqrt \k$-LQG on quantum plane $\approx e^{\sqrt \k \Phi} \dd A$ & $e^{2 \varphi} \,\dd A, \, \varphi \in \mathcal{E}(\m C)$ \\ \hline

  $\sqrt \k$-LQG on quantum half-plane on $\mathbb{H}$ & $e^{2 u} \,\dd A ,  \, u \in \mathcal{E}(\m H)$
  \\ \hline
    $\sqrt \k$-LQG boundary measure $\approx e^{\sqrt \k \Phi/2} \dd x$  & $e^{u}\,\dd x, \, u \in H^{1/2}(\mathbb{R})$ \\
    \hline
  $\SLE_\kappa$ cuts an independent quantum & A Weil-Petersson quasicircle $\gamma$ cuts \\
      plane $e^{\sqrt \k \Phi} \dd A (z)$ into  independent & $\varphi \in \mathcal{E}(\mathbb{C})$ into $u \in \mc E(\m H), v \in \mc E(\m H^*)$, \\
     quantum half-planes $e^{\sqrt \k \Phi_1}, e^{\sqrt \k \Phi_2}$ &  $I^L(\gamma) + \mc{D}_{\m{C}}(\varphi) = \mc{D}_{\m{H}}(u) + \mc{D}_{\m{H}^*}(v)$\\
  \bottomrule 
\end{longtable}
\end{center}

 To justify the analogy of the last line, we argue heuristically as follows.  From the left-hand side, one expects that under an appropriate choice of topology and for small $\k$,
 \begin{align}\label{eq:heuristic_cutting}
 \begin{split}
 `` & \P ( \SLE_\k  \text{ loop stays close to } \gamma,\, \sqrt{\k}\Phi \text{ stays close to } 2 \varphi) \\
 = &\;\P(\sqrt \k \Phi_1 \text{ stays close to } 2 u, \, \sqrt{\k} \Phi_2 \text{ stays close to } 2 v)".    
 \end{split}
 \end{align}
 
 From the large deviation principle and the independence between SLE and $\Phi$, we obtain similarly as \eqref{eq:heuristic_LDP}
 \begin{align*}
   `` &\;  \lim_{\k \to 0} - \k \log  \P(\SLE_\k  \text{ stays close to } \gamma,\, \sqrt{\k} \Phi \text{ stays close to } 2 \varphi) \\ 
     =  & \; \lim_{\k \to 0} - \k  \log  \P(\SLE_\k \text{ stays close to } \gamma) +\lim_{\k \to 0}- \k  \log  \P(\sqrt{\k} \Phi \text{ stays close to } 2 \varphi) \\
     = & \; I^L (\gamma) + \mc D_{\m C} (\varphi)".
 \end{align*}
 
 On the other hand the independence between $\Phi_1$ and $\Phi_2$ gives
 \begin{align*}
`` & \lim_{\k \to 0} -\k \log  \P(\sqrt{\k} \Phi_1  \text{ stays close to } 2 u,\, \sqrt{\k} \Phi_2 \text{ stays close to }  2 v)\\
 & =   \mc D_{\m H} (u) + \mc D_{\m H^*} (v)".
 \end{align*}
 We obtain the identity \eqref{jan26.1} using \eqref{eq:heuristic_cutting} heuristically.  
 
 We now present our short proof of Theorem~\ref{thm:welding_coupling1} in the case where $\gamma$ is smooth and $\varphi \in C_c^{\infty} (\m C)$ to illustrate the idea. The general case follows from an approximation argument, see \cite{VW1} for the complete proof. 
 
\begin{proof}[Proof of Theorem~\ref{thm:welding_coupling1} in the smooth case]
 From Remark~\ref{rem_infinite_curve_LE}, if $ \infty 
 \in \gamma$, then
 \begin{equation*}
 I^L(\gamma) = \frac{1}{\pi} \int_{\HH} \left|\nabla \s_f \right|^2 \, \dd A (z) + \frac{1}{\pi} \int_{\HH^*} \left|\nabla \s_h \right|^2 \, \dd A (z),
 \end{equation*}
 where $\s_f$ and $\s_h$ are the shorthand notation for $\log |f'|$ and $\log |h'|$.
 The conformal invariance of Dirichlet energy gives
 $$\mc D_{\m H} (\varphi\circ f) + \mc D_{\m H^*} (\varphi\circ h)  =  \mc D_{H} (\varphi) + \mc D_{H^*} (\varphi) =  \mc D_{\m C} (\varphi). $$
     To show \eqref{jan26.1}, after expanding the Dirichlet energy terms, it suffices to verify the cross terms
     vanish:
      \begin{equation}\label{eq:cross_term}
      \int_{\m H} \brac{\nabla \s_f(z), \nabla (\varphi\circ f) (z)}\dd A (z) + \int_{\m H^*} \brac{\nabla \s_h(z), \nabla (\varphi \circ h) (z)}\dd A (z) = 0.
      \end{equation} 
    Indeed, by Stokes' formula, the first term on the left-hand side 
    equals  
\begin{align*}
 \int_{\m R} \partial_n \s_f(x) \varphi (f (x))\dd x & = \int_{\m R} k(f(x)) \abs{f'(x)} \varphi (f(x))\dd x
 = \int_{\partial H} k(z) \varphi(z) \abs{\dd z}
\end{align*}    
where $k(z)$ is the geodesic curvature of $ \gamma = \partial H$ at $z$ using the identity
 $\partial_{n} \s_f (x) = |f'(x)|k (f(x))$  (this follows from an elementary differential geometry computation, see, e.g., \cite[Appx.\,A]{W2}). 
The geodesic curvature at the same point $z \in \gamma$, considered as a point of $\partial H^*$, equals $-k(z)$. Therefore the sum in \eqref{eq:cross_term} cancels out and completes the proof in the smooth case. 
 \end{proof}
 
The following result is on the converse operation of the cutting, which shows that we can also recover $\g$ and $\varphi$ from $u$ and $v$ by \emph{conformal welding}.
More precisely, an increasing homeomorphism $\weld:  \m R \to  \m R$ is said to be a (conformal) welding homeomorphism of a Jordan curve $\g$ through $\infty$, if there are conformal maps $f,h$ of the upper and lower half-planes onto the two components of $\m C \smallsetminus \g$, respectively, such that $\weld = h^{-1} \circ f|_{\m R}$.
 In general, for a given homeomorphism $\weld$, there might not exist a triple $(\g, f,h)$ which solves the \emph{welding problem}. Even when the solution exists, it might not be unique.

  However, we now construct a welding homeomorphism starting from $u \in \mc E(\m H)$ and $v \in \mc E(\m H^*) $ and show that there exists a unique normalized solution to the welding problem, and the curve obtained is Weil-Petersson. For this, we recall that
the {\em trace} of a generalized function in a Sobolev space of a domain $D$ is the boundary value of the function on $\partial D$. It is defined through a {\em trace operator} extending the restriction map for smooth functions. More precisely, for $u \in \mc E (\m H) \cap C^{\infty} (\ad {\m H})$, we define 
$\mc R [u] := u|_{\m R}.$
We have
$$\norm{ \mc R [u]}_{H^{1/2} (\m R)}^2 = \mc D_{\m H} (\mc P[u]) \le \mc D_{\m H} (u),$$
where $\mc P[u]$ is the Poisson integral of $u$, the equality follows from Douglas formula \eqref{eq:douglas_R}, and the inequality follows from the Dirichlet principle.
Therefore, $\mc R$ extends to a bounded operator $\mc E (\m H) \to H^{1/2}(\m R)$ using the density of smooth functions in $\mc E (\m H)$.

 There is a more concrete way to describe the trace of $u \in \mc E (\m H)$ following Jonsson and Wallin \cite{JW1984} using averages over balls  as follows. We remark that this definition even generalizes to rougher domains bounded by chord-arc curves.

 Suppose $\tilde u \in \mc E(\m C)$ and $\g$ is a chord-arc curve in $\hat{\mathbb{C}}$. The {\em Jonsson-Wallin trace} of $u$ on $\g$ is defined for arclength a.e. $z \in \g$ by the following limit of averages
\begin{equation}\label{def:trace}
\mc R_{\g}[\tilde u](z):=\lim_{r \to 0+}\tilde u_{B(z,r)}, 
\end{equation}
where $B(z,r) = \{w: |w-z| < r\}$. Let $\O$ be a domain bounded by a chord-arc curve $\g$ and $u \in \mc E (\O)$,
the trace of $u$ on a $\g = \partial \O$ is
\begin{equation*}
\mc R_{\O \to \g}[u](z) := \mc R_{ \g}[\tilde u] (z) ,\quad \text{ for arclength a.e. }z\in \g,
\end{equation*}
 where  $\tilde u \in \mc E (\m C)$ is any function such that $\tilde u|_{\O} =u$. In particular, the definition $\mc R_{\O \to \g}[u]$
  does not depend on the choice of the extension $\tilde u \in \mc E (\m C)$.  Moreover, 
  \begin{equation} \label{eq:trace_H1/2}
      \mc R_{\Omega \to \g}[u] \in H^{1/2}(\g).
  \end{equation} 
  We refer to \cite[Appx.\,A]{VW1} for more details.

With a slight abuse of notation, we write $u \in H^{1/2} (\m R)$ also for the trace of $u \in \mc E(\m H)$ on $\m R$.
As $H^{1/2} (\m R) \subset \operatorname{VMO}(\m R)$ (see Equation~\eqref{eq:H1/2_VMO}), John-Nirenberg inequality implies that $\dd \mu : = e^{u} \dd x$ is a $\sigma$-finite measure supported on $\m R$. 
  \begin{lem}[\!\!{\cite[Lem.\,2.4]{VW1}}] \label{lem:infinite_measure}
  Suppose $u \in H^{1/2}(\mathbb{R})$ and $d \mu = e^{u} dx$. 
      Then $\mu(I) = \infty$ for any unbounded interval $I$.
   \end{lem}
   Let $v \in \mc E (\m H^*)$, we set similarly $\dd \nu = e^{v} \dd x$. 
  We define an increasing homeomorphism $\weld$ by $\weld(0)=0$ and then
\begin{equation}\label{def:homeo1}
\weld(x)=
\begin{cases}
      \inf\left\{ y\ge 0: \mu[0,x] = \nu[0,y]\right\} & \text{ if } x > 0;\\
    -\inf\left\{ y\ge 0: \mu[x,0] = \nu[-y,0]\right\} & \text{ if } x <0.
   \end{cases} \end{equation}
From Lemma~\ref{lem:infinite_measure},
$\weld$ is well-defined and $\mu([a,b]) = \nu(\weld([a,b]))$ for any choice of $a \le b$.   We say $\weld$ is the \emph{isometric welding homeomorphism} associated with $\mu$ and $\nu$.

\begin{thm}[Isometric conformal welding {\cite[Thm.\,1.2]{VW1}}]\label{thm:tuple12}
Let $u \in \mc E(\mathbb{H})$ and  $v \in \mc E(\mathbb{H}^*)$. The isometric welding problem for the measures $e^{u} \dd x$ and $e^{v} \dd x$ on $\m R$ has a solution $(\gamma, f, h)$ and the welding curve $\gamma$ is a Weil-Petersson quasicircle. Moreover, there exists a unique
$\varphi \in \mc E(\mathbb{C})$ such that \eqref{eq:ppS} is satisfied.
 \end{thm}
\begin{proof}[Proof sketch]
We first prove that the increasing isometry $\weld : \m R \to \m R$ obtained from  the measures $e^{u} \dd x$ and $e^{v} \dd x$ satisfies $\log \weld' \in H^{1/2}$, which is equivalent to $\weld$ being the welding homeomorphism of a Weil-Petersson quasicircle \cite{Shen-Tang,Shen-Tang-Wu} and shows the existence of solution $(\g, f, h)$. 
We then check that the function $\varphi$ defined a priori in $\m C \smallsetminus \g$ from $u, v$ and the transformation law \eqref{eq:ppS}, extends to a function in $\mc E (\m C)$. See \cite[Sec.\,3.1]{VW1} for the details.
\end{proof}
\begin{remark}
 The solution $(\gamma, f, h)$ in Theorem~\ref{thm:tuple12} is unique if appropriately normalized since quasicircles are conformally removable. Theorem~\ref{thm:welding_coupling1} then shows an upper bound of the welded curve's Loewner energy:
    \begin{equation} \label{eq:dissipation}
        I^L(\gamma)  = \mc D_{\mathbb{H}}(u) + \mc D_{\mathbb{H}^*}(v) - \mc D_\mathbb{C}(\varphi) \le \mc D_{\mathbb{H}}(u) + \mc D_{\mathbb{H}^*}(v).
    \end{equation} 
\end{remark}

\subsection{Flow-line identity}\label{sec:flowline}
Now let us turn to the second identity between Loewner energy and Dirichlet energy.  The idea is very simple: since the Dirichlet energy of a harmonic function equals that of its harmonic conjugate,  \eqref{eq:loewner_energy_infinite} can be written as
\begin{equation}\label{eq:loewner_winding}
I^L(\gamma) = \frac{1}{\pi} \int_{\HH} \left|\nabla \arg f' \right|^2 \, \dd A (z) + \frac{1}{\pi} \int_{\HH^*} \left|\nabla \arg h' \right|^2 \, \dd A (z).
\end{equation}
We will interpret this identity as a flow-line identity.

More precisely, let $\g$ be a Weil-Petersson quasicircle passing through $\infty$. Since $\g$ is asymptotically smooth (Theorem~\ref{thm:asymptotically_smooth}), we can parametrize it by arclength $s\mapsto \g(s)$.
  By Theorem~II.4.2 of \cite{GM}\footnote{ In \cite{GM} the result is stated for conformal maps $f$ defined on $\m D$. However, as the existence of $\g'$ and the limit of $\arg f'$ are local property, the result also applies to $\m H$.}, for almost every $\zeta = \g(s)$ such that $\g'(s)$ exists, 
   $\arg \frac{f(z) - \zeta}{z - f^{-1} (\zeta)}$ has a non-restricted
  limit as $z$ approaches $f^{-1} (\zeta)$ in $\m H$ (which also coincides with the non-tangential limit of $\arg f'$) and we denote this limit by $\tau(\zeta)$. Moreover,
\begin{equation}\label{eq:tau_def}
 \g'(s) = \lim_{t\to s} \frac{\g(t) - \zeta}{t-s} = \pm \lim_{t \to s \pm} \frac{\g(t) -\zeta}{|\g(t) - \zeta| } = e^{\ii \tau (\zeta)}.
 \end{equation}
 The second equality uses the fact that $\g$ is asymptotically smooth. Identity \eqref{eq:tau_def} shows that $\tau$ can be interpreted as the ``winding'' of $\g$.  We note that the harmonic measure and the arclength measure are mutually absolutely continuous on $\g$ (see, e.g., \cite[Thm.\,VII.4.3]{GM}), $\arg f'$ has limits almost everywhere on $\m R$ and coincides with the Jonsson-Wallin trace $\mc R_{\m H \to \m R}[\arg f']$. Therefore without ambiguity, we write simply the trace as $\arg f'|_{\m R}$.
 
 Since $\arg f'$ is harmonic in $\m H$, the following lemma is not surprising.
 \begin{lem}[\!\!{\cite[Lem.\,3.9]{VW1}}]\label{lem:arg_equals_tau}
 Suppose $\g$ is a Weil-Petersson curve through $\infty$. Then,
 \[
\arg f' (z) = \mc P[\tau] \circ f(z), \quad \forall z \in  {\m H},
 \]
 where $\mc P[\tau]$ is the Poisson extension of $\tau$ to $\m C \smallsetminus \g$. 
 \end{lem}
 Identity \eqref{eq:loewner_winding} implies that $\arg f' \in \mc E(\m H)$, we have $ \tau \circ f  = \arg f'|_{\m R} \in H^{1/2}(\m R)$  from \eqref{eq:trace_H1/2}. 
  
  \begin{lem}\label{lem:arg_h_2pi}
  With the same assumptions and notations as above, there exist a continuous branch of $\arg (\cdot )$ such that $\arg h' = \mc P[\tau] \circ h$ in $\m H^*$.
  \end{lem}
  
  \begin{proof}
  The welding homeomorphism $\weld : = h^{-1} \circ f|_{\m R}$ is a quasisymmetric homeomorphism (and so is also $\weld^{-1}$) since $\g$ is a quasicircle.
  Using $\tau \circ h = \tau \circ f \circ \weld^{-1}$ and the fact that the composition of an
$H^{1/2}(\m R)$ function with a quasisymmetric homeomorphism is still in $H^{1/2}(\m R)$
 (see \cite[Section 3]{NS95} for a proof in the setting of the unit circle and the proof for the line is exactly the same), we obtain that $\tau\circ h \in H^{1/2}(\m R)$.

  Since $\arg h'$ defined using a continuous branch of $\arg(\cdot)$ is also harmonic and has finite Dirichlet energy, the difference $v := \arg h'|_{\m R} - \tau \circ h$ is in $H^{1/2}(\m R)$. On the other hand, $v$ takes value in $2\pi \m Z$.  We conclude with the following lemma which shows that VMO functions behave like continuous functions.
  \end{proof}
  
  \begin{lem}\label{lem:VMO_constant}
  If a function $v \in \VMO (\m R)$ takes only integer values, then $v$ is constant.
  \end{lem}
  This lemma follows immediately from a more general result \cite[Thm.\,1]{BBM_15}. However, we provide an elementary proof in this much simpler case.
  \begin{proof}
  We write $\norm{v}_{*, I}$ for the BMO norm of $v$ on an interval $I$, defined by
  $$\norm{v}_{*,I} := \sup_{J \subset I} \frac{1}{|J|} \int_{J} \abs{u(x) - u_{J}} \dd x.$$
  We have $v \in \VMO (\m R)$ if and only if $\sup_{I \subset \m R, \,|I| < \vare} \norm{v}_{*,I} \xrightarrow[]{\vare \to 0} 0. $
  Since $v$ only takes values in $\m Z$, for any interval $I$ such that $\norm{v}_{*,I} < 1/6$, there exists a unique $n_I \in \m Z$ such that 
  \begin{equation}\label{eq:integer_mean}
      |v_I - n_I| \le  \norm{v}_{*,I} \quad \text{and} \quad |\{x \in I \,|\, v (x) \neq n_I\}|/|I| < 2\norm{v}_{*,I} < 1/3.
  \end{equation}
  For any small number $0 < \d < 1/6$, take $\vare > 0$ such that $\sup_{I \subset \m R, \,|I| < \vare} \norm{v}_{*,I} < \d$. The map $x\mapsto n_{[x, x + \vare]}$ is constant by \eqref{eq:integer_mean}. We call this value $n_{\vare}$. By subdividing further, we see that $n_{\vare'}$ does not depend on $\vare'$ when $\vare' < \vare$. We write $n_{\vare} = n$.
  Let $\d' <\d$ and subdividing $[x_0, x_0+ \vare]$ into smaller disjoint intervals $J_1 \bigsqcup \cdots \bigsqcup J_k$ such that $\norm {v} _{*,J_i} < \d'$ for all $i$, \eqref{eq:integer_mean} implies
  $$|\{x \in J_i \,|\, v (x) \neq n\}|/|J_i| < 2 \d'.$$
  Let $\d' \to 0$, we obtain that $|\{x \in [x_0, x_0 +\vare] \,|\, v (x) \neq n\}| = 0$. It implies that $v$ equals to $n$ almost everywhere on $[x_0, x_0 +\vare]$, hence on $\m R$.  
  \end{proof}

\begin{thm}[Flow-line identity {\cite[Thm.\,3.10]{VW1}}] \label{thm:flow-line}
If $\g$ is a Weil-Petersson  curve through $\infty$, we have the identity 
\begin{equation}\label{eq_flow_identity}
I^L(\g) = \mc D_{\m C} (\mc P[\tau]).
\end{equation}
Conversely, if $\varphi \in \mc E(\m C)$ is continuous and $\lim_{z \to  \infty} \varphi(z)$ exists and is finite, then for all $z_0 \in \m C$, any solution to the differential equation
\begin{equation}\label{eq:def_flowline}
\dot{\g}(t) = \exp\left(\ii \varphi (\g(t))\right),\quad  t\in (-\infty, \infty) \quad \text{and} \quad \g(0) = z_0
\end{equation}
is a $C^1$ Weil-Petersson curve through $\infty$. Moreover,
\begin{equation}\label{eq:flowline}
 \mc D_{\m C} (\varphi) = I^L(\g) + \mc D_{\m C} (\varphi_0),   
\end{equation}
where $ \varphi_0=\varphi- \mc P [\varphi|_{\g}]$ has zero trace on $\g$.
\end{thm}

The identity \eqref{eq_flow_identity} is simply a rewriting of \eqref{eq:loewner_winding}.
A solution to \eqref{eq:def_flowline} is called a \emph{flow-line} of the winding field $\varphi$ passing through $z_0$.
Here, we put a stronger condition by assuming $\varphi$ is continuous and admits a limit in $\m R$ as $z \to \infty$ (in other words, $\varphi \in \mc E(\m C) \cap C^0(\Chat)$). This condition allows us to use Cauchy-Peano theorem to show the  existence of the flow-line. However, we cautiously note that the solution to \eqref{eq:def_flowline} may not be unique. 
    The orthogonal decomposition of $\varphi$ for the Dirichlet inner product gives $\mc D_{\m C}(\varphi) = \mc D_{\m C}(\mc P [\varphi|_{\gamma}]) + \mc D_{\m C}(\varphi_0)$. Using \eqref{eq_flow_identity} and the observation that for all $z \in \g$, $\varphi (z) = \tau (z)$,  we obtain \eqref{eq:flowline}.
    
\begin{remark}
The additional assumption of $\varphi \in C^0(\Chat)$ is for technical reason to consider the flow-line of $e^{\ii \varphi}$ in the classical differential equation sense. One may drop this assumption by defining a flow-line to be a chord-arc curve $\g$ passing through $\infty$ on which $\varphi = \tau$ arclength almost everywhere. We will further explore these ideas in a setting adapted to bounded curves (see Theorem~\ref{thm:complex_id_bounded}).
\end{remark}

This identity is analogous to the \emph{flow-line coupling} between SLE and GFF, of critical importance, e.g., in the imaginary geometry framework of Miller-Sheffield \cite{IG1}: very loosely speaking, an $\SLE_\k$ curve is coupled with a GFF $\Phi$ and may be thought of as a flow-line of the vector field $e^{\ii \Phi/\chi}$, where $\chi = 2/\g - \g/2$. As $\g \to 0$, we have  $e^{\ii \Phi/\chi} \sim e^{\ii \g \Phi/2}$.

Let us finally remark that by combining the cutting-welding \eqref{jan26.1} and flow-line \eqref{eq:flowline} identities, we obtain the following complex identity. See also Theorem~\ref{thm:complex_id_bounded} the complex identity for a bounded Jordan curve.

\begin{cor}[Complex identity {\cite[Cor.\,1.6]{VW1}}]\label{cor:complex_field}
Let $\psi$ be a complex-valued function on $\m C$ with finite Dirichlet energy and whose imaginary part is continuous in $\Chat$.
   Let $\gamma$ be a flow-line of the vector field $e^{\psi}$.
  Then we have
  $$ \mc D_{\m C}(\psi) = \mc D_{\m H}(\zeta) + \mc D_{\m H^*}(\xi),$$
  where $\zeta = \psi \circ f + (\log f')^*$, $\xi = \psi \circ h + (\log h')^*$ and $z^*$ is the complex conjugate of $z$.
\end{cor}
\begin{remark}
  A flow-line $\g$ of the vector field $e^{\psi}$ is understood as a flow-line of $e^{\ii \Im \psi}$, as the real part of $\psi$ only contributes to a reparametrization of $\g$.
\end{remark}

\begin{proof}
From the identity   $\arg f'= \mc P [\Im \psi] \circ f$, we have 
   \begin{align*}
       \zeta  &= \left( \Re \psi \circ f + \log |f'|\right) + \ii  \left( \Im \psi \circ f - \arg f'\right) = u + \ii  \Im \psi_0 \circ f; \\
       \xi &= v + \ii \Im \psi_0 \circ h,
   \end{align*}
   where $u : = \Re \psi \circ f + \log |f'| \in \mc E (\m H)$, $v : = \Re \psi \circ h + \log |h'| \in \mc E (\m H^*)$ and $\psi_0 = \psi - \mc P[\psi|_{\gamma}]$.
   From the cutting-welding identity \eqref{jan26.1}, we have 
   $$\mc D_{\m C} (\Re \psi) + I^L(\gamma) = \mc D_{\m H} (u) + \mc D_{\m H^*} (v).$$
   On the other hand, the flow-line identity gives
   $\mc D_{\m C} (\Im \psi) = I^L(\gamma) + \mc D_{\m C} (\Im \psi_0).$
   Hence,
   \begin{align*}
       \mc D_{\m C} (\psi) &= \mc D_{\m C} (\Re \psi) +\mc D_{\m C} (\Im \psi) =  \mc D_{\m C} (\Re \psi) + I^L(\gamma) + \mc D_{\m C} (\Im \psi_0) \\
       &=   \mc D_{\m H} (u) + \mc D_{\m H^*} (v) + \mc D_{\m C} (\Im \psi_0)\\
       & =\mc D_{\m H}(\zeta) + \mc D_{\m H^*}(\xi)
   \end{align*}
   as claimed.
\end{proof}

\begin{remark}\label{rem:complex_identity_two_theorems}
From Corollary~\ref{cor:complex_field} we recover the flow-line identity (Theorem~\ref{thm:flow-line}) by taking $\Im \psi = \varphi$ and $\Re \psi = 0$. 
Similarly, the cutting-welding identity \eqref{jan26.1} follows from taking $\Re \psi = \varphi$ and $\Im \psi = \mc P [\tau]$ where $\tau$ is the winding function along the curve $\g$.  
Therefore, the complex identity is equivalent to the union of  cutting-welding and flow-line identities. 
\end{remark}

\subsection{Applications} \label{sec:application}
We now show that these identities between Loewner and Dirichlet energies have interesting consequences in geometric function theory. 

The cutting-welding identity has the following application.
  Suppose $\gamma_1,\gamma_2$ are locally rectifiable Jordan curves in $\Chat$ of the same length (possibly infinite if both curves pass through $\infty$) bounding two domains $\O_1$ and $\O_2$ and we mark a point on each curve. Let $\weld$ be an arclength isometry $\gamma_1 \to \gamma_2$ matching the marked points. We obtain a topological sphere from $\O_1 \cup \O_2$ by identifying the matched points. 
  Following Bishop \cite{bishop_isometric}, the arclength isometric welding problem is to find a Jordan curve $\gamma \subset \Chat$, and conformal mappings $f_1, f_2$ from $\O_1$ and $\O_2$ to the two connected components of $\Chat \smallsetminus \gamma$, such that $f_2^{-1} \circ f_1|_{\gamma_1} = \weld$.  
  The arclength welding problem is in general a hard question and have many pathological examples. For instance, the mere rectifiability of $\gamma_1$ and $\gamma_2$ does not guarantee the existence nor the uniqueness of $\gamma$, but the chord-arc property does. However, chord-arc curves are not closed under isometric conformal welding: the welded curve can have Hausdorff dimension arbitrarily close to $2$, see \cite{david_chord_arc, Semmes86, bishop_isometric}.
 Rather surprisingly,  Theorem~\ref{thm:welding_coupling1} and Theorem~\ref{thm:tuple12}
  imply that Weil-Petersson quasicircles are closed under arclength isometric welding. Moreover,
$I^L(\gamma) \le I^L(\gamma_1) + I^L(\gamma_2)$. 

We describe this result more precisely in the case when both $\g_1$ and $\g_2$ are Weil-Petersson quasicircles through $\infty$ (see \cite[Sec.\,3.2]{VW1} for the bounded curve case). Let $H_i, H_i^*$ be the connected components of $\m C \smallsetminus \g_i$.
\begin{cor} [\!\! {\cite[Cor.\,3.4]{VW1}}]\label{cor:arc-length}
   Let $\gamma$ (resp. $\tilde \gamma$) be the arclength isometric welding curve of the domains $H_1$ and $H_2^*$ (resp.  $H_2$ and $H_1^*$). 
  Then $\gamma$ and $\tilde \gamma$ are also Weil-Petersson quasicircles. 
  Moreover, 
  $$I^L(\gamma)  + I^L(\tilde \gamma) \le I^L(\gamma_1) + I^L(\gamma_2).$$
\end{cor}

\begin{proof} For $i=1,2$, let $f_i$ be a conformal equivalence $\m H \to H_i$, and $h_i:  \m H^* \to H_i^*$ both fixing $\infty$.  
By \eqref{eq:loewner_energy_infinite},
$$I^L(\gamma_i) = \mc D_{\m H}\left(\log |f_i'| \right) + \mc D_{\m H^*}\left(\log |h_i'|\right).$$

Set $u_i : = \log |f'_i|$, $v_i := \log |h'_i|$. Then $\gamma$ is the welding curve obtained from Theorem~\ref{thm:tuple12} with $u = u_1$, $v= v_2$ and $\tilde \gamma$ is the welding curve for $u = u_2$, $v= v_1$.
Then \eqref{eq:dissipation} implies
\begin{align*}
    I^L(\gamma) + I^L(\tilde \gamma)  \le \mc D_{\m H}\left(u_1\right) + \mc D_{\m H^*}\left(v_2\right) + 
\mc D_{\m H}\left(u_2\right) + \mc D_{\m H^*}\left(v_1\right)  =  I^L(\gamma_1) + I^L(\gamma_2)
\end{align*}
as claimed.
\end{proof}

The flow-line identity has the following consequence that we omit the proof.
When $\gamma$ is a bounded Weil-Petersson quasicircle (resp. Weil-Petersson quasicircle passing through $\infty$), we let $f$ be a conformal map from $\m D$ (resp. $\m H$) to one connected component of $ \m C \smallsetminus \gamma$. 

\begin{cor}[\!\!{\cite[Cor.\,1.5]{VW1}}]
Consider the family of analytic curves $\gamma_r := f(r \m T)$, where $0< r <1$  \textnormal{(}resp. $\gamma^{r} := f(\m R + \ii r)$, where $r > 0$\textnormal{)}.
   For all $0 < s < r <1$ \textnormal{(}resp. $0 < r < s$\textnormal{)}, we have
   $$ I^L(\gamma_s) \le I^L(\gamma_r) \le I^L(\gamma), \quad (\text{resp. } I^L(\gamma^s) \le I^L(\gamma^r) \le I^L(\gamma),)$$
   and equalities hold if only if $\gamma$ is a circle \textnormal{(}resp. $\gamma$ is a line\textnormal{)}.
   Moreover, $I^L(\gamma_r)$ \textnormal{(}resp. $I^L(\gamma^r)$\textnormal{)} is continuous in $r$ and
   \begin{align*}
   &  I^L(\gamma_r) \xrightarrow{r \to 1-} I^L(\gamma); \quad I^L(\gamma_r) \xrightarrow{r \to 0+} 0\\
   (\text{resp. } & I^L(\gamma^r) \xrightarrow{r \to 0+} I^L(\gamma); \quad I^L(\gamma^r) \xrightarrow{r \to \infty} 0).
   \end{align*}
\end{cor}

\begin{remark}
Both limits and the monotonicity are consistent with the fact that the Loewner energy measures the ``roundness'' of a Jordan curve. In particular, the vanishing of the energy of $\gamma_r$ as $r \to 0$ expresses the fact that conformal maps take infinitesimal circles to infinitesimal circles. 
\end{remark}

\section{Large deviations of multichordal SLE$_{0+}$} \label{sec:multi}

\subsection{Multichordal SLE}

We now consider the multichordal $\SLE_\k$, that are families of random curves (multichords) connecting pairwise distinct boundary points of a simply connected planar domain $\domain$. Constructions for multichordal SLEs have been obtained by many groups 
\cite{Car03, WW_toulouse_Girsanov, BBK, Dub_comm, KL07, Law09, IG1, IG2, BPW, PW19}, and models the interfaces in two-dimensional statistical mechanics models with alternating boundary condition.

As in the single-chord case, we include the marked boundary points to the domain data $(\domain; x_1, \ldots, x_{2\np})$, assuming that 
they appear in counterclockwise order along the boundary $\partial \domain$. 
The objects considered in this section are defined in a conformally invariant or covariant way. So without loss of generality, we assume that $\partial \domain$ is smooth in a neighborhood of the marked points. 
Due to the planarity, there exist $\Catalan_\np$ different possible pairwise non-crossing connections for the curves, where
\begin{align} \label{eq:catalan}
\Catalan_\np = \frac{1}{\np+1} \binom{2\np}{\np}
\end{align}
is the $\np$:th Catalan number.
We enumerate them in terms of \emph{$\np$-link patterns} 
\begin{align} \label{eq: alpha}
\a = \{ \link{a_1}{b_1}, \link{a_2}{b_2}, \ldots, \link{a_\np}{b_\np} \},
\end{align}
that is, partitions of $\{1,2,\ldots,2\np\}$ giving a non-crossing pairing of the marked points. 
Now, for each $\np \geq 1$ and $\np$-link pattern $\a$,
we let $\mc X_{\a}(\domain;  x_1, \ldots,  x_{2\np})  \subset \prod_j \mc X(\domain; x_{a_j}, x_{b_j})$ 
denote the set of \emph{multichords} $\ad \g = (\g_1, \ldots, \g_\np)$
consisting of pairwise disjoint chords where $\g_j \in \mc X(\domain; x_{a_j}, x_{b_j})$  for each $j\in\{1,\ldots,\np\}$. We endow  $\mc X_{\a}(\domain;  x_1, \ldots,  x_{2\np}) $ with the relative product topology and recall that $\mc X(\domain; x_{a_j}, x_{b_j})$ is endowed with the topology induced from a Hausdorff metric defined in Section~\ref{sec:LDP_single}.
\emph{Multichordal $\SLE_\k$} is a random multichord $\ad \g = (\g_1, \ldots, \g_\np)$ in $\mc X_\a(\domain;x_1, \ldots,  x_{2\np})$, characterized in two equivalent ways, when $\k >0$.

{\bf By re-sampling property:} From the statistical mechanics model viewpoint, the natural definition of multichordal SLE is such that for each $j$,  
the chord $\g_j$ has the same law as the trace of a chordal $\SLE_\k$ in $(\hat \domain_j; x_{a_j}, x_{b_j})$, conditioned on the other curves $\{\g_i \; | \; i \neq j \}$. Here, $\hat \domain_j$ is the component of $\domain \smallsetminus \bigcup_{i \neq  j} \g_i$
containing $\g_j$, highlighted in grey in Figure~\ref{fig:NSLE}. 
In~\cite{BPW}, the authors
proved that when $\k \in (0,4]$, the multichordal $\SLE_\k$ is
the unique stationary measure of a Markov chain
on $\mc X_{\a}(\domain;x_1,\ldots,x_{2\np})$ 
defined by re-sampling the curves from their conditional laws. This idea was already introduced and used earlier in~\cite{IG1, IG2},
where Miller \& Sheffield studied interacting SLE curves 
coupled with the Gaussian free field (GFF) in the framework of the so-called imaginary geometry.

\begin{figure}[ht]
\centering
\includegraphics[width=0.35\textwidth]{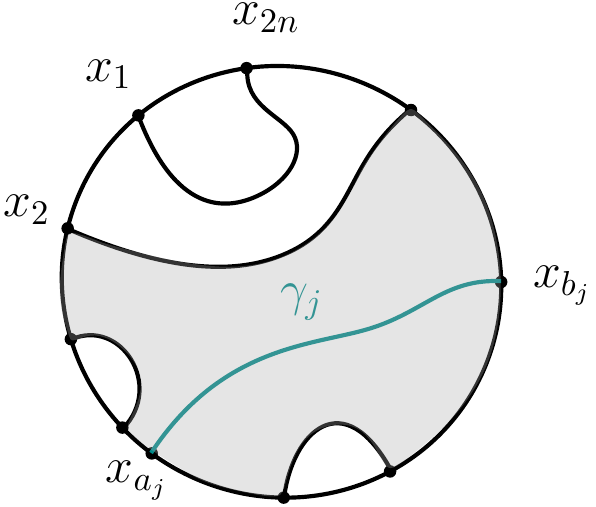}
\caption{Illustration of a multichord and the domain $\hat \domain_j$ containing $\g_j$.
\label{fig:NSLE} }
\end{figure}

{\bf By Radon-Nikodym derivative:} We assume\footnote{The same result holds for $8/3\le \k \le 4$, when $c(\k) \ge 0$, if one includes into the exponent in \eqref{eq:RN} the indicator function of the event that all $\g_j$ are pairwise disjoint.} that $0<\k < 8/3$.
Multichordal $\SLE_\k$ in $\mc X_{\a}(\domain;x_1,\ldots,x_{2\np})$  can be obtained by reweighting $n$ independent $\SLE_\k$ (of the same domain data and link pattern) by
    \begin{equation}\label{eq:RN}
       \frac{1}{\mc Z}   \exp \left( \frac{c(\k)}{2} m_D (\g_1, \ldots, \g_n) \right),
\text{    where }
    c (\k): =  \frac{(3\k-8)(6-\k)}{2\k} < 0
    \end{equation}
    is the central charge associated to $\SLE_\k$. The quantity $m_\domain(\ad \g)$ is defined using  the \emph{Brownian loop measure} $\mu^{loop}_\domain$ introduced by Lawler, Schramm, and Werner \cite{LSW_CR_chordal,LW2004loupsoup}:
    \begin{align}\label{eq:loop_formula}
    \begin{split}
m_\domain(\ad \g) & := 
 \sum_{p=2}^\np 
 \mu_\domain^{\mathrm{loop}} \Big( \big\{ \ell \; \big| \; \ell \cap \g_i \neq \emptyset \text{ for at least $p$ chords } \g_i \big \} \Big) \\
 & =  
 \int \max\Big( \# \{\text{chords hit by } \ell\} - 1, \, 0\Big) \,\dd \mu_\domain^{\mathrm{loop}} (\ell)
    \end{split}
\end{align}
which is positive and finite whenever the family $(\g_i)_{i = 1 \ldots n}$ is disjoint.
In fact, the Brownian loop measure is an infinite measure on Brownian loops, which is conformally invariant, and for $D' \subset D$, $\mu^{loop}_{D'}$ is simply $\mu^{loop}_{D}$ restricted to loops contained in $D'$.
When $D$ has non-polar boundary, the divergence of total mass of $\mu^{loop}_{D}$ comes  only from the contribution of small loops. In particular, the summand $\{ \ell \; \big| \; \ell \cap \g_i \neq \emptyset \text{ for at least $p$ chords } \g_i \}$ is finite if $p \ge 2$ and the chords are disjoint. 
For $n$ independent chordal SLEs connecting $(x_1, \ldots, x_{2n})$, chords may intersect each other. However, in this case $m_D$ is infinite and the Radon-Nikodym derivative \eqref{eq:RN} vanishes since $c < 0$. 
We note that $m_\domain(\g) = 0$ if $\np = 1$ (which is expected since no reweighting is needed for the single SLE).

\begin{remark}
The central charge $c(\k)$ and Brownian loop measure  appear in the conformal restriction formula for a single SLE \cite{LSW_CR_chordal}, which compares the law of SLE trace under the change of the ambient domain. See also \cite[Prop.\,3.1]{KL07}.
It is therefore not surprising to see such terms in the Radon-Nikodym derivatives 
\eqref{eq:RN} of multichordal SLE from the re-sampling property. Indeed, the expression \eqref{eq:loop_formula} already appears in \cite{KL07} for  multichords with ``rainbow'' link pattern. We refer the readers to \cite[Thm.\,1.3]{PW19} for the case of multichords with general link patterns. Note that our expression looks different from \cite{PW19} but is simply a combinatorial rearrangement. The precise definition of Brownian loop measure is not important for the presentation here, so we choose to omit it from our discussion.
\end{remark}
\begin{remark}
Notice that when $\k = 0$, $c = -\infty$, the second characterization does not apply. We first show the existence and uniqueness of multichordal $\SLE_0$ using the first characterization by making links to rational functions. 
\end{remark}

\subsection{Real rational functions and Shapiro's conjecture}\label{subsec:rational}

From the re-sampling property, the multichordal $\SLE_0$ in $\mc X_\a(\domain; x_1, \ldots, x_{2\np})$ 
as a deterministic multichord $ \ad \g = (\g_1, \ldots, \g_\np)$ with the property that 
each $\g_j$ is the $\SLE_0$ curve in its own component $(\hat \domain_j; x_{a_j}, x_{b_j})$. 
In other words, each $\g_j$ is the \emph{hyperbolic geodesic} in $(\hat \domain_j; x_{a_j}, x_{b_j})$, see Remark~\ref{rem:geodesic}.
We call a multichord with this property a \emph{geodesic multichord}. Without loss of generality, we assume that $\domain = \m H$.

The existence of geodesic multichord for each $\a$ follows by characterizing them as minimizers of a lower semicontinuous Loewner energy which is the large deviation rate function of multichordal $\SLE_{0+}$, 
to be discussed in Section~\ref{subsec:LDP_multi}. 
Assuming the existence, the uniqueness is a consequence of the following algebraic result.

We first recall some terminology.
A \emph{rational function} is an analytic branched cover of $\Chat$ over $\Chat$, or equivalently, the ratio of two polynomials $P,Q \in \m C[X]$. 
A point $x_0 \in \Chat$ is a \emph{critical point} (equivalently, a branched point) of a rational function $h$
with \emph{index} $k \ge 2$ if 
\begin{align*}
h(x) = h(x_0) + C (x - x_0)^k  + O ((x - x_0)^{k+1})
\end{align*}
for some constant $C \neq 0$ in a local chart  of $\Chat$ around $x_0$. 
A point $y \in \Chat$ is a \emph{regular value} of $h$ if $y$ is not image of any critical point.
The \emph{degree} of $h$ is the number of preimages of any regular value. We call $h^{-1} (\m R \cup \{\infty\})$ the \emph{real locus} of $h$, and $h$ is a  \emph{real} rational function if $P$ and $Q$ can be chosen from $\m R[X]$, or equivalently, $h(\m R \cup \{\infty\}) \subset \m R \cup \{\infty\}$.

\begin{thm}[\!\!{\cite[Thm.\,1.2, Prop.\,4.1]{peltola_wang}}]
\label{thm:real_rational}
Let $\bar{\eta} \in \mc X_\a (\m H; x_1, \ldots, x_{2\np})$ be a geodesic multichord.
The union of $\bar{\eta}$, its complex conjugate $\ad \eta^*$, and $\m R \cup \{\infty\}$ is the real locus of a real rational function $h_\eta$  of degree $\np+1$ with critical points $\{x_1, \ldots, x_{2\np}\}$. The rational function is unique up to post-composition by $\PSL(2, \m R)$\footnote{The group 
$$\PSL(2,\m R) = \Big\{A = \begin{pmatrix}
a & b \\ c & d 
\end{pmatrix} :  a,b,c,d \in \m R, \, ad - bc = 1\Big\}_{/A \sim -A} $$
acts on $\m H$ by 
$A : z \mapsto \frac{az + b}{cz + d}$, a M\"obius transformation of $\m H$.} and by the map $\m H \to \m H^* : z \mapsto -z$.
\end{thm}

\begin{remark}
By the Riemann-Hurwitz formula  on Euler characteristics,
a rational function of degree $\np+1$ has $2\np$ distinct critical points 
if and only if they all have index two:
\begin{align*}
(\np+1) \chi (\Chat)  - 2\np (2-1) = 2\np + 2 - 2\np = 2 = \chi (\Chat).
\end{align*}
\end{remark}

We prove Theorem~\ref{thm:real_rational} by constructing the rational function associated to a geodesic multichord $\ad \eta$.
\begin{proof}
The complement $\m H \smallsetminus \ad \eta$ has $\np+1$ components that we call \emph{faces}. 
We pick an arbitrary face $F$ and consider a uniformizing conformal map $h_\eta$ from $F$ onto $\m H$.
Without loss of generality, we assume that $F$ is adjacent to $\eta_1$. We call   $F'$ the other face adjacent to $\eta_1$.
Since $\eta_1$ is a hyperbolic geodesic in $\hat \domain_{1}$, 
the map $h_\eta$ extends by reflection to a conformal map on $\hat \domain_{1}$.
In particular, this extension of $h_\eta$ maps $F'$ conformally onto $\m H^*$. 
By iterating the analytic continuation across all the chords $\eta_k$, we obtain a meromorphic function $h_\eta \colon \m H \to \Chat$. Furthermore, 
$h_\eta$ also extends to $\ad{\m H}$,  
and its restriction $h_\eta|_{\m R \cup \{\infty\}}$ 
takes values in $\m R \cup \{\infty\}$. Hence, 
Schwarz reflection allows us to extend $h_\eta$ to $\Chat$ by setting $h_\eta(z) := h_\eta(z^*)^*$ for all $z \in \m H^*$.

Now, it follows from the construction that $h_\eta$ is a real rational function of degree $\np+1$, as exactly $\np+1$ faces are mapped to $\m H$ and $\np+1$ faces to $\m H^*$. Moreover, $h_\eta^{-1} (\m R \cup \{\infty\})$ is precisely the union of $\ad \eta$, its complex conjugate $\ad \eta^*$, and $\m R \cup\{\infty\}$. Finally, another choice of the face $F$ we started with yields the same function up to post-composition by $\PSL(2,\m R)$ and $z \mapsto -z$.
This concludes the proof.
\end{proof}

To find out all the geodesic multichords connecting $\{x_1, \ldots, x_{2\np}\}$, it thus suffices to classify all the rational functions with this set of critical points. The following result is due to Goldberg.

\begin{thm}[\!\!\cite{Gol91}] Let $z_1, \ldots, z_{2n}$ be $2n$ distinct complex numbers. There are at most $C_n$ rational functions \textnormal{(}up to post-composition by $\PSL(2, \m C)$\footnote{Namely, by M\"obius transformations of $\Chat$.}\textnormal{)} of degree $n+1$ with critical points $z_1, \ldots, z_{2n}$.
\end{thm}

Assuming the existence of geodesic multichord in $\mc X_\a (\m H; x_1, \ldots, x_{2n})$ and observing that two rational functions constructed in Theorem~\ref{thm:real_rational} are $\PSL(2,\m C)$ equivalent if and only if they are equivalent under the action of the group generated by $\brac{\PSL(2,\m R), z\mapsto -z}$, we obtain:

\begin{cor}
\label{thm:main_unique_minimizing}
 There exists a unique geodesic multichord in $\mc X_\a (\domain; x_1, \ldots, x_{2\np})$ for each~$\a$. 
\end{cor}
The multichordal $\SLE_0$ is therefore well-defined. We also obtain a by-product of this result:

\begin{cor}[\!\!{\cite[Cor.\,1.3]{peltola_wang}}]
\label{cor:Catalan}
If all critical points of a rational function are real, then it is a real rational function up to post-composition by a M\"obius 
transformation of $\Chat$.
\end{cor}
This corollary is a special case of the Shapiro conjecture  concerning real solutions to enumerative geometric problems on Grassmannians, see~\cite{Sot00}. 
Eremenko and Gabrielov \cite{EG02} first proved this conjecture 
for the Grassmannian of $2$-planes, when the conjecture is equivalent to Corollary~\ref{cor:Catalan}. See also \cite{EG11} for another elementary proof.

\subsection{Large deviations of multichordal $\SLE_{0+}$}\label{subsec:LDP_multi}
We now introduce the Loewner potential and energy and discuss the large deviations of multichordal $\SLE_{0+}$. 
\begin{df}\label{df:potential}
Let $\ad \g : = (\g_1, \ldots, \g_n)$.
The \emph{Loewner potential} of  $\ad \g$ is given by
\begin{align}\label{eq:initialdef_Hn}
\mc H_\domain(\ad \g) := \; & \frac{1}{12} \sum_{j=1}^\np I_{\domain}(\g_j) 
+ m_\domain(\ad \g) - \frac{1}{4} \sum_{j=1}^\np \log P_{\domain; x_{a_j}, x_{b_j}},
\end{align}
where $I_D (\g_j) = I_{D; x_{a_j}, x_{b_j}}(\g_j)$ is the chordal Loewner energy of $\g_j$ (Definition~\ref{df:chordalLE}), $m_D(\ad \g)$ is defined in \eqref{eq:loop_formula}, and $P_{\domain;\bpt,\ept}$ is the
\emph{Poisson excursion kernel}:
\begin{align*}
P_{\domain;\bpt,\ept} := |\varphi'(\bpt)| |\varphi'(\ept)| P_{\m H;\varphi(\bpt),\varphi(\ept)} , \qquad
\text{and} \qquad P_{\m H;\bpt,\ept} := |\ept - \bpt|^{-2}.
\end{align*}
 Here $\varphi \colon \domain \to \m H$ is a conformal map such that $\varphi(x), \varphi(y) \in \m R$, and $\varphi'(x)$ and $\varphi'(y)$ are well-defined since we assumed that $\partial D$ is smooth in a neighborhood of $x$ and $y$.

We denote the \emph{minimal potential} by
\begin{align} \label{eq:initialdef_Hmin}
\Hmin^\a_\domain (x_1, \ldots,x_{2\np})
:= \inf_{\ad \g} \mc H_\domain (\ad \g) \in (-\infty, \infty),
\end{align}
with infimum taken over all multichords $\ad \g \in \mc X_\a(\domain;x_1, \ldots,x_{2\np})$.
\end{df}
\begin{remark}
When $n = 1$,
$$\mc H_{D} (\g) = \frac{1}{12} I_D(\g) - \frac{1}{4} \log P_{D;x,y}, \qquad \forall \g \in \mc X(D; x,y).$$
The infimum of $\mc H_D$ in $\mc X(D; x,y)$ is realized for the minimizer of $I_{D;x,y}$, which is the hyperbolic geodesic in $(D;x,y)$.
\end{remark}
One important property of the Loewner potential is that it satisfies the following cascade relation which follows from a conformal restriction formula for Loewner energy and the definition of $m_\domain (\ad \g)$.
\begin{lem}[\!\!{\cite[Lem.\,3.8, Cor.\,3.9]{peltola_wang}}]
\label{lem:H_cascade}
For each $j \in \{1,\ldots,\np\}$, we have
\begin{align} \label{eq:induction_H}
\mc H_\domain(\ad \g) 
= \; & \mc H_{\hat{\domain}_j} (\g_j) 
+ \mc H_\domain (\g_1, \ldots, \g_{j-1}, \g_{j+1} \ldots, \g_\np). 
\end{align}
In particular, any minimizer of $\mc H_\domain$ in $\mc X_\a (\domain;x_1, \ldots,  x_{2\np})$ is a geodesic multichord, and $\mc H_\domain (\ad \g) < \infty$ if and only if $\{\g_j\}$ are disjoint chords with $I_D(\g_j) < \infty$. 
\end{lem}

Using techniques from quasiconformal mappings and the fact that multichords with finite potential consist of quasichords by Theorem~\ref{thm:single_qchord}, the Loewner potential is shown to have the following properties.
\begin{prop}[\!\!{\cite[Prop.\,3.13]{peltola_wang}}] \label{prop: H_semicontinuous}
 The sub-level set 
$$\{\ad \g \in \mc X_\a(\domain; x_1, \ldots, x_{2\np}) \;|\; \mc H_D(\ad  \g) \le c\}$$ 
is compact for any $c \ge \Hmin^\a_\domain$. In particular, there exists a multichord  in $\mc X_\a(\domain; x_1, \ldots, x_{2\np})$ minimizing $\mc H_\domain$. 
\end{prop}

At this point, we know from Lemma~\ref{lem:H_cascade} that the infimum in \eqref{eq:initialdef_Hmin} is attained by a geodesic multichord in $\mc X_\a(\domain; x_1, \ldots, x_{2\np})$. This shows the existence of geodesic multichord and completes the proof of the uniqueness as in Corollary~\ref{thm:main_unique_minimizing}.

\begin{df} \label{df:multi_energy}
We define the \emph{multichordal Loewner energy} of $\ad \g$ as
\begin{align*}
I_\domain^\a (\ad \g) 
& := 12 ( \mc H_\domain (\ad \g) - \Hmin^\a_\domain (x_1, \ldots,x_{2\np}) )\\
& = \left(\sum_{j=1}^\np I_{\domain}(\g_j) 
+ 12 m_\domain(\ad \g)\right) - \inf_{\ad \g'}\left(\sum_{j=1}^\np I_{\domain}(\g'_j) 
+ 12 m_\domain(\ad \g')\right).
\end{align*}
\end{df}

\begin{thm}[\!\!{\cite[Thm.\,1.5]{peltola_wang}}]
\label{thm:main_multi_LDP}
The family of laws $(\m P_\a^\k)_{\k > 0}$ of multichordal $\SLE_\k$ 
satisfies the large deviation principle in $\mc X_\a (\domain;x_1,\ldots,x_{2\np})$ with good rate function $I^\a_\domain$.
\end{thm}
\begin{remark}
When $ \np = 1$, Theorem~\ref{thm:main_multi_LDP} is equivalent to  Theorem~\ref{thm:main_single_LDP}.
\end{remark}

\begin{remark}\label{rem:expectation_asymptotics}
The expression of the rate function can be guessed from the Radon-Nikodym derivative \eqref{eq:RN}. In fact, we write heuristically the density of a single SLE as $\exp(-I_D (\g)/\k)$ for small $\k$ from Theorem~\ref{thm:main_single_LDP}. Taking the expectation $\m E_\k^{ind}$ of \eqref{eq:RN}  with respect to the distribution of $n$ independent SLE$_\k$ in $\prod_j \mc X(\domain; x_{a_j}, x_{b_j})$,
$$\m E_\k^{ind} \exp \left( \frac{c(\k)}{2} m_D  \right) \sim_{\k \to 0+}  \exp \left(-\frac{1}{\k} \inf_{\ad \g'}\left(\sum_{j=1}^\np I_{\domain}(\g'_j) 
+ 12 m_\domain(\ad \g')\right)\right)$$
since $c(\k)/2 \sim -12/\k$. The density of multichordal SLE$_\k$ is thus given by 
$$\frac{\exp \left( \frac{c(\k)}{2} m_D (\ad \g) \right) \prod_{j} \exp\left(-\frac{I_D (\g_j)}{\k}\right)}{\m E_\k^{ind} \exp \left( \frac{c(\k)}{2} m_D  \right)} \sim_{\k \to 0+} \exp\left(- \frac{I^\a_D(\ad \g) }{\k}\right). $$
\end{remark}

Theorem~\ref{thm:main_multi_LDP} and the uniqueness of the energy minimizer imply immediately:

\begin{cor} 
\label{cor:main_limit}
As $\k \to 0+$, multichordal $\SLE_\k$ 
in $\mc X_\a(\domain;x_1, \ldots, x_{2\np})$ converges in probability to 
the unique geodesic multichord $\ad \eta$
in $\mc X_\a(\domain;x_1, \ldots, x_{2\np})$. 
\end{cor}

\begin{proof}
Let $\mc B^h_\vare(\bar{\eta}) \subset \mc X_\a(\domain) : = \mc X_\a(\domain;x_1,\ldots,x_{2\np})$ 
be the Hausdorff-open ball of radius $\vare$ around the unique geodesic multichord $\bar{\eta}$.
Then, we have 
\begin{align*}
\limsup_{\k \to 0+} \k \log \m P^\k [\ad \g^\k \in \mc X_\a(\domain) \smallsetminus \mc B^h_\vare(\bar{\eta})] 
\leq - \inf_{\ad \g \in \mc X_\a(\domain) \smallsetminus \mc B^h_\vare(\bar{\eta})} I_\domain^\a(\ad \g) < 0.
\end{align*}
This proves the corollary. 
\end{proof}

\subsection{Minimal potential}

To define the energy $I_\domain^{\a}$, one could have added
to the potential $\mc H_\domain$ an arbitrary constant that depends only on the boundary data $(x_1, \ldots, x_{2\np};\a)$, e.g., one may drop the Poisson kernel terms in $\mc H_\domain$ which then alters the value of the minimal potential. 
The advantage of using the Loewner potential \eqref{eq:initialdef_Hn} is that it
allows comparing the potential of geodesic multichords of different boundary data. This becomes interesting when $n \ge 2$ as the moduli space of the boundary data is non-trivial.
We now discuss equations satisfied by the minimal potential based on \cite{peltola_wang} and the more recent work \cite{alberts2020pole}.

We first use Loewner's equation to describe each individual chord in the geodesic multichord, whose Loewner driving function can be expressed in terms of the minimal potential. We state the result when $D  = \m H$ and let $\mc U_\a = 12 \mc M_{\m H}^\a$.
\begin{thm}[\!\!{\cite[Prop.\,1.7]{peltola_wang}}]\label{thm:marginal}
Let $\ad \eta$ be the geodesic multichord in $\mc X_\a (\m H; x_1, \ldots, x_{2\np})$.
For each $j \in \{1,\ldots,\np\}$, the Loewner driving function $W$
of the chord $\eta_j$
and the evolution $V_t^i = g_t(x_i)$ of the other marked points satisfy the differential equations
\begin{align}\label{eqn:DE_for_drivers}
\begin{dcases}
\frac{\dd  W_t}{ \dd t} =  -\partial_{a_j} \mc U_\a (V_t^1, \ldots, V_t^{a_j-1}, W_t, V_t^{a_j+1}, \ldots, V_t^{2\np} )  , 
\qquad W_0 = x_{a_j} , \\
\frac{\dd V_t^i}{\dd t} = \frac{2}{V_t^i-W_t}, \qquad V_0^i = x_i, 
\quad \textnormal{for } i\neq a_j, 
\end{dcases}
\end{align}
for $0 \leq t < T$, where $T$ is the lifetime of the solution
and $(g_t)_{t \in [0,T]}$ is the Loewner flow generated by $\eta_j$. Similar equations hold with $a_j$ replaced by $b_j$.
\end{thm}

Here again, SLE large deviations enable us to speculate the form of Loewner differential equations \eqref{eqn:DE_for_drivers}. 
In fact, for each $\np$-link pattern $\a$, one associates to the multichordal $\SLE_\k$ 
a (\emph{pure}) \emph{partition function} 
$\PartF_\a$ defined  as 
\begin{align*}
\PartF_\a(\m H;x_1, \ldots,x_{2\np}) := 
\Big( \prod_{j=1}^\np P_{\m H;x_{a_j},x_{b_j}} \Big)^{(6-\k)/2\k} \times
 \m E_\k^{ind}\exp\left(\frac{c(\k)}{2} \, m_\domain(\ad \g) \right).
\end{align*}
As $-\k \log \left[P_{\m H;x_{a_j},x_{b_j}}{}^{(6-\k)/2\k}\right] \sim - 3 \log P_{\m H;x_{a_j},x_{b_j}}$, from Remark~\ref{rem:expectation_asymptotics} and \eqref{eq:initialdef_Hmin}
we obtain 
\begin{align}\label{eq:intro_limit_Z}
-\k \log \PartF_\a(\m H;x_1, \ldots,x_{2\np}) 
\quad \overset{\k \to 0+}{\longrightarrow}  \quad
\mc U_\a (x_1, \ldots,x_{2\np}).
\end{align}
The marginal law of the chord $\g_j^\k$ in the multichordal $\SLE_\k$ in $\mc X_\a (\m H; x_1, \ldots, x_{2\np})$ is given by the stochastic Loewner equation derived from $\mc Z_\a$:
\begin{align*}
\begin{dcases}
\dd W_t 
 = \sqrt{\kappa} \,\dd B_t 
+\kappa \,\partial_{a_j}\log\PartF_{\alpha} \left(V_t^1, \ldots, V_t^{a_j-1}, W_t, V_t^{a_j+1}, \ldots, V_t^{2n}\right) \dd t, \\ 
W_0  =  x_{a_j} , \\
\dd V_t^i  = \dfrac{2 \,\dd t}{V_t^i-W_t}, \qquad V_0^i = x_i, \quad \text{for } i\neq a_j. 
\end{dcases}
\end{align*}
 See~\cite[Eq.\,(4.10)]{PW19}). 
Replacing naively $\k \log \PartF_\a$ by 
$-\mc U_\a$, we obtain \eqref{eqn:DE_for_drivers}.

To prove Theorem~\ref{thm:marginal} rigorously, we analyse the geodesic multichords and the minimal potential directly and do not need to go through the SLE theory, which might be more tedious to control the errors when interchanging derivatives and limits.
Let us check \eqref{eqn:DE_for_drivers}  when $ n = 1$. For $n \ge 2$, we conformally map $(\hat D_j; x_{a_j}, x_{b_j})$ to $(\m H; 0, \infty)$ and use the conformal restriction formula which gives the change of the driving function under conformal maps.  See \cite[Sec.\,4.2]{peltola_wang}.

When $n =1$,
the minimal potential has an explicit formula:
\begin{align} \label{eq:Hmin_explicit2}
\Hmin_{\m H}(x_1, x_2) = \frac{1}{2} \log |x_2 - x_1|
\quad \Longrightarrow \quad 
\partial_1 \Hmin_{\m H}(x_1, x_2) = \frac{1}{2(x_1 - x_2)} .
\end{align}
The hyperbolic geodesic in $(\m H;x_1, x_2)$ is the semi-circle $\eta$ with endpoints $x_1$ and $x_2$.
We compute directly that $\frac{\dd}{\dd t} W_t |_{t = 0}  = 6(x_2 - x_1)^{-1}$. See, e.g., \cite[Eq.\,(4.3)]{peltola_wang} or \cite[Sec.\,5]{KNK_exact}.
Since hyperbolic geodesic is preserved under its own Loewner flow, i.e., $g_t (\eta_{[t,T]})$ is the semi-circle with end points $W_t$ and $V_t = g_t (x_2)$, we obtain 
\begin{align*} \begin{dcases}
\frac{\dd W_t}{\dd t}  = \; & \dfrac{6}{V_t - W_t}  , 
\qquad W_0 = x_1 , \\
\frac{\dd V_t}{\dd t}  = \; & \dfrac{2 }{V_t - W_t} , \qquad V_0 = x_2.
\end{dcases}
\end{align*}
By~\eqref{eq:Hmin_explicit2}, 
this is exactly Equation~\eqref{eqn:DE_for_drivers} when $n = 1$.

Similarly, the level two null-state Belavin-Polyakov-Zamolodchikov equations satisfied by the SLE partition function
\begin{align}\label{eq: BPZ}
\left( \frac{\kappa}{2} \partial_{x_j}^2 + \sum_{i \neq j} \left( \frac{2}{x_i - x_j} \partial_{x_i} - \frac{(6-\kappa)/\kappa}{(x_i - x_j)^2} \right) \right) \mathcal{Z}_\a = 0, \quad j=1,\ldots,2n, 
\end{align}
prompts us to find the following equations (see also \cite{BBK,alberts2020pole}). 

\begin{thm}[\!\! {\cite[Prop.\,1.8]{peltola_wang}}] \label{thm:null-state}
For $j \in \{1,\ldots, 2\np\}$, we have
\begin{align} \label{eq:deterministic_PDEs}
\frac{1}{2} ( \partial_j \mc U_\a (x_1, \ldots,  x_{2\np}) )^2 - \sum_{i \neq j} 
\frac{2}{x_i - x_j} \partial_i \mc U_\a (x_1, \ldots,  x_{2\np})
 = \sum_{ i \neq j}  \frac{6}{(x_i - x_j)^2}.
\end{align}
\end{thm}

The recent work \cite{alberts2020pole} gives further an explicit expression of $\mc U_\a (x_1, \cdots, x_{2n})$ in terms of the rational function $h_\eta$ associated to the geodesic multichord in $\mc X_\a (\m H; x_1, \ldots, x_{2n})$ as considered in Section~\ref{subsec:rational}. 
More precisely, following \cite{alberts2020pole}, we normalize the rational function such that $h_\eta (\infty) = \infty$ by possibly post-composing $h_\eta$ by an element of $\PSL(2,\m R)$ and denote the other $n$ poles  $(\zeta_{\a,1}, \cdots, \zeta_{\a,n})$ of $h_\eta$.

\begin{thm}[\!\! {\cite[Thm.\,2.8]{alberts2020pole}}]
For the boundary data $(x_1, \ldots, x_{2n}; \a)$, we have
\begin{align} \label{eq: Z}
\exp (-\mc U_\a) = \, C \!\!\prod_{1 \leq j < k \leq 2n} \!\! (x_j - x_k)^2 \!\! \prod_{1 \leq l < m \leq n} \!\! (\zeta_{\alpha,l} - \zeta_{\alpha,m})^8 \prod_{k=1}^{2n} \prod_{l=1}^n (x_k - \zeta_{\alpha,l})^{-4},
\end{align}
where $C$ is a constant which only depends on $n$.
\end{thm}

\begin{remark}
Finally let us remark that another reason to include the Poisson kernel to the Loewner potential is that it relates to the more general framework of defining Loewner energy in terms of the zeta-regularized determinants of Laplacians.
We do not enter into further details here and refer the interested readers to \cite[Thm.\,1.9]{peltola_wang}.
\end{remark}

\section{Large deviations of radial $\SLE_\infty$}\label{sec:radial_infty}
We now turn to the large deviations of $\SLE_\infty$, namely, when $\vare : = 1/\kappa$ using the notation in Definition~\ref{df:LDP}. From \eqref{eqn:LE}, one can easily show that in the chordal setup, for any fixed $t$, the conformal map $f_t = g_t^{-1} : \m H \to \m H\smallsetminus K_t$ converges uniformly on compact sets to the identity map as $\k \to \infty$ almost surely. In other words, the complement of the $\SLE_\k$ hull converges for the \emph{Carath\'eodory topology} towards $\m H$, which is not interesting for the large deviations. The main hurdle is that the driving function can be arbitrarily close to the target boundary point (i.e., $\infty$) where we normalize the conformal maps in the Loewner evolution.
For this reason, we switch to the radial version of SLE.

\subsection{Radial SLE}\label{sec:radial_SLE}

We now describe the radial SLE on the unit disk $\mathbb{D}$ targeting at $0$. 
The \emph{radial Loewner differential equation} driven by a continuous function $\m R_+ \to S^1 : t \mapsto \zeta_t$ is defined as follows: For all $z \in \mathbb{D}$, consider the equation
\begin{equation}
\partial_t g_t(z) = g_t(z)\frac{\zeta_t+g_t(z)}{\zeta_t-g_t(z)}, \qquad g_0 (z) = z.  \label{eq:sle_ODE}
\end{equation}
As in the chordal case, the solution $t \mapsto g_t(z)$ to~\eqref{eq:sle_ODE} is defined up to the swallowing time
\begin{align*}
\tau(z) := \sup\{ t \geq 0 \, |\, \inf_{s\in[0,t]}|g_{s}(z)-\zeta_{s}|>0\},
\end{align*}
and the growing hulls are given by $K_t = \ad{\{z \in \m D \, |\,  \tau(z) \leq t\}}$. 
The solution $g_t$ is the conformal map from $D_t := \m D \smallsetminus K_t$ onto $\m D$ satisfying $g_t (0) = 0$ and $g_t'(0) = e^t$.

\emph{Radial SLE$_\kappa$} is the curve $\g^\k$ tracing out the growing family of hulls $(K_t)_{t \ge 0}$ driven by a Brownian motion on the unit circle $S^1=\{\zeta\in\mathbb{C}: \left\lvert \zeta \right \rvert=1\}$ of variance $\kappa$, i.e.,
\begin{equation}\zeta_t := \b^\k_t = e^{iB_{\kappa t}}, \label{def:circ-bm}\end{equation}
where $B_t$ is a standard one dimensional Brownian motion.
Radial SLEs exhibit the same phase transitions as in the chordal case as $\kappa$ varies. In particular, when $\kappa \geq 8$, $\g^\k$ is almost surely space-filling and $K_t = \g^\k_{[0,t]}$.

We now argue heuristically to intuit the $\k\to \infty$ limit and the large deviation result of radial SLE proved in \cite{APW}. 
  During a short time interval $[t, t + \Delta t]$ where the Loewner flow is well-defined for a given point $z \in \m D$, we have $g_s(z) \approx g_{t}(z)$ for $s \in [t, t + \Delta t]$. Hence, writing the time-dependent vector field  $(z (\zeta_t+z)(\zeta_t -z)^{-1})_{t \ge 0}$ generating the Loewner chain as 
$(\int_{S^1} z (\zeta+z)(\zeta-z)^{-1} \delta_{\b^\k_t}(\dd \zeta))_{t \ge 0}$, where $\delta_{\b^\k_t}$ is the Dirac measure at $\b^\k_t$, we obtain that  $\Delta g_t (z)$ is approximately
\begin{align}\label{eq:heuristic_uniform}
  \int_{t}^{t+\Delta t} \int_{S^1} g_t (z) \frac{\zeta +g_t (z)}{\zeta - g_t (z)} \,\delta_{\b^\k_s}(\dd \zeta)\dd s = \int_{S^1} g_t (z) \frac{\zeta + g_t (z)}{\zeta - g_t (z)} \,\dd (\ell^\k_{t+\Delta t} (\zeta)- \ell^\k_{t}(\zeta)),   
\end{align}
where $\ell^\k_t$ is the \emph{occupation measure} (or \emph{local time}) on $S^1$ of $\b^\k$ up to time $t$. 
As $\kappa \to \infty$, the occupation measure of $\b^\k$ during $[t,t+\Delta t]$ converges to the uniform measure on $S^1$ of total mass $\Delta t$.  Hence the radial Loewner chain converges to a \emph{measure-driven Loewner chain} (also called \emph{Loewner-Kufarev chain}) with the uniform probability measure on $S^1$ as driving measure, i.e.,
$$\partial_t g_t (z) = \frac{1}{2\pi} \int_{S^1} g_t(z)\frac{\zeta +g_t(z)}{\zeta - g_t(z)} \,|\dd \zeta| = g_t(z). $$ 
This implies $g_t (z) = e^{t} z$. Similarly, \eqref{eq:heuristic_uniform} suggests that the large deviations of SLE$_\infty$ can also be obtained from the large deviations of the process of occupation measures $(\ell^\k_t)_{t\ge 0}$.

\subsection{Loewner-Kufarev equations in $\m D$}\label{subsec:LK_eq}
We now give a more detailed account of the Loewner-Kufarev chain. 
The Loewner chains described in Section~\ref{sec:chordal_Loewner} and \ref{sec:radial_SLE} are driven by a function taking values in $\m R$ or $S^1$. It is well-adapted to the study the conformal map from the unit disk or the upper half-plane to a slit domain by progressively growing the slit and we obtain an evolution family of slit domains.
This method was extended by Kufarev \cite{Kufarev} and further developed by Pommerenke \cite{Pom1965} to cover general evolution families beyond slit domains. In this case, a family of measures drive the dynamics and is described by the Loewner-Kufarev equation.

Let $\mc M (\O)$ (resp. $\mc M_1 (\O)$) be the space of Borel measures (resp. probability measures) on $\O$. We define
$$\mathcal{N}_+ = \{\rho\in\mathcal{M}(S^1 \times \m R_+): \rho(S^1\times I)= |I| \text{ for all intervals } I \subset \m R_+\}.$$
From the disintegration theorem (see e.g. \cite[Theorem 33.3]{billingsley}),
for each measure $\rho\in\mathcal{N}_+$ there exists a Borel measurable map $t\mapsto \rho_t$ from $\m R_+$ to $\mathcal{M}_1(S^1)$ such that
$\dd \rho = \rho_t(\dd \zeta)\,\dd t$.
We say $(\rho_t)_{t \ge 0}$ is a \emph{disintegration} of $\rho$; it is unique in the sense that any two disintegrations $(\rho_t)_{t \ge 0}, (\widetilde \rho_t)_{t \ge 0}$ of $\rho$ must satisfy $\rho_t = \widetilde \rho_t$ for a.e. $t \ge 0$. We denote by $(\rho_t)_{t \ge 0}$ one such disintegration of $\rho\in\mathcal{N}_+$.

For $z \in \mathbb{D}$, consider the \emph{Loewner-Kufarev ODE}
\begin{equation}\label{eq:ODE_LK}
\partial_t g_t(z) = g_t(z) \int_{S^1} \frac{ \zeta + g_t(z)}{\zeta - g_t (z) } \,\rho_t(\dd \zeta), \quad  g_0(z)=z.   
\end{equation}
Let $\tau(z)$ be the supremum of all $t$ such that the solution is well-defined up to time $t$ with $g_t(z)\in\mathbb{D}$, and $D_t:=\{z\in \mathbb{D}: \tau(z) >t\}$ is a simply connected open set containing $0$. 
The function $g_t$ is the unique conformal map of $D_t$ onto $\mathbb{D}$ such that $g_t (0) = 0$ and $g_t'(0) > 0$.
Moreover, it is straightforward to check that 
$\partial_t \log g_t'(0) =  |\rho_t| = 1$. Hence,  $g_t'(0) = e^t$, namely, $D_t$ has conformal radius $e^{-t}$ seen from $0$.
We call $(g_t)_{t \ge 0}$ the \emph{Loewner-Kufarev chain} (or simply Loewner chain) driven by $\rho \in \mc N_+$. 

It is also convenient to use its inverse $(f_t := g_t^{-1})_{t  \ge 0}$, which satisfies the \emph{Loewner PDE}:
\begin{equation} \label{eq:loewner-pde}
\partial_t f_t (z) =   -z f_t'(z)  \int_{S^1} \frac{\zeta + z}{\zeta -z}  \,\rho_t(\dd \zeta)  , \quad f_0 (z) = z.
\end{equation}

We write $\mc {L}_+$ for the set of Loewner-Kufarev chains defined for time  $\m R_+$. An element of $\mathcal{L}_+$ can be equivalently represented by $(f_t)_{t\ge 0}$ or $(g_t)_{t\ge 0}$ or the evolution family of domains $(D_t)_{t\ge 0}$ or the evolution family of hulls $(K_t = \ad{\m D \smallsetminus D_t})_{t\ge 0}$.  

\begin{remark}\label{rem:Pommerenke_existence}
In terms of the domain evolution,  according to a theorem of Pommerenke \cite[Satz 4]{Pom1965} (see also \cite[Thm.\,6.2]{Pom_uni} and \cite{rosenblum}), $\mc L_+$ consists exactly of those $(D_t)_{t \ge 0}$  such that $D_t \subset \m D$ has conformal radius  $e^{-t}$ and for all $0 \le s \le t$, $D_t \subset D_s$.
\end{remark}

We now restrict the Loewner-Kufarev chains to the time interval $[0,1]$ for the topology discussion and simplicity of notation. The results can be easily generalized to other finite intervals $[0,T]$ or to $\m R_+$ as the projective limit of chains on all finite intervals.
Define
$$\mathcal{N}_{[0,1]} = \{\rho\in\mathcal{M}_1(S^1 \times [0,1]): \rho(S^1\times I)= |I| \text{ for all intervals } I \subset [0,1]\},$$
endowed with the Prokhorov topology (the topology of weak convergence) and the corresponding set of restricted Loewner chains $\mc L_{[0,1]}$.
 Identifying an element $(f_t)_{t\in [0,1]}$ of $\mc L_{[0,1]}$ with the function $f$ defined by $f(z,t) = f_t(z)$ and endow $\mathcal{L}_{[0,1]}$ with the topology of uniform convergence of $f$ on compact sets of $\mathbb{D} \times [0,1]$. (Or equivalently, viewing $\mathcal{L}_{[0,1]}$ as the set of domain evolutions $(D_t)_{t \in [0,1]}$, this is the topology of \emph{uniform Carath\'eodory convergence}.)
The following result allows us to study the limit and large deviations with respect to the  topology of uniform Carath\'eodory convergence.
\begin{thm}[\!\! {\cite[Prop.\,6.1]{MilShe2016}},\cite{JohSolTur2012}]\label{thm:cont-bij-loewner-transf}
	The Loewner transform $ \mathcal{N}_{[0,1]} \to \mathcal{L}_{[0,1]}$:  $\rho \mapsto f$ is a homeomorphism. 
\end{thm}

By showing that the random measure $\d_{\b_t^\k} (\dd \zeta) \,\dd t \in \mc N_{[0,1]}$ converges almost surely to the uniform measure $(2\pi)^{-1}|\dd \z|\, \dd t $ on $S^1 \times [0,1]$ as $\k \to \infty$, we obtain:

\begin{thm}[\!\! {\cite[Prop.\,1.1]{APW}}]\label{thm-LLN-main}
	As $\kappa \to \infty$, the domain evolution $(D_t)_{t\in [0,1]}$ of the radial $\SLE_\k$  converges  almost surely to $( e^{-t} \m D)_{t\in [0,1]}$ for the uniform Carath\'eodory topology.
\end{thm}

\subsection{Loewner-Kufarev energy and large deviations}

From the contraction principle Theorem~\ref{thm:contraction} and Theorem~\ref{thm:cont-bij-loewner-transf}, the large deviation principle of radial $\SLE_\k$ as $\k \to \infty$ boils down to the large deviation principle of $\d_{\b_t^\k} (\dd \zeta) \, \dd t \in \mc N_{[0,1]}$ with respect to the Prokhorov topology. For this, we approximate $\d_{\b_t^\k} (\dd \zeta) \,\dd t$ by 
\[ \rho^\k_n : =  \sum_{i = 0}^{2^n -1} \mu_{n,i}^\k (\dd \z) \1_{t \in [i/2^n, (i+1)/2^n)} \,\dd t, \]
where $\mu_{n,i}^\k \in \mc M_1 (S^1)$ is the time average of the measure $\delta_{\b^\k_t}$ on the interval $[i/2^n, (i+1)/2^n)$. In terms of the occupation measures, 
$$\mu_{n,i}^\k = 2^{n}( \ell^\k_{(i+1)/2^n} - \ell^\k_{i/2^n}).$$

We start with the large deviation principle for $\mu_{n,i}^\k$ as $\k \to \infty$. Let $\ad \ell^\k_t := t^{-1} \ell^\k_t$ be the average occupation measure of $\b^\k$ up to time $t$.
From the Markov property of Brownian motion, we have $\mu^\k_{n,i} = \ad \ell^\k_{2^{-n}}$ in distribution up to a rotation (by $\b^\k_{i/2^n}$). 
The following result is a special case of a theorem of Donsker and Varadhan. 

	Define the functional $I^{DV}: \mathcal{M}(S^1) \to [0,\infty]$ by
	\begin{equation} \label{eq:I_DV}
	I^{DV}(\mu)	= \frac12 \int_{S^1} |v' (\zeta)|^2 \, |\dd \zeta|,	    
	\end{equation}
		if $\mu = v^2(\zeta) |\dd \zeta|$ for some function $v \in W^{1,2}(S^1)$ and $\infty$ otherwise.

\begin{remark} Note that $I^{DV}$ is rotation-invariant and
  $I^{DV} (c \mu_i) = c I^{DV} (\mu_i)$ for $c > 0$.
\end{remark}

\begin{thm}[\!\! {\cite[Thm.\,3, Thm.\,5]{DonVar1975}}]\label{thm-DV}

Fix $t > 0$. The average occupation measure $\{\ad \ell^\k_t\}_{\k > 0}$ admits a large deviation principle as $\kappa \to \infty$ with good rate function $t I^{DV}$. 	Moreover, $I^{DV}$ is  convex.
	\end{thm}
	\begin{remark}
	The expression $I^{DV}$ in \cite[Thm.\,3]{DonVar1975} is in a different form but is shown in \cite[Thm.\,5]{DonVar1975} to equal to \eqref{eq:I_DV} for general Markov processes. See also \cite[Thm.\,3.5]{APW} for an alternative elementary proof of the identity which is adapted to the specific case of Brownian occupation measure on $S^1$.
	\end{remark}
 The $\k \to \infty$ large deviation principle is understood in the sense of Definition~\ref{df:LDP} with $\vare = 1/\k$, i.e.,
 for any open set $\open$ and closed set $\closed \subset \mathcal M_1(S^1)$,
	\begin{align*}
		\liminf_{\kappa \to \infty} \frac{1}{\kappa} \log \mathbb{P} [ \ad \ell^\k_t \in \open ] \geq -\inf_{\mu \in \open} t I^{DV}(\mu); \\
	\limsup_{\kappa \to \infty} \frac{1}{\kappa} \log \mathbb{P} [ \ad \ell^\k_t \in \closed ] \leq -\inf_{\mu \in \closed} t I^{DV}(\mu).
\end{align*}

Theorem~\ref{thm-DV} and the Markov property of Brownian motion imply that the $2^n$-tuple $(\mu^\k_{n,0}, \ldots. \mu^\k_{n,2^n -1})$ satisfies the large deviation principle with rate function as $\k \to \infty$
$$I_n^{DV} (\mu_0, \ldots, \mu_{2^n -1}) := 2^{-n} \sum_{i = 0}^{2^n-1}  I^{DV} (\mu_i). $$

Taking the $n \to \infty$ limit, it leads to the following definition.
 \begin{df}
 We define the \emph{Loewner-Kufarev energy} on $\mathcal{L}_{[0,1]}$ (or equivalently on $\mathcal{N}_{[0,1]}$)
 $$S_{[0,1]} ((D_t)_{t \in [0,1]}) : = 
 S_{[0,1]} (\rho) :=\int_0^1 I^{DV}(\rho_t)\, dt$$ 
 	where $\rho$ is the driving measure generating $(D_t)_{t \in [0,1]}$.    
 \end{df}
 
\begin{thm}[\!\!{\cite[Thm.\,1.2]{APW}}]
	\label{thm:ldp-circ-bm}
	The measure $\d_{\b_t^\k} (\dd \zeta) \,\dd t\in\mathcal{N}_{[0,1]}$ satisfies the large deviation principle with good rate function $S_{[0,1]}$ as $\k \to \infty$.
\end{thm}

\begin{proof}[Proof sketch]
We show that $\mc N_{[0,1]}$ is homeomorphic to the projective limit (Definition~\ref{df:projective}) of the projective system consisting of $\{\mc Y_n : = \mc M_1(S^1)^{2^n}\}_{n \ge 1}$ and 
$$\pi_{n,n+1} (\mu_0, \ldots, \mu_{2^{n+1}-1}) = \left(\frac{\mu_0 + \mu_1}{2}, \ldots, \frac{\mu_{2^{n+1}-2} +\mu_{2^{n+1}-1}}{2}\right)$$
(other projections $\pi_{ij}$ are obtained by composing consecutive projections). The canonical projection $\pi_n$ is given by $\mc N_{[0,1]} \to \mc Y_n :  \rho \mapsto (\mu_{n,i})_{i = 0, \cdots, 2^n-1}$, where 
$$ \quad \mu_{n,i} = 2^n \int_{i 2^{-n}}^{(i+1)2^{-n}} \rho_t \, \dd t \in \mc M_1 (S^1), \quad i = 0, \cdots, 2^n-1.$$
(Note that  $\pi_n (\d_{\b_t^\k} (\dd \zeta) \,\dd t) = (\mu_{n,i}^\k)_{i = 0, \cdots, 2^n-1}$.) See \cite[Lem.\,3.1]{APW}. 
We then show that 
$$\lim_{n \to \infty} I_n^{DV}(\pi_n (\rho)) = \sup_{n \ge 1} I_n^{DV}(\pi_n (\rho))  =  S_{[0,1]}(\rho),$$
see \cite[Lem.\,3.8]{APW}, and conclude with Dawson-G\"artner's Theorem~\ref{thm:Dawson-Gartner}.
\end{proof}

\begin{remark}
We note that if $\rho \in \mc N_{[0,1]}$ has finite Loewner-Kufarev energy, then $\rho_t$ is absolutely continuous with respect to the Lebesgue measure for a.e. $t$ with density being the square of a function in $W^{1,2} (S^1)$. In particular, $\rho_t$ is much more regular than a Dirac measure. We see once more the regularizing phenomenon from the large deviation consideration.
\end{remark}

From the contraction principle Theorem~\ref{thm:contraction} and Theorem~\ref{thm:cont-bij-loewner-transf}, we obtain immediately:

\begin{cor}[\!\!{\cite[Cor.\,1.3]{APW}}]
	\label{cor:ldp-sle}
	The family of $\operatorname{SLE}_{\kappa}$ on the time interval $[0,1]$ satisfies the $\kappa \to \infty$ large deviation principle with the good rate function $S_{[0,1]}.$
\end{cor}

\section{Foliations by Weil-Petersson quasicircles} \label{sec:foliation}

SLE processes enjoy a remarkable duality \cite{Dub_duality,Zhan_duality,IG1}
coupling $\operatorname{SLE}_{\kappa}$ to the outer boundary of $\operatorname{SLE}_{16/\kappa}$ for $\kappa < 4$. 
It suggests that the rate functions of $\operatorname{SLE}_{0+}$ (Loewner energy) and  $\operatorname{SLE}_\infty$ (Loewner-Kufarev energy) are also dual to each other. Let us first remark that when $S_{[0,1]} (\rho) = 0$, the generated family $(D_t)_{t \in [0,1]}$ consists of concentric disks centered at $0$. In particular $(\partial D_t)_{t \in [0,1]}$ are circles and thus have zero Loewner energy.  This trivial example supports the guess that some form of energy duality holds.

Viklund and the author  investigated in \cite{VW2} the duality between these two energies, and more generally, the interplay with the Dirichlet energy of a so-called winding function. We now describe briefly those results. While our approach is originally inspired by SLE theory, they are of independent interest from the analysis perspective and the proofs do not involve probability theory.

\subsection{Whole-plane Loewner evolution}\label{subsec:whole_plane_Loewner_chain}

To describe our results in the most generality, we consider the Loewner-Kufarev energy for Loewner evolutions defined for $t \in \m R$, namely the whole-plane Loewner chain.
We define in this case the space of driving measures to be
$$\mc N := \{\rho\in \mc M (S^1 \times \m R): \rho(S^1\times I) = |I| \text{ for all intervals } I \}.$$  
The whole-plane Loewner chain driven by $\rho \in \mc N$, or equivalently by its measurable family of disintegration measures $\m R \to \mc M_1 (S^1): \, t \mapsto \rho_t$, is the unique family of conformal maps $(f_t : \m D \to D_t)_{t \in \m R}$ such that 
\begin{enumerate}[label= (\roman*)]
    \item For all $s < t$, $0 \in D_t \subset D_s$. \label{it:whole_monotone}
    \item For all $t \in \m R$, $f_t (0) = 0$ and $f_t'(0) = e^{-t}$ (namely,  the conformal radius of $D_t$ is $e^{-t}$). \label{it:whole_radius}
    \item For all $s \in \m R$, $( f_t^{(s)} : = f_{s}^{-1} \circ f_t :  \m D \to D_{t}^{(s)})_{t \ge s}$ is the Loewner chain driven by $(\rho_t)_{t\ge s}$, which satisfies \eqref{eq:loewner-pde} with the initial condition $f_s^{(s)} (z) = z$. 
    \label{it:whole_Loewner}
\end{enumerate}
See, e.g., \cite[Sec.\,7.1]{VW2}
for a proof of the existence and uniqueness of such family.

\begin{remark}
   If $\rho_t$ is the uniform probability measure for all $t \le 0$, then $f_t (z) = e^{-t} z$ for $t \le 0$ and $(f_t)_{t\ge 0}$ is the Loewner chain driven by $(\rho_t)_{t \ge 0} \in \mc N_+$. 
   Indeed, we check directly that $(f_t)_{t\in \m R}$ satisfy the three conditions above. A Loewner chain in $\m D$ considered in Section~\ref{subsec:LK_eq} can therefore be seen as a special case of whole-plane Loewner chain.
\end{remark}

 Note that the condition \ref{it:whole_Loewner} is equivalent to for all $t \in \m R$ and $z  \in \m D$,
\begin{equation} \label{eq:whole-plane-radial}
 \partial_t f_t (z) = -z f'_t (z) \int_{S^1} \frac{\zeta + z}{\zeta -z}  \,\rho_t(\dd \zeta).
\end{equation}
As in \eqref{eq:ODE_LK}, the family uniformizing maps $(g_t : = f_t^{-1})_{t\in \m R}$ satisfies the Loewner-Kufarev ODE: For $z \in D_{t_0}$ and $t \in (-\infty, t_0)$, we have
\begin{equation}\label{eq:whole_ODE}
    \partial_t g_t (z) = g_t(z) \int_{S^1} \frac{\zeta + g_t(z)}{\zeta - g_t(z)}  \,\rho_t(\dd \zeta).
\end{equation}

We remark that for $\rho \in \mc N$, then  $\cup_{t\in \m R} D_t = \m C$.   Indeed, $D_{t}$ has conformal radius $e^{-t}$, therefore contains the centered ball of radius $e^{-t}/4$ by Koebe's $1/4$ theorem. 
Since $\{ t\in \m R :  z \in D_t\} \neq \emptyset$, we define for all $z \in \m C$, 
$$\tau(z) : =  \sup\{ t\in \m R :  z \in D_t\} \in (-\infty, \infty].$$

We say that $\rho \in \mc N$ generates a \emph{foliation} $(\g_t : = \partial D_t)_{t\in \m R}$ of $\m C \smallsetminus \{0\}$ if
\begin{enumerate}
    \item For all $t \in \m R$, $\g_t$ is a chord-arc Jordan curve.
    \item It is possible to parametrize each curve $\g_t$ by $S^1$ so that the mapping $t \mapsto \g_t$ is continuous in the supremum norm.
    \item For all $z \in \m C \smallsetminus \{0\}$, $\tau(z) <\infty$.
\end{enumerate}
Each $\g_t$ of a foliation is called a \emph{leaf}. 
\begin{df}\label{df:rho_infty}
 We define similarly the \emph{Loewner-Kufarev energy} on $\mathcal{N}$ by
 $$
 S (\rho) :=\int_{-\infty}^\infty I^{DV}(\rho_t)\, dt,$$ 
 	where $I^{DV}$ is given by \eqref{eq:I_DV}. 
 \end{df}

\subsection{Energy duality}
The following result gives a qualitative relation between finite Loewner-Kufarev energy measures and finite Loewner energy curves (i.e., Weil-Petersson quasicircles  by Theorem~\ref{thm:intro_equiv_energy_WP}).

\begin{prop}[Weil-Petersson foliation {\cite[Thm.\,1.1]{VW2}}]\label{prop:WP-leaf}
  Suppose $\rho \in \mc N$ has finite Loewner-Kufarev energy. Then $\rho$ generates a foliation $(\g_t  = \partial D_t)_{t\in \m R}$ of $\m C \smallsetminus \{0\}$ in which all leaves are Weil-Petersson quasicircles.
\end{prop}

Every $\rho$ with $S(\rho) < \infty$ thus generates a family of Weil-Petersson quasicircles that continuously sweep out the Riemann sphere,
starting from $\infty$ and moving towards $0$, as $t$ goes from $-\infty$ to $\infty$.  Therefore, we can view $S(\rho)$ as the energy of the generated foliation. Conversely, Theorem~\ref{thm:main-jordan-curve} shows that any Weil-Petersson quasicircles can be generated by a measure with finite Loewner-Kufarev energy (if they separate $0$ from $\infty$).

\begin{remark}
This class of measures is a rare case for which a complete description of the geometry of the generated non-smooth interfaces is possible. Becker gave a sufficient condition on the Loewner chain to generate quasicircles \cite{becker1972}, but not all quasicircles can be generated this way \cite[Thm.\,3]{GP}.
\end{remark}

\begin{figure}
    \centering
    \includegraphics[width=.4\textwidth]{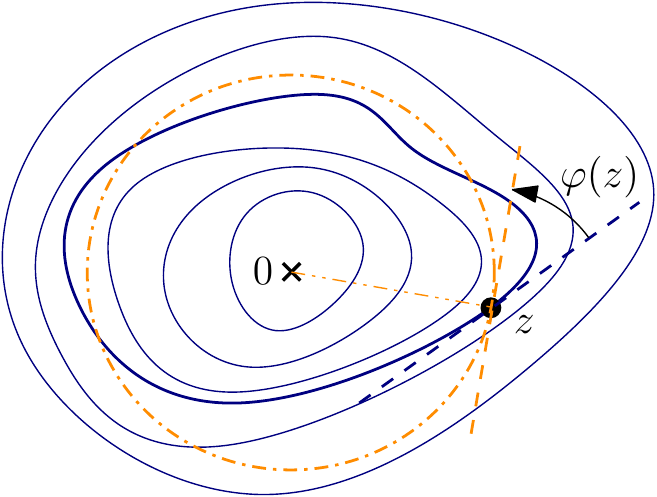}
    \caption{Illustration of the winding function $\varphi$.}
    \label{fig:intro_fol}
\end{figure}

We also show a quantitative relation among energies by introducing a real-valued \emph{winding function} $\varphi$ associated to a foliation $(\g_t)_{t \in \m R}$ as follows. 
Let $g_t = f_t^{-1}: D_t \to \m D$. 
Since $\g_t$ is by assumption chord-arc, thus rectifiable, $\g_t$ has a tangent arclength-a.e. Given $z$ at which $\g_t$ has a tangent, we define 
\begin{equation}\label{eq:winding}
\varphi(z) = \lim_{w \to z} \arg \frac{g'_{t}(w)w}{g_t(w)}
\end{equation}
where the limit is taken inside $D_t$ and approaching $z$ non-tangentially. We choose the continuous branch of $\arg$ which vanishes at $0$. 
Monotonicity of $(D_t)_{t \in \m R}$ implies that there is no ambiguity in the definition of $\varphi(z)$ if $z \in \g_t \cap \g_s$. See \cite[Sec.\,2.3]{VW2} for more details.
\begin{remark} \label{rem:geom}
Geometrically, $\varphi(z)$ equals the difference of the argument of the tangent to the circle centered at $0$ passing through $z$ and that of  $\gamma_{t}$  modulo $2\pi$, see Figure~\ref{fig:intro_fol}. 
Note also that in the trivial example when $\rho$ has zero energy, the generated foliation consists of concentric circles centered at $0$ whose winding function is identically $0$. 
\end{remark}

The following is our main theorem.

\begin{thm}[Energy duality {\cite[Thm.\,1.2]{VW2}}]\label{thm:main0} Assume that $\rho \in \mc N$ generates a foliation and
   let $\varphi$ be the associated winding function on $\m C$. Then $\mc D_{\m C} (\varphi) <\infty$  if and only if $S(\rho) <\infty$ and 
   $\mc D_{\m C} (\varphi) = 16 \, S(\rho)$.
   \end{thm}
   
Although this result is about deterministic growth processes, it has the heuristic interpretation of being the large deviation counterpart of a potential \emph{radial mating of trees} coupling between the whole-plane SLE$_\k$ with large $\k$  (whose large deviation rate function is $S(\rho)$) with a whole-plane Gaussian free field with vanishing multiplicative factor (whose rate function is $\mc D_{\m C} (\varphi)$) that we also speculated in \cite[Sec.\,10]{VW2}. Here, the SLE is the flow-line of a unit random vector field which makes the angle a multiple of GFF with the vector field  $\ii z$. It is closely related to \cite{IG4} and analogous to the mating of trees theorem of Duplantier, Miller, and Sheffield \cite{MoT}  (see also \cite{GHS_survey_MoT} for a recent survey) in the chordal SLE setup.  
In fact, our intuition for Theorem~\ref{thm:main0} comes from the belief that such a radial mating of trees result should hold although its exact form was not clear to us before we proved this deterministic result. 

We also note that if one stops the whole-plane $\SLE_\k$ at any time $t$, the boundary of the domain $D_t$ is locally a $\SLE_{16/\k}$ curve. 
This is consistent with the fact that a measure with finite Loewner-Kufarev energy drives an evolution of domains bounded by Jordan curves with finite Loewner energy as stated in Proposition~\ref{prop:WP-leaf}. It will become more apparent in Theorem~\ref{thm:main-jordan-curve}, which is a quantitative version of Proposition~\ref{prop:WP-leaf},  that the factor $16$ is consistent with the SLE duality which relates $\SLE_\k$ to $\SLE_{16/\k}$. See also \cite[Sec.\,10]{VW2} for more discussion.

\begin{remark}
   A subtle point in defining the winding function is that  in the general case of a chord-arc foliation, a function defined arclength-a.e. on each leaf need not be defined Lebesgue-a.e., see, e.g., \cite{FubiniFoiled}. Thus to consider the Dirichlet energy of $\varphi$, we use the following extension to  $W^{1,2}_{loc}$. A function $\varphi$ defined arclength-a.e.\ on all leaves of a foliation $(\g_t = \partial D_t)$ is said to have an \emph{extension} $\phi$ in $W^{1,2}_{loc}$ if for all $t \in \m R$, the Jonsson-Wallin trace  (see Equation \eqref{def:trace})  of $\phi$ on $\g_t$, that we simply write as $\phi|_{\g_t}$,
   coincides with $\varphi$ arclength-a.e on $\g_t$. We also show that if such extension exists then it is unique.
   The Dirichlet energy of $\varphi$ in the statement of Theorem~\ref{thm:main0} is understood as the Dirichlet energy of this extension.
\end{remark}

Theorem~\ref{thm:main0} has several applications which show that the foliation of Weil-Petersson quasicircles generated by $\rho$ with finite Loewner-Kufarev energy exhibits several remarkable features and symmetries. 
     
     The first is the reversibility of the Loewner-Kufarev energy. Consider $\rho \in \mc N$ and the corresponding family of domains $(D_t)_{t \in \m R}$. Applying $\iota \colon z\mapsto 1/z$ to $D_t^* := \hat{\m C} \smallsetminus D_t$, we obtain an evolution family of domains $(\tilde D_t = \iota (D_t^*))_{t \in \m R}$ upon time-reversal and reparametrization, which may be described by the Loewner equation with an associated driving measure $\tilde \rho$. While there is no known simple description of $\tilde \rho$ in terms of $\rho$, energy duality implies remarkably that the Loewner-Kufarev energy is invariant under this transformation. 
\begin{thm}[Energy reversibility {\cite[Thm.\,1.3]{VW2}}]  \label{thm:main_rev}
We have $S (\rho) = S(\tilde \rho).$
\end{thm} 
To see this, we prove first the following result.
\begin{lem}\label{lem:matching_trace}
Let $\g$ be a rectifiable Jordan curve separating $0$ from $\infty$, $D$ and $D^*$ be respectively the bounded and unbounded domain of $\Chat \smallsetminus \g$. Let $g$ be a conformal map from $D\to \m D$ fixing $0$ and  $k$ be a conformal map $D^* \to \m D^*$ fixing $\infty$ such that $g'(0) > 0$ and $k'(\infty) > 0$. For a differentiable point  $z \in \g$, we have
$$\varphi (z) =  \lim_{w \to z, \, w \in D} \arg \frac{g'(w)w}{g(w)} = \lim_{ w \to z, \,w \in D^*} \arg \frac{k'(w)w}{k(w)} = : \psi (z)  $$
where the limits are taken non-tangentially and we choose the continuous branches of $\arg (\cdot)$ such that $\varphi (w) := \arg  \frac{g'(w)w }{g(w)} \to 0$ as $w \to 0$ and $\psi (w) := \arg \frac{k'(w)w}{k(w)} \to 0$ as $w \to \infty$.
\end{lem}
In particular, this lemma implies that the winding function at $\iota (z)$ of the curve $\iota (\g)$ is
$$\lim_{w \to \iota (z), \, w \in \iota (D^*)} \arg \frac{\tilde g'(w)w}{\tilde g(w)} = \lim_{\iota (w) \to z, \,\iota (w) \in D^*} \arg \frac{k'(\iota (w)) \iota (w)}{k(\iota (w))} = \psi (z) = \varphi (z),  $$
since $\tilde g (w)= \iota \circ k \circ \iota (w)$. 
With this observation, the proof of Theorem~\ref{thm:main_rev} is immediate.
\begin{proof}[Proof of Theorem~\ref{thm:main_rev}]
Let $\tilde \varphi$ be the winding function associated to the foliation $(\partial \tilde D_t =\iota (\partial D_t))_{t \in \m R}$. 
Since $\tilde \varphi \circ \iota = \varphi$ and the Dirichlet energy is invariant under conformal mappings, we obtain $\mc D_{\m C} (\tilde \varphi) = \mc D_{\m C} ( \varphi)$ which implies $S (\rho) = S(\tilde \rho)$ by Theorem~\ref{thm:main0}.
\end{proof}

\begin{proof}[Proof of Lemma~\ref{lem:matching_trace}]
If $\g$ is smooth and $\psi$ takes values in $(-2\pi, 2\pi)$, then by the assumption that
$\varphi(w) \to 0$ as $w\to 0$,
 the harmonicity of $\varphi$ and the maximum principle, we have $\varphi|_\g = \psi|_\g$.

For the general case, we consider the foliation in $D^*$ formed by the family of equipotentials $(\g_r := h (r S^1))_{r \ge 1}$ where $h = k^{-1}$. Let $D_r$ and $D_r^*$ be the bounded and unbounded connected components of $\Chat \smallsetminus \g_r$ respectively.
Let $g_r : D_r \to \m D$ and $k_r : D_r^* \to \m D^*$ be the uniformizing conformal maps associated to $\g_r$.
We define $\varphi_r$ and  $\psi_r$ along $\g_r$
in a similar manner as in the statement of Lemma~\ref{lem:matching_trace}. From the construction, we have 
$$\psi_r = \psi|_{\g_r}, \quad \forall r \ge 1.$$ 
We extend $\varphi$ to $D^*$ by setting $\varphi|_{\g_r} = \varphi_r|_{\g_r}$ for all $r > 1$. 
Since $f_r = g_r^{-1}$ extends to a smooth function on $\ad {\m D}$, one can show that the function associated to the equipotentials 
$$(1, \infty) \to C^{\infty} (S^1, \m C) \colon r \mapsto f_r|_{S^1}$$
 is continuous for the $C^\infty$ norm, we obtain that $\varphi$ is continuous in $D^*$. As $|h' (z) z /h(z) -1 | < 1/2$ in a neighborhood of $\infty$, we see that $\psi$ takes values in $(-\pi/2, \pi/2)$ on that neighborhood. From the previous case, we obtain that $\varphi = \psi$ on a neighborhood of $\infty$.
 Since $\varphi - \psi$ takes values in $2\pi \m Z$, is continuous, and equals $0$ in a neighborhood of $\infty$, we obtain that $\varphi = \psi$ in $D^*$.
 Recall that we also defined $\varphi$ in $D$ as $\arg [g'(w)w /g(w)]$. We only need to show that the limit of $\varphi$ taken from $D^*$ coincides with the limit taken from $D$.

 For this, notice that $\g$ is a rectifiable curve, so is $\iota (\g)$. From  \cite[Thm.\,6.8]{Pommerenke_boundary},  $(\iota \circ h \circ \iota) '$ is in the Hardy space $\mc H^1 (\m D)$. In particular, 
 we have that the length $|\iota (\g_r)|$ converges to $|\iota (\g)|$, and $|\g_r|$ converges to $|\g|$ as $r \to 1 \scriptstyle +$. We have also $f_r  \to f := g^{-1}$ locally uniformly as $r \to 1 \scriptstyle +$ from the Carath\'eodory kernel convergence. A theorem of Warschawski \cite[Thm.\,6.12]{Pommerenke_boundary} shows that
 \begin{equation}\label{eq:Wars}
     \int_{S^1} |f_r'(z) - f '(z)|  \, |\dd z| \to 0 \quad \text{as} \quad r \to 1 \scriptstyle +.
 \end{equation}
 Define the function 
 $$F_r(e^{\ii\t}) : =  \sup_{0 \le \l <1} |f_r'(\l e^{\ii\t}) - f '(\l e^{\ii\t})|.$$
 A theorem of Hardy-Littlewood \cite[Thm.\,1.9]{Duren_Hardy} 
 implies that $F_r$ converges to $0$ in 
 $L^1 (S^1)$ as $r \to 1 \scriptstyle +$ 
 since $f'_r - f'$  converges to $0$ in $\mc H^1 (\m D)$ by \eqref{eq:Wars}.
  This limit and \eqref{eq:Wars} imply that along a subsequence $r_n\to 1 \scriptstyle +$, $f_{r_n}'(z)$ converges to $f'(z)$ and $F_{r_n}(z)$ converges to $0$ on a full measure set $E \subset S^1$, which shows that $f_{r_n}'$ converges to $f'$ uniformly on the radial segment $[0,z]$ for all $z \in E$.
  
  The limit \eqref{eq:Wars} also implies that $f_r \to f$ uniformly on $\ad {\m D}$, see, \cite[Ex.\,6.3.4]{Pommerenke_boundary}. Therefore,
  $$\frac{f_{r_n}' (w) w}{f_{r_n} (w)} \xrightarrow[{\text{unif. on } [0,z]}]{n \to \infty} \frac{f' (w) w}{f (w)}, \quad \forall z \in E.$$
  Since $\frac{f' (w) w}{f (w)}$ does not vanish on $[0,z]$, the uniform convergence implies that the argument also converges uniformly (without the ambiguity in the choice of multiples of $2\pi$). In particular,
  $$\varphi(f_{r_n} (z)) = \varphi_{r_n} (f_{r_n} (z)) =  -\arg \frac{f_{r_n}'(z)z}{f_{r_n}(z)}  \xrightarrow[]{ n \to \infty} -\arg \frac{f'(z)z}{f(z)} = \varphi (f(z)). $$
  We obtain that for $z \in E$, the limit of $\varphi$ at $f(z)$ taken from $D^*$ coincides with the limit taken from $D$ which concludes the proof.
 \end{proof}

\begin{remark}
    It is not known whether whole-plane SLE$_\k$ for $\k > 8$ is reversible. 
(For $\k \le 8$, reversibility was established in \cite{zhan_rev_whole,IG4}.) Therefore Theorem~\ref{thm:main_rev} cannot be predicted from the SLE point of view by considering the $\k \to \infty$ large deviations as we did for the reversibility of chordal Loewner energy in Theorem~\ref{thm:intro_rev}. This result on the other hand suggests that reversibility for whole-plane radial SLE might hold for large $\kappa$ as well.
\end{remark}

From Proposition~\ref{prop:WP-leaf}, being a Weil-Petersson quasicircle (separating $0$ from $\infty$) is a necessary condition to be a leaf in the foliation generated by a measure with finite Loewner-Kufarev energy.
The next result shows that this is also a sufficient condition and we can relate the Loewner energy of a leaf with the Loewner-Kufarev energy of the generating measure. 
In particular, we obtain a new and quantitative characterization of Weil-Petersson quasicircles.

Let $\g$ be a Jordan curve separating $0$ from $
  \infty$, $f$ (resp. $h$) a conformal map from $\m D$ (resp. $\m D^*$) to the bounded (resp. unbounded) component of $\m C \smallsetminus \g$ fixing $0$ (resp. fixing $\infty$).
\begin{thm}[Characterization  {\cite[Thm.\,1.4]{VW2}}]
\label{thm:main-jordan-curve}  
The curve $\g$ is a Weil-Petersson quasicircle if and only if $\g$ can be realized as a leaf in the foliation generated by a measure $\rho$ with $S(\rho) < \infty$. Moreover,  we have
  \[I^L(\g)  = 16 \inf_{\rho} S(\rho) + 2 \log |f'(0)/h'(\infty)|,\]
  where the infimum, which is attained, is taken over all $\rho \in \mc N$ such that $\g$ is a leaf of the generated foliation. 
\end{thm} 

\begin{remark}\label{rem:harmonic_foliation}
The infimum is only realized for the measure $\rho^\g$ generating the family of equipotentials on both sides of $\g$, i.e., the image of the circles centered at $0$ under $f$ and $h$. In this case, the winding function $\varphi^\g$ is harmonic in $\m C \smallsetminus \g$ and tends to $0$ as $z$ tends to $0$ or $\infty$. 
More precisely, we have
\begin{equation}\label{eq:harmonic_winding}
    \varphi^\g \circ f (z)= - \arg \frac{f'(z)z}{f(z)}, \quad \varphi^\g \circ h(z) = - \arg \frac{h'(z)z}{h(z)}.
\end{equation}
Theorem~\ref{thm:main-jordan-curve} and Theorem~\ref{thm:main0} then give the identity
\begin{equation}\label{eq:energy_varphi_gamma}
    I^L(\g) = \mc D_{\m C}(\varphi^\g) +   2 \log |f'(0)/h'(\infty)|.
\end{equation}

 Note also that this minimum is zero if and only if $\g$ is a circle centered at $0$, whereas $I^L(\g)$ is zero for all circles. This explains the presence of the derivative terms. 
\end{remark}
 \begin{cor}[Energy bound {\cite[Cor.\,8.7]{VW2}}] Any leaf $\g$ of the foliation generated by $\rho$  satisfies $I^L(\g) \le 16\, S(\rho)$.
 \end{cor}
 \begin{proof}
A consequence of generalized Grunsky inequality, see, e.g., \cite[p.\,70-71]{TT06}, shows that 
$\log |f'(0)/h'(\infty)| \le 0$.  Theorem~\ref{thm:main-jordan-curve} then implies
$I^L(\g) \le 16\, S(\rho^\g) \le 16 \, S(\rho).$
 \end{proof}

Another consequence of Theorem~\ref{thm:main-jordan-curve} is the following complex identity for bounded curve analogous to Corollary~\ref{cor:complex_field}, which simultaneously expresses the interplay between Dirichlet energies under ``welding'' and ``flow-line'' operations, see Theorem~\ref{thm:welding_coupling1} and \ref{thm:flow-line}.

More precisely, we say that a Weil-Petersson quasicircle $\g$ is \emph{compatible} with $\varphi \in W^{1,2}_{\textrm{loc}}$, if the winding function along $\g$, as defined in \eqref{eq:winding}, coincides with the trace $\varphi|_\g$ arclength-a.e.

\begin{thm}[Complex identity {\cite[Prop.\,1.5]{VW2}}] \label{thm:complex_id_bounded}
Let $\psi$ be a complex valued function on $\m C$ with 
$\mc D_{\m C}(\psi) = \mc D_{\m C}(\Re \psi) +  \mc D_{\m C}(\Im \psi)<\infty$ and $\g$ a Weil-Petersson quasicircle separating $0$ from $\infty$ compatible with $\Im \psi$. 
Let 
\begin{equation}\label{eq:complex_transform}
    \zeta (z): = \psi \circ f (z) + \log \frac{f'(z) z}{ f(z)} \quad  \text{and} \quad \xi(z) : = \psi \circ h(z) + \log \frac{h'(z) z}{ h(z)}.
\end{equation}
Then we have
$\mc D_{\m C} (\psi) = \mc D_{\m D} (\zeta) + \mc D_{\m D^*} (\xi)$. 
\end{thm}
\begin{remark}
We do not use the complex conjugate in defining the transformation law \eqref{eq:complex_transform} as opposed to  Corollary~\ref{cor:complex_field}. The reason is that the convention for compatibility and for being a flow-line in both theorems differ by a sign.
\end{remark}
Similar to Remark~\ref{rem:complex_identity_two_theorems}, Theorem~\ref{thm:complex_id_bounded} can be viewed as the combination of the following two corollaries. Taking $\Re \psi = 0$ and $\varphi = \Im \psi$, we obtain the following result.

\begin{cor}[Flow-line identity for bounded curve]
 Let $\varphi$ be a real-valued function with finite Dirichlet energy. If there exists a chord-arc curve $\g$ which is compatible with $\varphi$, then 
 $$I^L(\g) = \mc D_{\m C}(\varphi) - \mc D_{\m C}(\varphi_0) +   2 \log |f'(0)/h'(\infty)|,$$
 where $\varphi_0$ is the zero-trace part of $\varphi$ in the complement of $\g$. In particular, $\g$ is a Weil-Petersson quasicircle.
\end{cor}

\begin{remark}
As opposed to Theorem~\ref{thm:flow-line}, not every $\varphi$ admits a compatible Jordan curve. For instance, when $\varphi \equiv \pi /4$, the flow-line of the vector field $e^{\ii (\arg z  +\pi /2 - \varphi (z))} $ starting from any point in $\m C$ is a spiral which converges to $\infty$. See Figure~\ref{fig:intro_fol} and Remark~\ref{rem:geom}. In particular, not every real-valued function $\varphi$ with finite Dirichlet energy is the welding function of a foliation. It would be interesting to characterize the class of winding functions analytically.
\end{remark}

\begin{proof}
Using the notation in Theorem~\ref{thm:complex_id_bounded} and applying it to $\psi : = \ii\varphi$, we obtain
$$\zeta(z) = \ii \varphi_0 \circ f + \log \abs{ \frac{f'(z)z}{f(z)}}$$
since for $z \in S^1$, $\varphi \circ f (z) = - \arg (f'(z)z/f(z))$ by the assumption of $\varphi$ being compatible with $\g$ and \eqref{eq:winding}.  Similarly,
$$\xi(z) = \ii \varphi_0 \circ h + \log \abs{ \frac{h'(z)z}{h(z)}}.$$
Theorem~\ref{thm:complex_id_bounded} gives
\begin{align*}
\mc D_{\m C}(\varphi) & = \mc D_{\m C}(\varphi_0) + \mc D_{\m D}\left( \log \abs{ \frac{f'(z)z}{f(z)}} \right) + \mc D_{\m D^*}\left( \log \abs{ \frac{h'(z)z}{h(z)}}\right) \\
& =\mc D_{\m C}(\varphi_0) + \mc D_{\m D}\left( \arg \frac{f'(z)z}{f(z)}\right) + \mc D_{\m D^*}\left( \arg  \frac{h'(z)z}{h(z)}\right)\\
& = \mc D_{\m C}(\varphi_0) + \mc D_{\m C}(\varphi^\g)\\
& = \mc D_{\m C}(\varphi_0) + I^L(\g) -  2 \log |f'(0)/h'(\infty)|,
\end{align*}
where the last two equalities follow from Remark~\ref{rem:harmonic_foliation}. This completes the proof.
\end{proof}

\begin{cor}[Cutting identity for bounded curve]
 Let $\phi$ be a real-valued function with finite Dirichlet energy and $\g$ be a Weil-Petersson curve separating $0$ from $\infty$. Then we have the identity:
    \begin{equation}
    \mc D_\mathbb{C}(\phi) + I^L(\gamma) -  2 \log |f'(0)/h'(\infty)| = \mc D_{\mathbb{D}}(u) + \mc D_{\mathbb{D}^*}(v) ,\end{equation}
    where  
    \begin{equation}\label{eq:ppS_bounded}
     u (z)=  \phi \circ f (z)+ \log \abs{\frac{f'(z)z}{f(z)}}, \quad v (z)=  \phi \circ h (z)+ \log \abs{\frac{h'(z)z}{h(z)}}.
\end{equation}
We may view \eqref{eq:ppS_bounded} as the transformation law such that $e^{2u}  \dd A_{cyl}$ and $e^{2v}  \dd A_{cyl}$ are the pullback measures by $f$ and $h$ of $e^{2\phi} \dd A_{cyl}$, where $\dd A_{cyl} (z) = |z|^{-2} |\dd z|^2$. 
\end{cor}

\begin{proof}
Let $\varphi^\g \in \mc E (\m C)$ be the winding function that is harmonic in the complement of $\g$ as in \eqref{eq:harmonic_winding}.  
Applying Theorem~\ref{thm:complex_id_bounded}  to $\psi : = \phi + \ii \varphi^\g$, we obtain 
$\zeta = u, \quad \xi = v$ which implies
$\mc D_{\m C}(\phi) +  \mc D_{\m C}(\varphi^\g) = \mc D_{\mathbb{D}}(u) + \mc D_{\mathbb{D}^*}(v).$
We conclude with \eqref{eq:energy_varphi_gamma}.
\end{proof}

\section{Summary}
Let us end with a table summarizing the results presented in this survey, highlighting the close analogy between concepts and theorems from random conformal geometry and the finite energy/large deviation world. Although some results involving the finite energy objects are interesting on their own from the analysis perspective,  we choose to omit from the table those without an obvious stochastic counterpart such as results in Section~\ref{sec:application}. 

We hope that the readers are by now convinced that the ideas around large deviations of SLE are great sources for generating exciting results in the deterministic world. And vice versa, as finite energy objects are more regular and easier to handle, exploring their properties also provides a way to generate new conjectures in random conformal geometry. Rather than an end, we hope it to be the starting point of new development along those lines and this survey can serve as a first guide.

\renewcommand{\arraystretch}{1.06}

\begin{center}
\begin{longtable}{ l  l  }
 \toprule		
    \multicolumn{1}{c}{\bf SLE/GFF with $\k \ll 1$}  &   \multicolumn{1}{c}{\bf Finite energy} \\ \midrule 
  \multicolumn{2}{c}{\bf Sections~\ref{sec:intro}--\ref{sec:LDP_single}}\\
  \midrule
      Chordal $\SLE_\kappa$ in $(\domain; \bpt, \ept)$ & A chord $\g$ with $I_{\domain; \bpt, \ept} (\g) < \infty$ \\ 
         & (\emph{Definition~\ref{df:chordalLE}, Theorem~\ref{thm:main_single_LDP}})\\
            \hline 
            Chordal $\SLE_0$ in $(\domain; \bpt, \ept)$ & Hyperbolic geodesic in $(\domain; \bpt, \ept)$\\ 
            & (\emph{Remark~\ref{rem:geodesic}})\\
            \hline 
      Chordal $\SLE_\kappa$ is reversible & Chordal Loewner energy is reversible \\
      & (\emph{Theorem~\ref{thm:intro_rev}})
  \\ \hline 
       & Jordan curve $\gamma$ with $I^L(\g) <\infty$ \\
     $\SLE_\kappa$ loop  & i.e.,  a Weil-Petersson quasicircle $\g$\\
      & (\emph{Theorem~\ref{thm:intro_equiv_energy_WP}})\\
      \midrule
       \multicolumn{2}{c}{\bf Section~\ref{sec:interplay}}\\ \midrule
        Free boundary GFF $\sqrt \k \Phi$ on $\mathbb{H}$ (on $\m C$) &  $2u \in \mathcal{E}(\m H)$, i.e., $\mc D_{\m H}(u) < \infty$  ($2\varphi \in \mathcal{E}(\m C)$)   \\ \hline
              $\sqrt \k$-LQG on quantum plane $\approx e^{\sqrt \k \Phi} \dd A$ & $e^{2 \varphi} \,\dd A, \, \varphi \in \mathcal{E}(\m C)$ \\ \hline

  $\sqrt \k$-LQG on quantum half-plane on $\mathbb{H}$ & $e^{2 u} \,\dd A,  \, u \in \mathcal{E}(\m H)$
  \\ \hline
    $\sqrt \k$-LQG boundary measure on $\mathbb{R}$  & $e^{u}\,\dd x, \, u \in H^{1/2}(\mathbb{R})$ \\
     $\approx e^{\sqrt \k \Phi/2} \dd x$ & \\ \hline
       $\SLE_\kappa$ cuts an independent quantum & A Weil-Petersson quasicircle $\gamma$ cuts \\
      plane $e^{\sqrt \k \Phi} \dd A $ into  independent & $\varphi \in \mathcal{E}(\mathbb{C})$ into $u \in \mc E(\m H), v \in \mc E(\m H^*)$ and \\
     quantum half-planes $e^{\sqrt \k \Phi_1}, e^{\sqrt \k \Phi_2}$ &  $I^L(\gamma) + \mc{D}_{\m{C}}(\varphi) = \mc{D}_{\m{H}}(u) + \mc{D}_{\m{H}^*}(v)$\\
     & (\emph{Theorem~\ref{thm:welding_coupling1}})
     \\
 \hline
    Isometric welding  of  independent &  Isometric welding of $e^{u} \,\dd x$ and $e^{v} \,\dd x$, \\ 
 $\sqrt \k$-LQG  boundary measures on $\mathbb{R}$ &  $ u, v\in H^{1/2}(\m R)$ produces a Weil- \\ 
       produces $\SLE_\kappa$  & Petersson quasicircle (\emph{Theorem~\ref{thm:tuple12}})\\
 \hline
     Bi-infinite flow-line of $e^{\ii \Phi / \chi }\approx e^{\ii \sqrt \k \Phi / 2 }$ &  Bi-infinite flow-line of $e^{\ii \varphi}$ is   a Weil-\\ 
       is an $\SLE_\kappa$ loop & Petersson quasicircle (\emph{Theorem~\ref{thm:flow-line}}) \\
  \midrule 
       \multicolumn{2}{c}{\bf Section~\ref{sec:multi}}\\ \midrule
        Multichordal SLE$_\k$ in $\domain$ of link &  Multichord $\ad \g$ with $I^\a_\domain (\ad \g) <\infty$ \\
         pattern $\a$  & (\emph{Definition~\ref{df:multi_energy}, Theorem~\ref{thm:main_multi_LDP}})\\
  \hline
         Multichordal SLE$_0$ in $\m H$   & Real locus of a real rational function, \\
          & i.e., geodesic multichord  (\emph{Theorem~\ref{thm:real_rational}}) \\
  \hline
  Multichordal SLE$_\k$ pure partition & Minimal potential $\mc M^\a_\domain$ \\
  function  $\mc Z_\a$ of link pattern $\a$ & (\emph{Definition~\ref{df:potential}, Equation~(\ref{eq:intro_limit_Z}}))\\
  \hline
  Loewner evolution of a chord in   & Loewner evolution of a chord in a \\   
multichordal SLE$_\k$ &  geodesic multichord (\emph{Theorem~\ref{thm:marginal}}) \\
  \hline
  Level two null-state BPZ equation  & Semiclassical null-state equation  \\
  satisfied by  $\mc Z_\a$ &  satisfied by $\mc M^\a_{\m H}$ (\emph{Theorem~\ref{thm:null-state}})\\
  \midrule
       \multicolumn{2}{c}{\bf Sections~\ref{sec:radial_infty}--\ref{sec:foliation}}\\  \midrule
 & Loewner chain driven by $\rho \in \mc N$ with \\
    Whole-plane radial SLE$_{16/\k}$   &  Loewner-Kufarev energy $S(\rho) < \infty$; \\
    & Foliation by Weil-Petersson quasicircles\\
     & (\emph{Corollary~\ref{cor:ldp-sle}, Definition~\ref{df:rho_infty}}) \\
  \hline   
   Radial mating of trees (\emph{unknown}) & Energy duality $\mc D_{\m C} (\varphi) = 16 \, S(\rho)$\\ &(\emph{Theorem~\ref{thm:main0}}) \\
  \hline   
    & $\partial D_t$ is a Weil-Petersson quasicircle\\ The outer boundary of SLE$_{16/\k}$   &  for all $t \in \m R$ and \\ 
    is locally a SLE$_{\k}$ curve &  $I^L(\g)  = 16 \inf_{\rho} S(\rho) + 2 \log |f'(0)/h'(\infty)|$\\
   & (\emph{Theorem~\ref{thm:main-jordan-curve}}) \\
  \hline  
   Reversibility of whole-plane radial  & Loewner-Kufarev energy is reversible \\
 SLE$_{16/\k}$ (\emph{unknown})   & (\emph{Theorem~\ref{thm:main_rev}})\\
   \hline
  \bottomrule 
\end{longtable}
\end{center}

\bibliographystyle{alpha}
\bibliography{ref}

\newcommand{\etalchar}[1]{$^{#1}$}
\begin{thebibliography}{GGGPG{\etalchar{+}}13}

\bibitem[Ahl10]{Ahlfors_conformal_invariants}
Lars~V. Ahlfors.
\newblock {\em Conformal invariants}.
\newblock AMS Chelsea Publishing, Providence, RI, 2010.
\newblock Topics in geometric function theory, Reprint of the 1973 original,
  With a foreword by Peter Duren, F. W. Gehring and Brad Osgood.

\bibitem[AKM20]{alberts2020pole}
Tom Alberts, Nam-Gyu Kang, and Nikolai Makarov.
\newblock Pole dynamics and an integral of motion for multiple {SLE}$(0)$.
\newblock {\em arXiv:2011.05714}, 2020.

\bibitem[APW20]{APW}
Morris Ang, Minjae Park, and Yilin Wang.
\newblock Large deviations of radial {${\rm SLE}_\infty$}.
\newblock {\em Electron. J. Probab.}, 25:Paper No. 102, 13, 2020.

\bibitem[BB03]{BB:CFTSLE}
Michel Bauer and Denis Bernard.
\newblock Conformal field theories of stochastic {L}oewner evolutions.
\newblock {\em Comm. Math. Phys.}, 239(3):493--521, 2003.

\bibitem[BBK05]{BBK}
Michel Bauer, Denis Bernard, and Kalle Kyt\"{o}l\"{a}.
\newblock Multiple {S}chramm-{L}oewner evolutions and statistical mechanics
  martingales.
\newblock {\em J. Stat. Phys.}, 120(5-6):1125--1163, 2005.

\bibitem[BBM15]{BBM_15}
Jean Bourgain, Haim Brezis, and Petru Mironescu.
\newblock A new function space and applications.
\newblock {\em J. Eur. Math. Soc. (JEMS)}, 17(9):2083--2101, 2015.

\bibitem[BD16]{Benoist_loop}
St\'{e}phane Benoist and Julien Dub\'{e}dat.
\newblock An {${\rm SLE}_2$} loop measure.
\newblock {\em Ann. Inst. Henri Poincar\'{e} Probab. Stat.}, 52(3):1406--1436,
  2016.

\bibitem[Bec72]{becker1972}
Jochen Becker.
\newblock L\"{o}wnersche {D}ifferentialgleichung und quasikonform fortsetzbare
  schlichte {F}unktionen.
\newblock {\em J. Reine Angew. Math.}, 255:23--43, 1972.

\bibitem[Bef08]{Beffara_dimension}
Vincent Beffara.
\newblock The dimension of the {SLE} curves.
\newblock {\em Ann. Probab.}, 36(4):1421--1452, 2008.

\bibitem[Bil95]{billingsley}
Patrick Billingsley.
\newblock {\em Probability and measure}.
\newblock Wiley Series in Probability and Mathematical Statistics. John Wiley
  \& Sons, Inc., New York, third edition, 1995.
\newblock A Wiley-Interscience Publication.

\bibitem[Bis90]{bishop_isometric}
Christopher~J. Bishop.
\newblock Conformal welding of rectifiable curves.
\newblock {\em Math. Scand.}, 67(1):61--72, 1990.

\bibitem[Bis19]{Bishop_WP}
Christopher~J Bishop.
\newblock Weil-{P}etersson curves, $\beta$-numbers, and minimal surfaces.
\newblock {\em preprint}, 2019.

\bibitem[BPW21]{BPW}
Vincent Beffara, Eveliina Peltola, and Hao Wu.
\newblock On the uniqueness of global multiple {SLE}s.
\newblock {\em Ann. Probab.}, 49(1):400--434, 2021.

\bibitem[BR87]{BowickRajeev1987string}
M.~J. Bowick and S.~G. Rajeev.
\newblock String theory as the {K}\"{a}hler geometry of loop space.
\newblock {\em Phys. Rev. Lett.}, 58(6):535--538, 1987.

\bibitem[Car03]{Car03}
John~L. Cardy.
\newblock Stochastic {L}oewner evolution and {D}yson's circular ensembles.
\newblock {\em J. Phys. A}, 36(24):L379--L386, 2003.

\bibitem[Cui00]{cui00}
Guizhen Cui.
\newblock Integrably asymptotic affine homeomorphisms of the circle and
  {T}eichm\"{u}ller spaces.
\newblock {\em Sci. China Ser. A}, 43(3):267--279, 2000.

\bibitem[Dav82]{david_chord_arc}
Guy David.
\newblock Courbes corde-arc et espaces de {H}ardy g\'{e}n\'{e}ralis\'{e}s.
\newblock {\em Ann. Inst. Fourier (Grenoble)}, 32(3):xi, 227--239, 1982.

\bibitem[Din93]{Dinwoodie}
I.~H. Dinwoodie.
\newblock Identifying a large deviation rate function.
\newblock {\em Ann. Probab.}, 21(1):216--231, 1993.

\bibitem[DMS14]{MoT}
Bertrand Duplantier, Jason Miller, and Scott Sheffield.
\newblock Liouville quantum gravity as a mating of trees.
\newblock {\em arXiv preprint: 1409.7055}, 2014.

\bibitem[DS89]{DeuschelStroock}
Jean-Dominique Deuschel and Daniel~W. Stroock.
\newblock {\em Large deviations}, volume 137 of {\em Pure and Applied
  Mathematics}.
\newblock Academic Press, Inc., Boston, MA, 1989.

\bibitem[DS11]{duplantier_PRL}
Bertrand Duplantier and Scott Sheffield.
\newblock Schramm-{L}oewner evolution and {L}iouville quantum gravity.
\newblock {\em Physical review letters}, 107(13):131305, 2011.

\bibitem[Dub07]{Dub_comm}
Julien Dub\'{e}dat.
\newblock Commutation relations for {S}chramm-{L}oewner evolutions.
\newblock {\em Comm. Pure Appl. Math.}, 60(12):1792--1847, 2007.

\bibitem[Dub09a]{Dub_duality}
Julien Dub\'{e}dat.
\newblock Duality of {S}chramm-{L}oewner evolutions.
\newblock {\em Ann. Sci. \'{E}c. Norm. Sup\'{e}r. (4)}, 42(5):697--724, 2009.

\bibitem[Dub09b]{Dub_couplings}
Julien Dub\'{e}dat.
\newblock S{LE} and the free field: partition functions and couplings.
\newblock {\em J. Amer. Math. Soc.}, 22(4):995--1054, 2009.

\bibitem[Dub15]{Dub_SLEVir1}
Julien Dub\'{e}dat.
\newblock S{LE} and {V}irasoro representations: localization.
\newblock {\em Comm. Math. Phys.}, 336(2):695--760, 2015.

\bibitem[Dur70]{Duren_Hardy}
Peter~L. Duren.
\newblock {\em Theory of {$H\sp{p}$} spaces}.
\newblock Pure and Applied Mathematics, Vol. 38. Academic Press, New
  York-London, 1970.

\bibitem[DV75]{DonVar1975}
M.~D. Donsker and S.~R.~S. Varadhan.
\newblock Asymptotic evaluation of certain {M}arkov process expectations for
  large time. {I}. {II}.
\newblock {\em Comm. Pure Appl. Math.}, 28:1--47; ibid. 28 (1975), 279--301,
  1975.

\bibitem[DZ10]{DZ10}
Amir Dembo and Ofer Zeitouni.
\newblock {\em Large deviations techniques and applications}, volume~38 of {\em
  Stochastic Modelling and Applied Probability}.
\newblock Springer-Verlag, Berlin, 2010.
\newblock Corrected reprint of the second (1998) edition.

\bibitem[EG02]{EG02}
Alexandre Eremenko and Andrei Gabrielov.
\newblock Rational functions with real critical points and the {B}. and {M}.
  {S}hapiro conjecture in real enumerative geometry.
\newblock {\em Ann. of Math. (2)}, 155(1):105--129, 2002.

\bibitem[EG11]{EG11}
Alexandre Eremenko and Andrei Gabrielov.
\newblock An elementary proof of the {B}. and {M}. {S}hapiro conjecture for
  rational functions.
\newblock In {\em Notions of positivity and the geometry of polynomials},
  Trends Math., pages 167--178. Birkh\"{a}user/Springer Basel AG, Basel, 2011.

\bibitem[Fig10]{Figalli_circle}
Alessio Figalli.
\newblock On flows of {$H^{3/2}$}-vector fields on the circle.
\newblock {\em Math. Ann.}, 347(1):43--57, 2010.

\bibitem[FK04]{FK_CFT}
Roland Friedrich and Jussi Kalkkinen.
\newblock On conformal field theory and stochastic {L}oewner evolution.
\newblock {\em Nuclear Phys. B}, 687(3):279--302, 2004.

\bibitem[FW03]{Friedrich_Werner_03}
Roland Friedrich and Wendelin Werner.
\newblock Conformal restriction, highest-weight representations and {SLE}.
\newblock {\em Comm. Math. Phys.}, 243(1):105--122, 2003.

\bibitem[Gar07]{garnett}
John~B. Garnett.
\newblock {\em Bounded analytic functions}, volume 236 of {\em Graduate Texts
  in Mathematics}.
\newblock Springer, New York, first edition, 2007.

\bibitem[GGGPG{\etalchar{+}}13]{GGPPR}
Eva~A. Gallardo-Guti\'{e}rrez, Maria~J. Gonz\'{a}lez, Fernando
  P\'{e}rez-Gonz\'{a}lez, Christian Pommerenke, and Jouni R\"{a}tty\"{a}.
\newblock Locally univalent functions, {VMOA} and the {D}irichlet space.
\newblock {\em Proc. Lond. Math. Soc. (3)}, 106(3):565--588, 2013.

\bibitem[GHS19]{GHS_survey_MoT}
Ewain Gwynne, Nina Holden, and Xin Sun.
\newblock Mating of trees for random planar maps and {L}iouville quantum
  gravity: a survey, 2019.

\bibitem[Gir01]{Gir_BMOA}
Daniel Girela.
\newblock Analytic functions of bounded mean oscillation.
\newblock In {\em Complex function spaces ({M}ekrij\"{a}rvi, 1999)}, volume~4
  of {\em Univ. Joensuu Dept. Math. Rep. Ser.}, pages 61--170. Univ. Joensuu,
  Joensuu, 2001.

\bibitem[GM05]{GM}
John~B. Garnett and Donald~E. Marshall.
\newblock {\em Harmonic measure}, volume~2 of {\em New Mathematical
  Monographs}.
\newblock Cambridge University Press, Cambridge, 2005.

\bibitem[Gol91]{Gol91}
Lisa~R. Goldberg.
\newblock Catalan numbers and branched coverings by the {R}iemann sphere.
\newblock {\em Adv. Math.}, 85(2):129--144, 1991.

\bibitem[GP18]{GP}
Pavel Gumenyuk and Istv\'{a}n Prause.
\newblock Quasiconformal extensions, {L}oewner chains, and the
  {$\lambda$}-lemma.
\newblock {\em Anal. Math. Phys.}, 8(4):621--635, 2018.

\bibitem[JN61]{JohnNirenberg}
F.~John and L.~Nirenberg.
\newblock On functions of bounded mean oscillation.
\newblock {\em Comm. Pure Appl. Math.}, 14:415--426, 1961.

\bibitem[Joh21]{johansson2021strong}
Kurt Johansson.
\newblock Strong {S}zeg\"{o} theorem on a {J}ordan curve, 2021.

\bibitem[JVST12]{JohSolTur2012}
Fredrik Johansson~Viklund, Alan Sola, and Amanda Turner.
\newblock Scaling limits of anisotropic {H}astings-{L}evitov clusters.
\newblock {\em Ann. Inst. Henri Poincar\'{e} Probab. Stat.}, 48(1):235--257,
  2012.

\bibitem[JW84]{JW1984}
Alf Jonsson and Hans Wallin.
\newblock Function spaces on subsets of {${\bf R}^n$}.
\newblock {\em Math. Rep.}, 2(1):xiv+221, 1984.

\bibitem[KL07]{KL07}
Michael~J. Kozdron and Gregory~F. Lawler.
\newblock The configurational measure on mutually avoiding {SLE} paths.
\newblock In {\em Universality and renormalization}, volume~50 of {\em Fields
  Inst. Commun.}, pages 199--224. Amer. Math. Soc., Providence, RI, 2007.

\bibitem[KNK04]{KNK_exact}
Wouter Kager, Bernard Nienhuis, and Leo~P. Kadanoff.
\newblock Exact solutions for {L}oewner evolutions.
\newblock {\em J. Statist. Phys.}, 115(3-4):805--822, 2004.

\bibitem[Kuf43]{Kufarev}
P.~P. Kufareff.
\newblock On one-parameter families of analytic functions.
\newblock {\em Rec. Math. [Mat. Sbornik] N.S.}, 13(55).:87--118, 1943.

\bibitem[Law05]{Law05}
Gregory~F. Lawler.
\newblock {\em Conformally invariant processes in the plane}, volume 114 of
  {\em Mathematical Surveys and Monographs}.
\newblock American Mathematical Society, Providence, RI, 2005.

\bibitem[Law09]{Law09}
Gregory~F. Lawler.
\newblock Partition functions, loop measure, and versions of {SLE}.
\newblock {\em J. Stat. Phys.}, 134(5-6):813--837, 2009.

\bibitem[Leh87]{lehto2012univalent}
Olli Lehto.
\newblock {\em Univalent functions and {T}eichm\"{u}ller spaces}, volume 109 of
  {\em Graduate Texts in Mathematics}.
\newblock Springer-Verlag, New York, 1987.

\bibitem[Lin05]{Lind_sharp}
Joan~R. Lind.
\newblock A sharp condition for the {L}oewner equation to generate slits.
\newblock {\em Ann. Acad. Sci. Fenn. Math.}, 30(1):143--158, 2005.

\bibitem[LLN09]{LLN_capacity}
Steven Lalley, Gregory Lawler, and Hariharan Narayanan.
\newblock Geometric interpretation of half-plane capacity.
\newblock {\em Electron. Commun. Probab.}, 14:566--571, 2009.

\bibitem[Loe23]{Loewner23}
Karl Loewner.
\newblock Untersuchungen \"{u}ber schlichte konforme {A}bbildungen des
  {E}inheitskreises. {I}.
\newblock {\em Math. Ann.}, 89(1-2):103--121, 1923.

\bibitem[LSW03]{LSW_CR_chordal}
Gregory Lawler, Oded Schramm, and Wendelin Werner.
\newblock Conformal restriction: the chordal case.
\newblock {\em J. Amer. Math. Soc.}, 16(4):917--955, 2003.

\bibitem[LSW04]{LSW04LERWUST}
Gregory~F. Lawler, Oded Schramm, and Wendelin Werner.
\newblock Conformal invariance of planar loop-erased random walks and uniform
  spanning trees.
\newblock {\em Ann. Probab.}, 32(1B):939--995, 2004.

\bibitem[LW04]{LW2004loupsoup}
Gregory~F. Lawler and Wendelin Werner.
\newblock The {B}rownian loop soup.
\newblock {\em Probab. Theory Related Fields}, 128(4):565--588, 2004.

\bibitem[Mil97]{FubiniFoiled}
John Milnor.
\newblock Fubini foiled: {K}atok's paradoxical example in measure theory.
\newblock {\em Math. Intelligencer}, 19(2):30--32, 1997.

\bibitem[MR05]{Marshall_Rohde}
Donald~E. Marshall and Steffen Rohde.
\newblock The {L}oewner differential equation and slit mappings.
\newblock {\em J. Amer. Math. Soc.}, 18(4):763--778, 2005.

\bibitem[MS16a]{IG1}
Jason Miller and Scott Sheffield.
\newblock Imaginary geometry {I}: interacting {SLE}s.
\newblock {\em Probab. Theory Related Fields}, 164(3-4):553--705, 2016.

\bibitem[MS16b]{IG2}
Jason Miller and Scott Sheffield.
\newblock Imaginary geometry {II}: reversibility of {${\rm
  SLE}_\kappa(\rho_1;\rho_2)$} for {$\kappa\in(0,4)$}.
\newblock {\em Ann. Probab.}, 44(3):1647--1722, 2016.

\bibitem[MS16c]{IG3}
Jason Miller and Scott Sheffield.
\newblock Imaginary geometry {III}: reversibility of {$\rm SLE_\kappa$} for
  {$\kappa\in(4,8)$}.
\newblock {\em Ann. of Math. (2)}, 184(2):455--486, 2016.

\bibitem[MS16d]{MilShe2016}
Jason Miller and Scott Sheffield.
\newblock Quantum {L}oewner evolution.
\newblock {\em Duke Math. J.}, 165(17):3241--3378, 2016.

\bibitem[MS17]{IG4}
Jason Miller and Scott Sheffield.
\newblock Imaginary geometry {IV}: interior rays, whole-plane reversibility,
  and space-filling trees.
\newblock {\em Probab. Theory Related Fields}, 169(3-4):729--869, 2017.

\bibitem[MW21]{michelat2021loewner}
Alexis Michelat and Yilin Wang.
\newblock The {L}oewner {E}nergy via {M}oving {F}rames and {S}urfaces of
  {F}inite {R}enormalised {A}rea {B}ounding {W}eil-petersson {C}urves, 2021.

\bibitem[NS95a]{Nag_Sullivan}
Subhashis Nag and Dennis Sullivan.
\newblock Teichm\"{u}ller theory and the universal period mapping via quantum
  calculus and the {$H^{1/2}$} space on the circle.
\newblock {\em Osaka J. Math.}, 32(1):1--34, 1995.

\bibitem[NS95b]{NS95}
Subhashis Nag and Dennis Sullivan.
\newblock Teichm\"{u}ller theory and the universal period mapping via quantum
  calculus and the {$H^{1/2}$} space on the circle.
\newblock {\em Osaka J. Math.}, 32(1):1--34, 1995.

\bibitem[NV90]{NagVerjovsky}
Subhashis Nag and Alberto Verjovsky.
\newblock {${\rm Diff}(S^1)$} and the {T}eichm\"{u}ller spaces.
\newblock {\em Comm. Math. Phys.}, 130(1):123--138, 1990.

\bibitem[Pel19]{Peltola}
Eveliina Peltola.
\newblock Toward a conformal field theory for {S}chramm-{L}oewner evolutions.
\newblock {\em J. Math. Phys.}, 60(10):103305, 39, 2019.

\bibitem[Pom65]{Pom1965}
Christian Pommerenke.
\newblock \"{U}ber die {S}ubordination analytischer {F}unktionen.
\newblock {\em J. Reine Angew. Math.}, 218:159--173, 1965.

\bibitem[Pom75]{Pom_uni}
Christian Pommerenke.
\newblock {\em Univalent functions}.
\newblock Vandenhoeck \& Ruprecht, G\"{o}ttingen, 1975.
\newblock With a chapter on quadratic differentials by Gerd Jensen, Studia
  Mathematica/Mathematische Lehrb\"{u}cher, Band XXV.

\bibitem[Pom78]{Pommerenke_VMOA}
Ch. Pommerenke.
\newblock On univalent functions, {B}loch functions and {VMOA}.
\newblock {\em Math. Ann.}, 236(3):199--208, 1978.

\bibitem[Pom92]{Pommerenke_boundary}
Ch. Pommerenke.
\newblock {\em Boundary behaviour of conformal maps}, volume 299 of {\em
  Grundlehren der Mathematischen Wissenschaften [Fundamental Principles of
  Mathematical Sciences]}.
\newblock Springer-Verlag, Berlin, 1992.

\bibitem[PW19]{PW19}
Eveliina Peltola and Hao Wu.
\newblock Global and local multiple {SLE}s for {$\kappa \leq 4$} and connection
  probabilities for level lines of {GFF}.
\newblock {\em Comm. Math. Phys.}, 366(2):469--536, 2019.

\bibitem[PW21]{peltola_wang}
Eveliina Peltola and Yilin Wang.
\newblock Large deviations of multichordal {SLE}$_{0+}$, real rational
  functions, and zeta-regularized determinants of {L}aplacians.
\newblock {\em To appear in J. Eur. Math. Soc.}, 2021.

\bibitem[RR94]{rosenblum}
Marvin Rosenblum and James Rovnyak.
\newblock {\em Topics in {H}ardy classes and univalent functions}.
\newblock Birkh\"{a}user Advanced Texts: Basler Lehrb\"{u}cher. [Birkh\"{a}user
  Advanced Texts: Basel Textbooks]. Birkh\"{a}user Verlag, Basel, 1994.

\bibitem[RS05]{Rohde_Schramm}
Steffen Rohde and Oded Schramm.
\newblock Basic properties of {SLE}.
\newblock {\em Ann. of Math. (2)}, 161(2):883--924, 2005.

\bibitem[RW21]{RW}
Steffen Rohde and Yilin Wang.
\newblock The {L}oewner energy of loops and regularity of driving functions.
\newblock {\em Int. Math. Res. Not. IMRN}, 2021(10):7715--7763, 2021.

\bibitem[Sch00]{Schramm2000}
Oded Schramm.
\newblock Scaling limits of loop-erased random walks and uniform spanning
  trees.
\newblock {\em Israel J. Math.}, 118:221--288, 2000.

\bibitem[Sch07]{Schramm:ICM}
Oded Schramm.
\newblock Conformally invariant scaling limits: an overview and a collection of
  problems.
\newblock In {\em International {C}ongress of {M}athematicians. {V}ol. {I}},
  pages 513--543. Eur. Math. Soc., Z\"{u}rich, 2007.

\bibitem[Sem86]{Semmes86}
Stephen Semmes.
\newblock A counterexample in conformal welding concerning chord-arc curves.
\newblock {\em Ark. Mat.}, 24(1):141--158, 1986.

\bibitem[She16]{Quantum_zipper}
Scott Sheffield.
\newblock Conformal weldings of random surfaces: {SLE} and the quantum gravity
  zipper.
\newblock {\em Ann. Probab.}, 44(5):3474--3545, 2016.

\bibitem[She18]{shen13}
Yuliang Shen.
\newblock Weil-{P}etersson {T}eichm\"{u}ller space.
\newblock {\em Amer. J. Math.}, 140(4):1041--1074, 2018.

\bibitem[SM06]{sharon20062d}
Eitan Sharon and David Mumford.
\newblock 2d-shape analysis using conformal mapping.
\newblock {\em International Journal of Computer Vision}, 70(1):55--75, 2006.

\bibitem[Smi06]{Smi:ICM}
Stanislav Smirnov.
\newblock Towards conformal invariance of 2{D} lattice models.
\newblock In {\em International {C}ongress of {M}athematicians}, volume~2,
  pages 1421--1451. Eur. Math. Soc., Z\"{u}rich, 2006.

\bibitem[Smi10]{Smirnov10}
Stanislav Smirnov.
\newblock Conformal invariance in random cluster models. {I}. {H}olomorphic
  fermions in the {I}sing model.
\newblock {\em Ann. of Math. (2)}, 172(2):1435--1467, 2010.

\bibitem[Sot00]{Sot00}
Frank Sottile.
\newblock Real {S}chubert calculus: polynomial systems and a conjecture of
  {S}hapiro and {S}hapiro.
\newblock {\em Experiment. Math.}, 9(2):161--182, 2000.

\bibitem[SS09]{SS09GFF}
Oded Schramm and Scott Sheffield.
\newblock Contour lines of the two-dimensional discrete {G}aussian free field.
\newblock {\em Acta Math.}, 202(1):21--137, 2009.

\bibitem[ST20]{Shen-Tang}
Yuliang Shen and Shuan Tang.
\newblock Weil-{P}etersson {T}eichm\"{u}ller space {II}: {S}moothness of flow
  curves of {$H^{\frac 32}$}-vector fields.
\newblock {\em Adv. Math.}, 359:106891, 25, 2020.

\bibitem[STW18]{Shen-Tang-Wu}
Yuliang Shen, Shuan Tang, and Li~Wu.
\newblock Weil-{P}etersson and little {T}eichm\"{u}ller spaces on the real
  line.
\newblock {\em Ann. Acad. Sci. Fenn. Math.}, 43(2):935--943, 2018.

\bibitem[STZ99]{STZ_KdV}
Maria~E. Schonbek, Andrey~N. Todorov, and Jorge~P. Zubelli.
\newblock Geodesic flows on diffeomorphisms of the circle, {G}rassmannians, and
  the geometry of the periodic {K}d{V} equation.
\newblock {\em Adv. Theor. Math. Phys.}, 3(4):1027--1092, 1999.

\bibitem[TT06]{TT06}
Leon~A. Takhtajan and Lee-Peng Teo.
\newblock Weil-{P}etersson metric on the universal {T}eichm\"{u}ller space.
\newblock {\em Mem. Amer. Math. Soc.}, 183(861):viii+119, 2006.

\bibitem[Var66]{Var66}
S.~R.~S. Varadhan.
\newblock Asymptotic probabilities and differential equations.
\newblock {\em Comm. Pure Appl. Math.}, 19:261--286, 1966.

\bibitem[VW20a]{VW1}
Fredrik Viklund and Yilin Wang.
\newblock Interplay between {L}oewner and {D}irichlet energies via conformal
  welding and flow-lines.
\newblock {\em Geom. Funct. Anal.}, 30(1):289--321, 2020.

\bibitem[VW20b]{VW2}
Fredrik Viklund and Yilin Wang.
\newblock The {Loewner-Kufarev Energy and Foliations by Weil-Petersson
  Quasicircles}.
\newblock {\em arXiv preprint: 2012.05771}, 2020.

\bibitem[Wan19a]{W1}
Yilin Wang.
\newblock The energy of a deterministic {L}oewner chain: reversibility and
  interpretation via {${\rm SLE}_{0+}$}.
\newblock {\em J. Eur. Math. Soc. (JEMS)}, 21(7):1915--1941, 2019.

\bibitem[Wan19b]{W2}
Yilin Wang.
\newblock Equivalent descriptions of the {L}oewner energy.
\newblock {\em Invent. Math.}, 218(2):573--621, 2019.

\bibitem[Wer04a]{WW_toulouse_Girsanov}
Wendelin Werner.
\newblock Girsanov's transformation for {${\rm SLE}(\kappa,\rho)$} processes,
  intersection exponents and hiding exponents.
\newblock {\em Ann. Fac. Sci. Toulouse Math. (6)}, 13(1):121--147, 2004.

\bibitem[Wer04b]{WW_St_Flour}
Wendelin Werner.
\newblock Random planar curves and {S}chramm-{L}oewner evolutions.
\newblock In {\em Lectures on probability theory and statistics}, volume 1840
  of {\em Lecture Notes in Math.}, pages 107--195. Springer, Berlin, 2004.

\bibitem[Wer08]{Werner_loop}
Wendelin Werner.
\newblock The conformally invariant measure on self-avoiding loops.
\newblock {\em J. Amer. Math. Soc.}, 21(1):137--169, 2008.

\bibitem[Wit88]{witten88}
Edward Witten.
\newblock Coadjoint orbits of the {V}irasoro group.
\newblock {\em Comm. Math. Phys.}, 114(1):1--53, 1988.

\bibitem[Zha08a]{Zhan_duality}
Dapeng Zhan.
\newblock Duality of chordal {SLE}.
\newblock {\em Invent. Math.}, 174(2):309--353, 2008.

\bibitem[Zha08b]{Zhan}
Dapeng Zhan.
\newblock Reversibility of chordal {SLE}.
\newblock {\em Ann. Probab.}, 36(4):1472--1494, 2008.

\bibitem[Zha15]{zhan_rev_whole}
Dapeng Zhan.
\newblock Reversibility of whole-plane {SLE}.
\newblock {\em Probab. Theory Related Fields}, 161(3-4):561--618, 2015.

\bibitem[Zha21]{zhan2020sleloop}
Dapeng Zhan.
\newblock S{LE} loop measures.
\newblock {\em Probab. Theory Related Fields}, 179(1-2):345--406, 2021.

\end{thebibliography}

\end{document}